\documentclass[11pt,notitlepage]{article}
\usepackage{amssymb,amsthm,bbm,graphicx,placeins,enumitem}
\usepackage{amstext}
\usepackage{hyperref}
\hypersetup{colorlinks,
linkcolor=blue,
filecolor=blue,
urlcolor=blue,
citecolor=blue}
\usepackage{natbib}
\usepackage{snapshot}
\usepackage{setspace,mathtools,thmtools,thm-restate}
\usepackage[left=1.25in,right=1.25in,top=1in,bottom=1in]{geometry} 
\usepackage{float}
\usepackage{color,colortbl}
\usepackage{changepage}
\usepackage{booktabs}
\usepackage{array}
\usepackage{multirow}
\usepackage[font=small]{caption}
\usepackage{subcaption}
\usepackage[ruled,norelsize]{algorithm2e}
\usepackage{titlesec}
\titlespacing*{\subsection}{0pt}{5pt}{5pt}
\titlespacing*{\subsubsection}{0pt}{5pt}{5pt}
\setlist[itemize]{topsep=1pt,itemsep=-1ex}
\setlist[enumerate]{topsep=1pt,itemsep=-1ex}

\allowdisplaybreaks

\numberwithin{equation}{section}

\newcommand{\widebar}[1]{\mkern 1mu\overline{\mkern-1mu#1\mkern-1mu}\mkern 1mu}

\newcommand{\<}{\mspace{1mu}}

\DeclareMathOperator*{\argmin}{arg\,min}
\DeclareMathOperator*{\diag}{diag}
 
\newcommand{\indicate}[1]{\mathbf{1}\Bigl\{#1\Bigr\}}

\newtheorem{theorem}{Theorem}
\newtheorem{proposition}{Proposition}
\newtheorem{lemma}{Lemma}
\newtheorem{assumption}{Assumption}
\newtheorem{definition}{Definition}

\DeclareCaptionSubType[Alph]{figure}

\newlength\mylen
\newcommand\myinput[1]{%
  \settowidth\mylen{\KwIn{}}%
  \setlength\hangindent{\mylen}%
  \hspace*{\mylen}#1\\}

\begin{document}
\setlength{\abovedisplayskip}{6pt}
\setlength{\belowdisplayskip}{6pt}

\title{\textbf{\Large Risk-Averse Approximate Dynamic Programming with Quantile-Based Risk Measures}}
\author{Daniel R. Jiang and Warren B. Powell}

\maketitle
\begin{abstract}
In this paper, we consider a finite-horizon Markov decision process (MDP) for which the objective at each stage is to minimize a quantile-based risk measure (QBRM) of the sequence of future costs; we call the overall objective a \emph{dynamic quantile-based risk measure} (DQBRM). In particular, we consider optimizing dynamic risk measures where the one-step risk measures are QBRMs, a class of risk measures that includes the popular \emph{value at risk} (VaR) and the \emph{conditional value at risk} (CVaR). Although there is considerable theoretical development of risk-averse MDPs in the literature, the computational challenges have not been explored as thoroughly. We propose data-driven and simulation-based approximate dynamic programming (ADP) algorithms to solve the risk-averse sequential decision problem. We address the issue of inefficient sampling for risk applications in simulated settings and present a procedure, based on importance sampling, to direct samples toward the ``risky region'' as the ADP algorithm progresses. Finally, we show numerical results of our algorithms in the context of an application involving risk-averse bidding for energy storage.
\end{abstract}


\section{Introduction}
Sequential decision problems, in the form of Markov decision processes (MDPs), are most often formulated with the objective of minimizing an expected sum of costs or maximizing an expected sum of rewards \citep{Puterman,Bertsekas1996,Powell2011}. However, it is becoming more and more evident that solely considering the expectation is insufficient as risk-preferences can vary greatly from application to application. Broadly speaking, the expected value can fail to be useful in settings containing either heavy-tailed distributions or rare, but high-impact events. For example, heavy-tailed distributions arise frequently in finance (electricity prices are well-known to possess this feature; see \cite{Bystrom2005}, \cite{Kim2011c}). In this case, the mean of the distribution itself may not necessarily be a good representation of the randomness of the problem; instead, it is likely useful to introduce a measure of risk on the tail of the distribution as well. The rare event situation is, in a sense, the inverse case of the heavy-tail phenomenon, but it can also benefit from a risk measure other than the expectation. To illustrate, certain problems in operations research can be complicated by critical events that happen with small probability, such as guarding against stock-outs and large back-orders in inventory problems (see \cite{Glasserman1996}) or managing the risk of the failure of a high-value asset (see \cite{Enders2010a}). In these circumstances, the merit of a policy might be measured by the number of times that a \emph{bad event} happens over some time horizon. 

One way to introduce risk-aversion into sequential problems is to formulate the objective using \emph{dynamic risk measures} \citep{Ruszczynski2010}. A rough preview, without formal definitions, of our optimization problem is as follows: we wish to find a policy that minimizes risk, as assessed by a certain type of dynamic risk measure. The objective can be written as
\begin{equation*}
\min_{\pi \in \Pi} \; \rho_0^\alpha \Bigl ( C_1^\pi + \rho_1^\alpha \bigl( C_2^\pi + \cdots + \rho_{T-1}^\alpha( C_T^\pi ) \cdots \bigr) \Bigr),
\end{equation*}
where $\Pi$ is a set of policies, $\{C_t^\pi\}$ are costs under policy $\pi$, and $\{\rho_t^\alpha\}$ are one-step risk measures (i.e., components of the overall dynamic risk measure). Precise definitions are given in the subsequent sections. We focus on the case where the objective at each stage is to optimize a \emph{quantile-based risk measure} (QBRM) of future costs; we call the overall objective a \emph{dynamic quantile-based risk measure} (DQBRM).

This paper makes the following contributions. First, we describe a new data-driven or simulation-based ADP algorithm, called \emph{Dynamic-QBRM ADP}, that is similar in spirit to established asynchronous algorithms like $Q$-learning (see \cite{Watkins1992}) and lookup table approximate value iteration (see, e.g., \cite{Bertsekas1996}, \cite{Powell2011}), where one state is updated per iteration. The second contribution of the paper is a companion sampling procedure to Dynamic-QBRM ADP, which we call \emph{risk-directed sampling} (RDS). As we describe above, when dealing with risk, there is a large class of problems in which we are inherently dealing with rare, but very costly events. Broadly speaking, the evaluation of a QBRM that is focused on the tail of the distribution (e.g., CVaR at, say, the 99\% level) depends crucially on efficiently directing the algorithm toward sampling these ``risky'' regions. In this part of the paper, we consider the question: is there a way to learn, as the ADP algorithm progresses, the interesting values of the information process to sample?

The paper is organized as follows. We first provide a literature review in Section \ref{sec:litreview}. In Section \ref{sec:problem}, we give our problem formulation, a brief introduction to dynamic risk measures, and the definition of a class of quantile-based risk measures. Next, we introduce the algorithm for solving risk-averse MDPs in Section \ref{sec:algorithm} and give a theoretical analysis in Section \ref{sec:algconv}. In Section \ref{sec:sampling}, we discuss sampling issues and describe the companion sampling procedure. We show numerical results on an example energy trading application in Section \ref{sec:numerical} and conclude in Section \ref{sec:conclusion}.

\section{Literature Review}
\label{sec:litreview}
The theory of \emph{dynamic risk measures} and the notion of \emph{time-consistency} (see e.g. \cite{Riedel2004}, \cite{Artzner2006}, \cite{Cheridito2006}) is extended to the setting of sequential optimization problems in \cite{Ruszczynski2006} and \cite{Ruszczynski2010}, in which it is proved that any time-consistent dynamic risk measure can be written as compositions of \emph{one-step conditional risk measures} (these are simply risk measures defined in a conditional setting, analogous to the conditional expectation for the traditional case). From this, a Bellman recursion is obtained, becoming a familiar way of characterizing optimal policies. Building on the theory of dynamic programming, versions of exact value iteration and policy iteration are also developed in \cite{Ruszczynski2010}. Later, in \cite{Cavus2014a}, these exact methods are analyzed in the more specific case of undiscounted transient models.

Under the assumption that we use one-step \emph{coherent risk measures}, as axiomatized in \cite{Artzner1999}, the value functions of a risk-averse Markov decision process with a convex terminal value function can be easily shown to satisfy convexity using the fact that coherent risk measures are convex and monotone. Therefore, the traditional method of stochastic dual dynamic programming (SDDP) of \cite{Pereira1991} for multistage, risk-neutral problems, which relies on the \emph{convexity} of value functions, can be adapted to the risk-averse case. This idea is successfully explored in \cite{Philpott2012}, \cite{Shapiro2013}, and \cite{Philpott2013}, with applications to the large-scale problem of hydro-thermal scheduling using one-step mean-CVaR (convex combination of mean and CVaR) and one-step mean-upper semideviation risk measures. The main drawbacks of risk-averse SDDP are (1) the cost function must be linear in the state, (2) some popular risk measures, such as \emph{value at risk} (VaR), are excluded because they are not coherent and do not imply convex value functions, and (3) the full risk measure (can be recast as an expectation in certain instances) has to be computed at every iteration.
Since no convexity or linearity assumptions are made in this paper, we take an alternative approach from the SDDP methods and instead assume the setting of finite state and action spaces, as in $Q$-learning. At the same time, because the default implementation of our approach does not take advantage of structure, it is limited to smaller problems. Extensions to the methods proposed in this paper for exploiting structure can be made by following techniques such as those discussed in \cite{Powell2004}, \cite{Nascimento2009a}, and \cite{Jiang2013}.


Recursive stochastic approximation methods have been applied to estimating quantiles in static settings (see \cite{Tierney1983}, \cite{Bardou2009}, and \cite{Kan2011}). Related to our work, a policy gradient method for optimizing MDPs (with a risk-neutral objective) under a CVaR constraint is given in \cite{Chow2014}. All of these methods are related to ours in the sense that the minimization formula \citep[Theorem 10]{Rockafellar2002} for CVaR is optimized with gradient techniques. In our multistage setting with dynamic risk measures, which is also coupled with optimal control, there are some new interesting complexities, including the fact that every new observation (or data point) is generated from an \emph{imperfect distribution} of future costs that is ``bootstrapped'' from the previous estimate of the value function. This means that not only are the observations inherently biased, but the errors compound over time -- this was not the case for the static setting considered in earlier work. Under reasonable assumptions, we analyze both the almost sure convergence and convergence rates for our proposed algorithms.

Our risk-directed sampling procedure is inspired by adaptive importance sampling strategies from the literature, such as the celebrated \emph{cross-entropy method} of \cite{Rubinstein1999}. See, e.g., \cite{Al-Qaq1995}, \cite{Bardou2009}, \cite{Egloff2010}, and \cite{Ryu2015} for other similar approaches. The critical difference in our approach is that in an ADP setting, we have the added difficulty of not being able to assume perfect knowledge of the objective function; rather, our observations are noisy and biased. To our knowledge, this is the first time an adaptive sampling procedure has been combined with a value function approximation algorithm in the risk-averse MDP literature. The closest paper is by \cite{Kozmik2014}, which considers an importance sampling approach for policy evaluation.

\section{Problem Formulation}
\label{sec:problem}
In this section, we establish the setting of the paper. In particular, we describe the risk-averse model, introduce the concept of dynamic risk measures, and define a class of quantile-based risk measures. 
\subsection{Model}
We consider an MDP with a finite time-horizon,  $t=0, 1, 2, \ldots, T$, where the last decision is made at time $t=T-1$, so that the set of decision epochs is given by $\mathcal T = \{0,1,2,\ldots, T-1\}.$ Given a probability space $(\Omega, \mathcal F, \mathbf{P})$, we define a discrete-time stochastic process $\{W_t\}_{t=0}^T$, with $W_t \in \mathcal W$ for all $t$, as the exogenous information process in the sequential decision problem, where $W_t$ is adapted to a filtration $\{\mathcal F_t\}_{t=0}^T$, with $\{\varnothing,\Omega\} = \mathcal F_0 \subseteq \mathcal F_1 \subseteq \cdots \subseteq \mathcal F_T \subseteq \mathcal F$. We assume that all sources of randomness in the problem are encapsulated by the process $\{W_t\}$ and that it is independent across time. For computational tractability, we work in the setting of finite state and action spaces. Let the state space be denoted $\mathcal{S}$, and let the action space be $\mathcal A$, where $|\mathcal S| < \infty$ and $|\mathcal A| < \infty$. The set of feasible actions for each state $s \in \mathcal S$, written $\mathcal A_s$, is a subset of $\mathcal A$. The set $\mathcal U = \{(s,a) \in  \mathcal S \times \mathcal A: a \in \mathcal A_s\}$ is the set of all feasible state-action pairs.
The stochastic process describing the states of the system is $\{S_t\}_{t=0}^T$, where $S_t$ is an $\mathcal F_t$-measurable random variable taking values in $\mathcal S$, and $a_t$ is a feasible action determined by the decision maker using $S_t$. Furthermore, let $\mathcal Z_t$ denote the space of $\mathcal F_t$-measurable random variables and $\mathcal Z_{t,T} = \mathcal Z_t \times \cdots \times \mathcal Z_T$.

We model the system using a \emph{transition function} or \emph{system model} $S^M:\mathcal S \times \mathcal A \times \mathcal W \rightarrow \mathcal S$, which produces the next state $S_{t+1}$ given a current state $S_t$, action $a_t$, and an outcome of the exogenous process $W_{t+1}$: $S_{t+1} = S^M(S_t, a_t, W_{t+1})$. The \emph{cost} for time $t$ is given by $c_t(S_t,a_t,W_{t+1})$, where $c_t:\mathcal S \times \mathcal A \times \mathcal W \rightarrow \mathbb R$ is the \emph{cost function}. A policy is a sequence of decision functions $\{A_0^\pi, A_1^\pi, \ldots, A_{T-1}^\pi\}$ indexed by $\pi \in \Pi$, where $\Pi$ is the index set of all policies. Each decision function $A_t^\pi : \mathcal S \rightarrow \mathcal A$ is a mapping from a state to a feasible action, such that $A_t^\pi(s) \in \mathcal A_s$ for any state $s$. Let the sequence of costs under a policy $\pi$ be represented by the process $C_t^\pi$ for $t=1,2,\ldots,T$, where
\[
C_{t}^\pi = c_{t-1}(S_{t-1}^\pi, A_{t-1}^\pi(S_{t-1}^\pi), W_{t}) \in \mathcal Z_t,
\]
where $\{S_t^\pi\}$ are the states visited while following policy $\pi$.
Note that $C_t^\pi$ refers to the cost from time $t-1$, but the index of $t$ refers to its measurability: $C_t^\pi$ depends on information only known at time $t$.

\subsection{Review of Dynamic Risk Measures} In this subsection, we briefly introduce the notion of a dynamic risk measure; for a more detailed treatment, see, e.g., \cite{Frittelli2004}, \cite{Riedel2004}, \cite{Pflug2005}, \cite{Boda2006}, \cite{Cheridito2006}, and \cite{Acciaio2011}. Our presentation closely follows that of \cite{Ruszczynski2010}, which develops the theory of dynamic risk measures in the context of MDPs. First, a \emph{conditional risk measure} is a mapping $\rho_{t,T}: \mathcal Z_{t,T} \rightarrow \mathcal Z_t$ that satisfies the following monotonicity requirement: for $X, Y \in \mathcal Z_{t,T}$ and $X \le Y$ (componentwise and almost surely), $\rho_{t,T}(X) \le \rho_{t,T}(Y)$.

Given a sequence of future costs $C_t, \ldots, C_T$, the intuitive meaning of $\rho_{t,T}(C_t,\ldots,C_T)$ is a \emph{certainty equivalent} cost (i.e., at time $t$, one is indifferent between incurring $\rho_{t,T}(C_t,\ldots,C_T)$ and the alternative of being subjected to the stream of stochastic future costs). See \cite{Rudloff2014} for an in-depth discussion regarding the certainty equivalent interpretation in the context of multistage stochastic models. A \emph{dynamic risk measure} $\{\rho_{t,T}\}^T_{t=0}$ is a sequence of conditional risk measures $\rho_{t,T}:\mathcal Z_{t,T} \rightarrow \mathcal Z_t$, which allows us to evaluate the future risk at any time $t$ using $\rho_{t,T}$. Of paramount importance to the theory of dynamic risk measures is the notion of \emph{time-consistency}, which says that if from the perspective of some future time $\tau$, one sequence of costs is riskier than another and the two sequences of costs are identical from the present until $\tau$, then the first sequence is also riskier from the present perspective (see \cite{Ruszczynski2010} for the full technical definition).

Other definitions of time-consistency can be found in the literature, e.g., \cite{Boda2006}, \cite{Cheridito2009a}, and \cite{Shapiro2009}. Though they may differ technically, these definitions share the same intuitive spirit. Under the conditions:
\begin{equation*}
\rho_{t,T}(0,\ldots,0) = 0 \quad \textnormal{and} \quad \rho_{t,T}(C_t,C_{t+1},\ldots,C_T) = C_t + \rho_{t,T}(0,C_{t+1},\ldots,C_T),
\end{equation*}
it is proven in \cite{Ruszczynski2010} that for some \emph{one-step conditional risk measures} $\rho_t:\mathcal Z_{t+1} \rightarrow \mathcal Z_t$, a time-consistent, dynamic risk measure $\{\rho_{t,T}\}^T_{t=0}$ can be expressed using the following nested representation:
\[
\rho_{t,T}(C_t,\ldots,C_T) = C_t + \rho_t \bigl( C_{t+1} + \rho_{t+1}( C_{t+2} + \cdots+ \rho_{T-1}(C_T) \cdots )\bigr),
\]
It is thus clear that we can take the reverse approach and define a time-consistent dynamic risk measure by simply specifying a set of one-step conditional risk measures $\{\rho_t\}_{t=0}^T$. This is a common method that has been used in the literature when applying the theory of dynamic risk measures in practice (see, e.g., \cite{Philpott2012}, \cite{Philpott2013}, \cite{Shapiro2013}, \cite{Kozmik2014}, and \cite{Rudloff2014}).

\subsection{Quantile-Based Risk Measures} In this paper, we focus on simulation techniques where the one-step conditional risk measure belongs to a specific class of risk measures called \emph{quantile-based risk measures} (QBRM). Although the term \emph{quantile-based risk measure} has been used in the literature to refer to risk measures that are similar in spirit to VaR and CVaR (see, e.g., \cite{Dowd2006}, \cite{Neise2008}, {\cite{Sereda2010}), it has not been formally defined. First, let us describe these two popular risk measures, which serve to motivate a more general definition for a QBRM.

Also known as the \emph{quantile risk measure}, VaR is a staple of the financial industry (see, e.g., \cite{Duffie1997}). Given a real-valued random variable $X$ (representing a loss) and a risk level $\alpha \in (0,1)$, the \emph{VaR} or \emph{quantile} of $X$ is defined to be
\[
\textnormal{VaR}^\alpha(X) = q^\alpha(X) = \inf_{u} \bigl\{ \mathbf{P} (X \le u) \ge \alpha  \bigr\}.
\]
To simplify our notation, we use $q^\alpha(X)$ in the remainder of this paper. It is well known that VaR does not satisfy \emph{coherency} \cite{Artzner1999}, specifically the axiom of \emph{subadditivity}, an appealing property that encourages diversification. Despite this, several authors have given arguments in favor of VaR. For example, \cite{Danielsson2005} concludes that in practical situations, VaR typically exhibits subadditivity. \cite{Dhaene2006} and \cite{Ibragimov2007} give other points of view on why VaR should not be immediately dismissed as an effective measure of risk. A nested version of VaR for use in a multistage setting is proposed in \cite{Cheridito2009a}, though practical implications have not been explored in the literature. 

CVaR is a coherent alternative to VaR and has been both studied and applied extensively in the literature. Although the precise definitions may slightly differ, CVaR is also known by names such as \emph{expected shortfall}, \emph{average value at risk}, or \emph{tail conditional expectation}. 
Given a general random variable $X$, the following characterization is given in \cite{Rockafellar2002}:
\[
\quad \quad \quad \textnormal{CVaR}^\alpha(X) = \inf_{u} \Bigl\{ u + \frac{1}{1-\alpha} \, \mathbf{E}\bigl[ (X-u)^+ \bigr] \Bigr\} = q^\alpha(X) + \frac{1}{1-\alpha} \, \mathbf{E}\Bigl[ \bigl(X-q^\alpha(X)\bigr)^+ \Bigr].
\]
Applications of risk-averse MDPs using dynamic risk measures have largely focused on combining CVaR with expectation; once again, see \cite{Philpott2012}, \cite{Philpott2013}, \cite{Shapiro2013}, \cite{Kozmik2014}, and \cite{Rudloff2014}.

For the purposes of this paper, we offer the following general definition of a QBRM that allows dependence on more than one quantile; the definition includes the above two examples as special cases.
\begin{definition}[Quantile-Based Risk Measure (QBRM)]
Let $X$ be a real-valued random variable. A \emph{quantile-based risk measure} $\rho^\alpha$ can be written as the expectation of a function of $X$ and finitely many of its quantiles. More precisely, $\rho^\alpha$ takes the form
\begin{equation}
\rho^\alpha(X) = \mathbf{E} \Bigl[ \Phi\bigl(X,q^{\alpha_1}(X),q^{\alpha_2}(X),\ldots,q^{\alpha_m}(X) \bigr)   \Bigr],
\label{eq:rhoq}
\end{equation}
where $\alpha \in \mathbb R^m$ is a vector of $m$ risk levels, $\alpha_1$, $\alpha_2$, \ldots, $\alpha_m$, and a function $\Phi: \mathbb R^{m+1} \rightarrow \mathbb R$, chosen so that $\rho^\alpha$ satisfies monotonicity, translation invariance, and positive homogeneity (see \cite{Artzner1999} for the precise definitions and note that we interpret $X$ as a random loss or a cost).
\label{defn:qbrm}
\end{definition}
Our definition of QBRMs is largely motivated by practical considerations. First, the definition covers the two most widely used risk measures, VaR and CVaR, as special cases under a single framework; in addition, the flexibility allows for the specification of more sophisticated risk measures that may or may not be coherent. As previously mentioned, there are situations where nonconvex (and thus, not coherent) risk measures are appropriate \citep{Dhaene2006}. Another motivation for this definition of a QBRM is that it allows us to easily construct a risk measure such that $ \textnormal{VaR}^\alpha(X) \le \rho^\alpha(X) \le \textnormal{CVaR}^\alpha(X)$, because, as \cite{Belles-Sampera2014} points out, one issue with VaR is that it can underestimate large losses, but at the same time, some practitioners of the financial and insurance industries find $\textnormal{CVaR}$ to be too conservative.

We see that VaR is trivially a QBRM with $\Phi(X,q)=q$. CVaR can also be easily written as a QBRM, using the function $\Phi(X,q) = q+\frac{1}{1-\alpha} \, (X-q)^+$.
Although our approach can be applied to any risk measure of the form (\ref{eq:rhoq}), we use the CVaR risk measure in the empirical work of Section \ref{sec:numerical}, due to its popularity in a variety of application areas.

\subsection{Dynamic Quantile-Based Risk Measures}
Notice that, so far, we have developed QBRMs in a ``static'' setting (the value of the risk measure is in $\mathbb{R}$) for simplicity. Given a random variable $X \in \mathcal Z_{t+1}$ and a risk level $\alpha \in (0,1)$, the conditional counterpart for the quantile is
\[
q_t^\alpha(X) = \inf_{U \in \mathcal Z_t} \bigl\{ \mathbf{P} \bigl(X \le U \,|\, \mathcal F_t \bigr) \ge \alpha  \bigr\} \in \mathcal Z_t.
\] 
Using this new definition, we can similarly extend the definition of a QBRM to the conditional setting by replacing (\ref{eq:rhoq}) with
\begin{equation*}
\rho_t^\alpha(X) = \mathbf{E} \Bigl[ \Phi\bigl(X,q_t^{\alpha_1}(X),q_t^{\alpha_2}(X),\ldots,q_t^{\alpha_m}(X) \bigr) \, \bigr | \, \mathcal F_t  \Bigr],
\end{equation*}
and replacing the required properties of monotonicity, translation invariance, and positive homogeneity in Definition \ref{defn:qbrm} with their conditional forms given in \cite{Ruszczynski2010} (denoted therein by A2, A3, and A4).
For the sake of notational simplicity, let us assume that all parameters, i.e., $m$, $\alpha_1, \ldots, \alpha_m$, $\Phi$ are static over time, but we remark that an extension to time-dependent (and even state-dependent) versions of the one-step conditional risk measure is possible.
Let $\tilde{\rho}_t^{\alpha}$ be a (conditional) QBRM that measures tail risk. In applications, a weighted combination of a tail risk measure with the traditional expectation ensures that the resulting policies are not driven completely by the tail behavior of the cost distribution; we may use QBRMs of the form $\rho_t^{\alpha} (X) = (1-\lambda) \, \mathbf{E} \bigl[ X \,|\, \mathcal F_t  \bigr] + \lambda \, \tilde{\rho}_t^{\alpha}(X)$, where $\lambda \in [0,1]$.

Using one-step conditional risk measures as building blocks, we can define a dynamic risk measure to be $\{\rho_t^\alpha\}_{t=0}^T$, which we refer to as a \emph{dynamic quantile-based risk measure} (DQBRM). The dynamic risk measures obtained when $\rho_t^\alpha = \textnormal{VaR}_t^\alpha$ and $\rho_t^\alpha = \textnormal{CVaR}_t^\alpha$ (the conditional forms of VaR and CVaR) are precisely the time-consistent risk measures suggested in \cite{Cheridito2009a} under the names \emph{composed value at risk} and \emph{composed conditional value at risk}. 


\subsection{Objective Function} We are interested in finding optimal risk-averse policies under objective functions specified using a DQBRM. The problem is 
\begin{equation}
\min_{\pi \in \Pi} \; \rho_0^\alpha \Bigl ( C_1^\pi + \rho_1^\alpha \bigl( C_2^\pi + \cdots + \rho_{T-1}^\alpha( C_T^\pi ) \cdots \bigr) \Bigr).
\label{eq:rmdp}
\end{equation}
The upcoming theorem, proven in \cite{Ruszczynski2010}, gives the Bellman-like optimality equations for a risk-averse model. We state it under the assumption that the current period contribution is random, differing slightly from the original statement. A point of clarification: the original theorem is proved within the setting where the one-step risk measures satisfy conditional forms of the axioms of \cite{Artzner1999} for coherent risk measures. In our setting, however, the QBRM $\rho_t^{\alpha}$ is only assumed to satisfy (conditional forms of) monotonicity, positive homogeneity, and translation invariance, but not necessarily convexity. The crucial step of the proof given in \cite{Ruszczynski2010} relies only on monotonicity and an associated \emph{interchangeability} property (see \cite[Theorem 7.1]{Ruszczynski2006b}, \cite[Theorem 2]{Ruszczynski2010}). The assumption of convexity is therefore not necessary for the following theorem.
\begin{theorem}[Bellman Recursion for Dynamic Risk Measures, \cite{Ruszczynski2010}]
The sequential decision problem (\ref{eq:rmdp}) has optimal value functions given by
\begin{align*}
V_t^*(s) &= \min_{a_t \in \mathcal A_s} \rho_t^\alpha \left(c_t(s, a_t, W_{t+1}) + V_{t+1}^*(S_{t+1}) \right) \textnormal{ for all } s \in \mathcal S, \; t \in \mathcal T,\label{eq:bellman}\\
V_T^*(s) &= 0 \textnormal{ for all } s \in \mathcal S.
\end{align*}
The decision functions of an optimal policy $\pi^*$ are given by
\begin{align*}
A_t^{\pi^*}\!(s) \in \argmin_{a_t \in \mathcal A_s} \rho_t^\alpha \left(c_t(s, a_t, W_{t+1}) + V_{t+1}^*(S_{t+1}) \right) \textnormal{ for all } s \in \mathcal S, \; t \in \mathcal T,
\end{align*}
which map to a minimizing action of the optimality equation.
\label{thm:bellman}
\end{theorem}
For computational purposes, we are interested in interchanging the minimization operator and the risk measure $\rho_t^\alpha$ and thus appeal to the \emph{state-action value function} or \emph{$Q$-factor} formulation  of the Bellman equation. Define the state-action value function over the state-action pairs $(s,a) \in \mathcal U$ to be $Q_t^*(s,a) = \rho_t^\alpha \left(c_t(s, a, W_{t+1}) + V_{t+1}^*(S_{t+1}) \right)$, for $t\in \mathcal T$ and let $Q_T^*(s,a)=0$. Thus, the counterpart to the recursion in Theorem \ref{thm:bellman} is
\begin{equation}
Q_t^*(s, a) = \rho_t^\alpha \bigl(c_t(s, a, W_{t+1}) + \!\!\! \min_{a' \in \mathcal A_{S_{t+1}}} \!\!\! Q_{t+1}^*(S_{t+1},a') \bigr),
\label{eq:bellmanq}
\end{equation}
with the minimization occurring inside of the risk measure.
\subsection{Some Remarks on Notation} For simplicity, we henceforth refer to $Q^*$ simply as the \emph{optimal value function}. Let $d=|\mathcal U|$ and $D = |\mathcal U| \, (T+1)$. We consider $Q^*$ to be a vector in $\mathbb R^{D}$ with components $Q_t^*(s,a)$. We also frequently use the notation $Q_t^* \in \mathbb R^d$ for some $t\le T$, by which we mean $Q^*$ restricted to the components $Q_t^*(s,a)$ for all $(s,a) \in \mathcal U$. We adopt this system for any vector in $\mathbb R^{D}$ (e.g., $\bar{Q}^n$, $u^{i,*}$, and $\bar{u}^{i,n}$ to be defined later). The norms used in this paper are $\| \cdot \|_1$, $\| \cdot \|_2$, and $\|\cdot \|_\infty$, the $l_1$-norm, the Euclidean norm, and the maximum norm, respectively. Example usages of the latter two are
\[
\| Q_t^* \|_2 = \biggl(\sum_{(s,a) \in \, \mathcal U} Q_t^*(s,a)^2 \biggr)^\frac{1}{2} \quad \mbox{and} \quad \| Q_t^* \|_\infty = \max_{(s,a) \in \, \mathcal U} |Q^*_t(s,a)|.
\]
The following naming convention is used throughout the paper and appendix: stochastic processes denoted using $\epsilon$, i.e., $\epsilon_{t+1}^{q,n}$, $\epsilon_{t+1}^{i,n}$, and $\epsilon_{t+1}^{h,n}$, are \emph{conditionally unbiased} noise sequences and represent Monte Carlo sampling error. On the other hand, the processes denoted using $\xi$, i.e., $\xi_{t+1}^{q,n}$, $\xi_{t+1}^{i,n}$, and $\xi_{t+1}^{h,n}$, are \emph{biased} noise and represent approximation error from using a value function approximation. For a vector $v$, $\textnormal{diag}(v)$ is the diagonal matrix whose entries are the components of $v$. Lastly, for a nonnegative function $f:\mathbb R \rightarrow \mathbb R$, its \emph{support} is represented by the notation $\textnormal{supp}(f) = \{x \in \mathbb R: f(x) >  0 \}$. 


\section{Algorithm}
\label{sec:algorithm}
In this section, we introduce the risk-averse ADP algorithm for dynamic quantile-based risk measures, which aims to approximate the value function $Q^*$ in order to produce near-optimal policies. 
\subsection{Overview of the Main Idea} Like most ADP and reinforcement learning algorithms, the algorithm that we develop in this paper to solve (\ref{eq:rmdp}) is based on the recursive relationship of (\ref{eq:bellmanq}). The basic structure for the algorithm is a time-dependent version of $Q$-learning or approximate value iteration (see \cite[Chapter 10]{Powell2011} for a discussion). Recall the form of the QBRM:
\begin{equation*}
\rho_t^\alpha(X) = \mathbf{E} \Bigl[ \Phi\bigl(X,q_t^{\alpha_1}(X),q_t^{\alpha_2}(X),\ldots,q_t^{\alpha_m}(X) \bigr) \, \bigr | \, \mathcal F_t  \Bigr].
\end{equation*}
The main idea of our approach is to approximate the quantiles $q_t^{\alpha_i}(X)$ and then combine the approximations to form an estimate of the risk measure. In essence, every observation of the exogenous information process (real or simulated data) can be utilized to give an updated approximation of each of the $m$ quantiles. A second step then takes the observation and the quantile approximations to generate an refined approximation of the optimal value function $Q^*$. This type of logic is implemented using many concurrent \emph{stochastic gradient} \citep{Robbins1951,Kushner2003} steps within a framework that walks through a single forward trajectory of states and actions on each iteration. 

It turns out that there is a convenient characterization of the quantile through the so-called CVaR \emph{minimization formula}. Given a real-valued, integrable random variable $X$, a risk level $\alpha_i \in (0,1)$, and $u \in \mathbb R$, \cite{Rockafellar2000} proves that
\begin{equation}
q^{\alpha_i}(X) \in \argmin_{u \in \mathbb R} \, \mathbf{E} \left[u + \frac{1}{1-\alpha_i} \, (X-u)^+  \right].
\label{eq:cvarmin}
\end{equation}
Although the main result of \cite{Rockafellar2000} is that the optimal value of the optimization problem gives the $\textnormal{CVaR}^{\alpha_i}(X)$, the characterization of the quantile as the minimizer is particularly useful for our purposes. It suggests the use of \emph{stochastic approximation} or \emph{stochastic gradient descent} algorithms \citep{Robbins1951,Kushner2003} to iteratively optimize (\ref{eq:cvarmin}).

With this intuition in mind, let us move back to the context of the MDP and define the auxiliary variables $u^{i,*} \in \mathbb R^D$, for each $i \in \{1,2,\ldots,m\}$, to be the $\alpha_i$-quantiles of the future costs (recall that the quantiles $\alpha_i$ are defined as an argument to our QBRM in Definition \ref{defn:qbrm}). The component at time $t$ and state $(s,a)$ is
\begin{equation}
u^{i,*}_t(s,a) = q^{\alpha_i}\bigl(c_t(s, a, W_{t+1}) + \min_{a' \in \mathcal A_{S_{t+1}}} \! \! \! Q_{t+1}^*(S_{t+1},a') \bigr),
\label{eq:aux-u}
\end{equation}
for each $i \in \{1,2,\ldots,m\}$. Using (\ref{eq:rhoq}), this allows us to take advantage of the equation
\begin{equation}
\begin{aligned}
Q_{t}^* &(s,a) = \mathbf{E} \Bigl[ \Phi \bigl( c_t(s, a, W_{t+1}) + \min_{a' \in \mathcal A_{S_{t+1}}} \! \! \! Q_{t+1}^*(S_{t+1},a'), u^{1,*}_t(s,a), \ldots, u^{m,*}_t(s,a) \bigr)  \Bigr].
\end{aligned}
\label{eq:Qustar}
\end{equation}
The relationship between $Q^*$ and $u^{i,*}$ is fundamental to our algorithmic approach, which keeps track of mutually dependent approximations $\{ \bar{u}^{i,n} \}$ and $\{\bar{Q}^n\}$ to the optimal values $u^{i,*}$ and $Q^{*}$, respectively.

\subsection{The Dynamic-QBRM ADP Algorithm} 
Before discussing the details, we need some additional notation.
Clearly, at each time $t$, the random quantity with which we are primarily concerned (and attempt to approximate) is the future cost given the optimal value function $Q^*$. Thus, we explicitly define its distribution function for every $(s,a)$:
\[
F_t(x\<|\<s,a) = \mathbf{P} \Bigl  [ c_t\bigl(s,a,W_{t+1}\bigr)  + \min_{a' \in \mathcal A_{S_{t+1}}} \! \! \! Q^*_{t+1}\bigl(S_{t+1},a'\bigr)  \le x  \Bigr  ].
\]
Recall that $d$ is the cardinality of the state-action space. Next, suppose $\bar{u}^i_t \in \mathbb R^d$ is an approximation of $u_t^{i,*}$ and for each $t$ and $i \in \{1,2,\ldots,m\}$, define the \emph{stochastic gradient mapping} $\psi^i_t : \mathbb R^d \times \mathbb R^d \times \mathcal W \rightarrow \mathbb R^d$ to perform the stochastic gradient computation:
\begin{equation}
\begin{aligned}
\psi_t^i\bigl(\bar{u}^i_t, \bar{Q}_{t+1}, &W_{t+1}   \bigr)(s,a) \\
&= 1 - \frac{1}{1-\alpha_i} \, \mathbf{1} \Bigl \{ c_t\bigl(s,a,W_{t+1}\bigr)  + \min_{a' \in \mathcal A_{S_{t+1}}} \! \! \! \bar{Q}_{t+1}\bigl(S_{t+1},a'\bigr) \ge \bar{u}^i_t(s,a) \Bigr\}.
\end{aligned}
\label{eq:psi_def}
\end{equation}
To avoid confusion, we note that this is the stochastic gradient associated with the minimization formula (\ref{eq:cvarmin}) for CVaR, but this step is necessary for any QBRM, even if we are not utilizing CVaR. 

The second piece of notation we need is a specialized, stochastic version of the \emph{Bellman operator} to the risk-averse case: for each $t$, we define the mapping $H_t : \mathbb R^d \times \cdots \times \mathbb R^d \times \mathcal W \rightarrow \mathbb R^d$, with $m+1$ arguments in $\mathbb R^d$, to represent an approximation of the term within the expectation of (\ref{eq:Qustar}):
\begin{align*}
 H_t\bigl(\bar{u}^{1}_t,\ldots,&\bar{u}^{m}_t, \bar{Q}_{t+1}, W_{t+1}   \bigr)(s,a) \\
 &= \Phi \bigl( c_t\bigl(s,a,W_{t+1}\bigr)  + \!  \! \min_{a' \in \mathcal A_{S_{t+1}}} \! \! \! \bar{Q}_{t+1}\bigl(S_{t+1},a'\bigr), \bar{u}^{1}_t(s,a), \ldots, \bar{u}^{m}_t(s,a) \bigr).
 \end{align*}
Therefore, (\ref{eq:Qustar}) can be rewritten using the stochastic Bellman operator $H_t$ by replacing all approximate quantities with their true values:
\begin{equation}
Q_{t}^* (s,a) = \mathbf{E} \Bigl[ H_t\bigl(u^{1,*}_t,\ldots,u^{m,*}_t, Q^*_{t+1}, W_{t+1}   \bigr)(s,a)  \Bigr].
\label{eq:optimalrecursion}
\end{equation}
The Dynamic-QBRM ADP algorithm that we describe in the next section consists of both outer and inner iterations: for each outer iteration $n$, we step through the entire time horizon of the problem $t \in \mathcal T$. At time $t$, iteration $n$, the relevant quantities for our algorithms are a state-action pair $(S_t^n, a_t^n) \in \mathcal U$ and two samples $W_{t+1}^{u,n}, \, W_{t+1}^{q,n} \in \mathcal W$ from the distribution of $W_{t+1}$ corresponding to the ``two steps'' of our algorithm, one for approximating the auxiliary variables $u^{i,*}$ and the second for approximating the value function $Q^{*}$. Figure \ref{fig:algidea} illustrates the main idea behind the algorithm: we merge the results of $m$ adaptive minimizations of (\ref{eq:cvarmin}), corresponding to estimates of the $m$ quantiles, into an estimate of the optimal value function, $\bar{Q}_t^n$. The estimate $\bar{Q}_t^n$ is then used to produce estimates of the relevant quantities for the previous time period. Note that the $m$ objective functions shown in the figure differ only in their risk levels $\alpha_i$. The arrows on the curves indicate that the minimizations are achieved via gradient descent steps.


\begin{figure}[h]
	\centering
	\includegraphics[width=\textwidth]{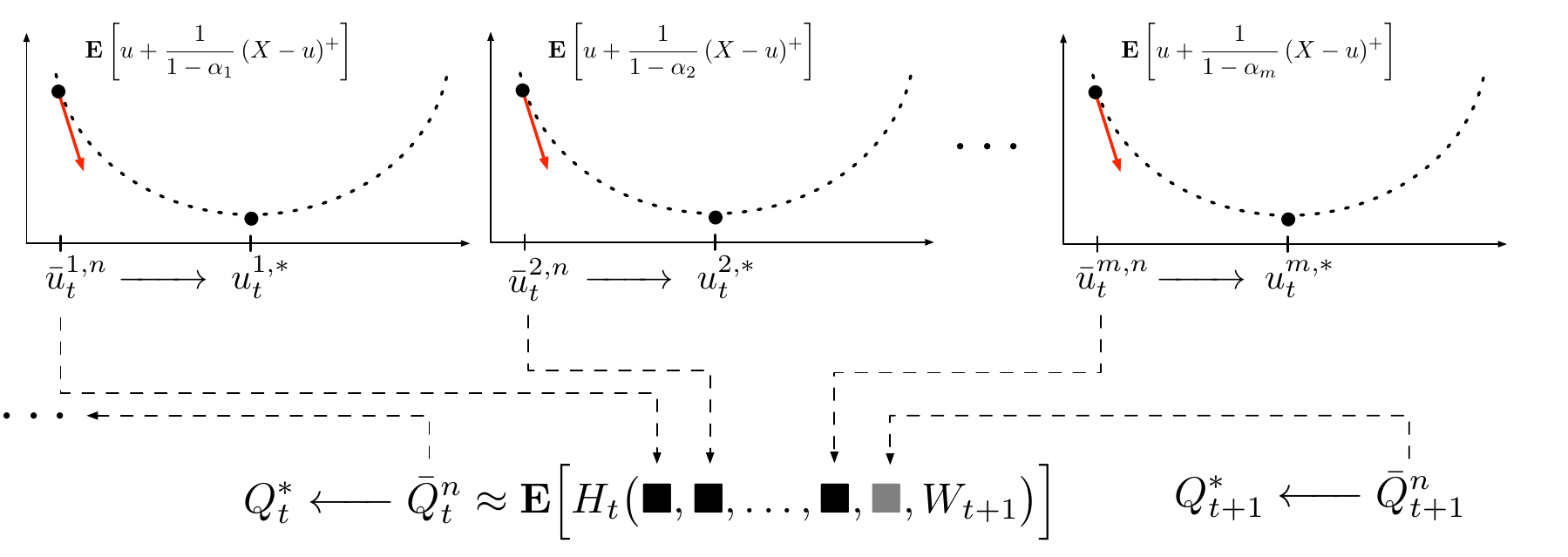}
	\caption{Main Algorithm Idea}
	\label{fig:algidea}
\end{figure}


Now that we are in an algorithmic setting, we consider a new probability space $(\Omega, \mathcal G, \mathbf P)$, where $\mathcal G = \sigma \bigl\{ (S_{t}^n, a_{t}^n, W_{t}^{u,n}, W_{t}^{q,n}),\,n \ge 0,\, t \le T \bigr\}$.
In order to describe the history of the algorithms, we define:
\[
\mathcal G_t^n = \sigma \bigl\{ \{ (S_{\tau}^k, a_{\tau}^k, W_{\tau}^{u,k}, W_{\tau}^{q,k}),\,k < n,\,\tau \le T \} \cup \{ (S_{\tau}^n, a_{\tau}^n, W_{\tau}^{u,n}, W_{\tau}^{q,n}),\, \tau \le t \}\bigr\},
\]
for $t \in \mathcal T$ and $n \ge 1$, with $\mathcal G_t^0 = \{\varnothing,\Omega\}$ for all $t \le T$. We therefore have a filtration that obeys $\mathcal G_t^n \subseteq \mathcal G_{t+1}^n$ for $t \in \mathcal T$ and $\mathcal G_{T}^n \subseteq \mathcal G_0^{n+1}$, coinciding precisely with the progression of the algorithm. The random variables $(S_t^n,a_t^n)$ are generated according to some sampling policy (to be discussed later)
while $W_{t+1}^{u,n}$ and $W_{t+1}^{q,n}$ are generated from the distribution of the exogenous process $W_{t+1}$.

Crucial to many ADP algorithms is the \emph{stepsize} (or \emph{learning rate}). In our case, we use $\{\gamma_t^n\}$ and $\{\eta_t^n\}$ for smoothing new observations with previous estimates, where $\gamma_t^n,\,\eta_t^n \in \mathbb R^d$ for each $t$ and $n$ and are $\mathcal G_t^n$-measurable. The stepsize $\gamma_t^n$ is used to update our approximation of $u_t^{i,*}$ while the stepsize $\eta_t^n$ is used to update $Q_t^*$; see Algorithm \ref{alg:dq}. We articulate the asynchronous nature of our algorithm by imposing the following condition on the stepsizes (included in Assumption \ref{ass:algorithm} of Section \ref{sec:algconv}):
\begin{equation*}
(s,a) \ne (S_t^n, a_t^n) \; \Longrightarrow \; \gamma_t^n(s,a) = \eta_t^n(s,a) = 0,
\label{eq:asynchstep}
\end{equation*}
which causes updates to only happen for states that we actually visit.

Stochastic approximation theory often requires a projection step onto a compact set (giving bounded iterates) to ensure convergence \citep{Kushner2003}. Hence, for each $t$ and $(s,a)$, let $\mathcal X^u_t(s,a) \subseteq \mathbb R$ and $\mathcal X^q_t(s,a) \subseteq \mathbb R$ be compact intervals and let
\begin{equation*}
\mathcal X^u_t = \bigl\{ x \in \mathbb R^d : x(s,a) \in \mathcal X_t^u(s,a)\bigr\} \mbox{ and } \mathcal X^q_t = \bigl\{ x \in \mathbb R^d : x(s,a) \in \mathcal X_t^q(s,a)\bigr\},
\end{equation*}
be our projection sets at time $t$. The Euclidean projection operator to a set $\mathcal X$ is given by the usual definition:
\[
\Pi_{\mathcal X}(y) = \argmin_{x \in \mathcal X} \| y-x   \|_2^2.
\]
These sets may be chosen arbitrarily large in practice and our first theoretical result (almost sure convergence) will continue to hold. However, there is a tradeoff: if, in addition, we want our convergence \emph{rate} results to hold, then these sets also cannot be too large (see Assumption \ref{ass:strongconv}).

The precise steps of Dynamic-QBRM ADP are given in Algorithm \ref{alg:dq}. A main characteristic of the algorithm is that sequences $\{\bar{u}^{i,n}\}$ and $\{\bar{Q}^n\}$ are intertwined (i.e., depend on each other). Consequently, there are multiple levels of approximation being used throughout the steps of the algorithm. The theoretical results of the subsequent sections shed light onto these issues.

\IncMargin{1em}
\begin{algorithm}
\small
  \SetKwInput{Input}{Input}
  \SetKwInput{Output}{Output}
  \DontPrintSemicolon
\Indm  
  \Input{Initial estimates of the value function $\bar{Q}^0 \in \mathbb R^D$ and quantiles $\bar{u}^{i,0} \in \mathbb R^D$ for $i \in \{1,2,\ldots,m\}$.}
  \myinput{Stepsize rules $\gamma_t^n$ and $\eta_t^n$ for all $n$, $t$.}
  \BlankLine
  \Output{Approximations $\{\bar{Q}^n\}$ and $\{\bar{u}^{i,n}\}$.}
\Indp
  \BlankLine
  \nl Set $\bar{Q}_T^n = 0$ for all $n$.\;
  \For{$n = 1, 2, \ldots$}{
  \nl  Choose an initial state $(S_0^n, a_0^n)$.\;
  \For{$t = 0, 1, \ldots, T-1$}{
  \nl  Draw samples of the information process $W_{t+1}^{u,n}, W_{t+1}^{q,n} \in \mathcal W$.
    \BlankLine

\nl  Update auxiliary variable approximations for $i=1,\ldots,m$: \\
    $\quad \quad \quad \quad \bar{u}_t^{i,n} = \Pi_{\mathcal X_t^u} \Bigl \{ \bar{u}_t^{i,n-1} - \textnormal{diag}(\gamma_t^n) \, \psi^i_t \bigl(\bar{u}_t^{i,n-1},\bar{Q}_{t+1}^{n-1}, W_{t+1}^{u,n}\bigr)\Bigr\}.$\;
  \BlankLine

  \nl  Compute an estimate of the future cost based on the current approximations: \\
  $\quad \quad \quad \quad \hat{q}_t^n = H_t\bigl(\bar{u}^{1,n-1}_t,\ldots,\bar{u}^{m,n-1}_t, \bar{Q}_{t+1}^{n-1}, W_{t+1}^{q,n}   \bigr)$.\;
  \BlankLine

 \nl  Update approximation of value function: \\
    $\quad \quad \quad \quad \bar{Q}_t^n = \Pi_{\mathcal X_t^q} \Bigl \{\bar{Q}_t^{n-1} - \textnormal{diag}(\eta_t^n) \, \bigl(\bar{Q}_t^{n-1}-\hat{q}_t^n\bigr) \Bigr \}.$\;
  \BlankLine

 \nl  Choose next state $(S_{t+1}^n, a_{t+1}^n)$.
  }
  }
    \caption{Dynamic-QBRM ADP}
    \label{alg:dq}
\end{algorithm}
\DecMargin{1em}

\section{Analysis of Convergence}
\label{sec:algconv}
In this section, we state and prove convergence theorems for Algorithm \ref{alg:dq}. First, we give an overview of our analysis and the relationship to existing work. 

\subsection{A Preview of Results}
The two main results of this section are: (1) the almost sure convergence of Dynamic-QBRM ADP and (2) a convergence rate result under a particular sampling policy called $\varepsilon$-greedy.
The proof of almost sure convergence uses techniques from the stochastic approximation literature \citep{Kushner2003}, which were applied to the field of reinforcement learning and $Q$-learning in \cite{Tsitsiklis1994a}, \cite{Jaakkola1994} and \cite{Bertsekas1996}. However, our algorithm differs from risk-neutral $Q$-learning in that it tracks multiple quantities, $\bar{u}_t^{1,n}, \bar{u}_t^{2,n}, \ldots, \bar{u}_t^{m,n}, \bar{Q}_t^n$, over a horizon $1, 2, \ldots, T$. The intuition behind the proof is that multiple ``stochastic approximation instances'' are pasted together in order to obtain overall convergence of all relevant quantities. Accordingly, the interdependence of various approximations means that in several parts of the proof, we require careful analysis of \emph{biased noise terms} (or \emph{approximation error}) in addition to \emph{unbiased statistical error}. See, e.g., \cite{Kearns1999}, \cite{EvenDar2004} and \cite{Azar2011}, for convergence rate results for standard $Q$-learning. The proof technique used to analyze the high probability convergence rate of risk-neutral $Q$-learning in \cite{EvenDar2004} is based on the same types of stochastic approximation results that we utilize in this paper.


Let us now make a few remarks regarding some simplifying assumptions made in this paper. As proven in \cite[Theorem 10]{Rockafellar2002}, the set of minimizers $\argmin_{u \in \mathbb R} \mathbf{E} \bigl[u + (1-\alpha_i)^{-1} \, (X-u)^+  \bigr]$ is a nonempty, closed, and bounded interval for a general $X$. We shall for ease of presentation, however, make assumptions (strictly increasing and continuous cdf, Assumption \ref{ass:problem}(iii)) to guarantee that $u_t^{i,*}(s,a)$ is the unique minimizer when $X$ is the optimal future cost $c_t(s, a, W_{t+1}) + \min_{a' \in \mathcal A_{S_{t+1}}} Q_{t+1}^*(S_{t+1},a')$ and that gradient computations to remain valid. This assumption is sufficient for almost sure convergence (Theorem \ref{thm:asconv}). To further examine the convergence rate of the algorithm (Theorem \ref{thm:convrate1}), we must additionally have Assumption \ref{ass:strongconv}, which states that the density of the future cost exists and is positive within the constraint sets $\mathcal X_t^u$ --- this provides us the technical condition of strong convexity (discussed more in Section \ref{sec:algconvrate} below). 

Since $\min_{a' \in \mathcal A_{S_{t+1}}} \! Q_{t+1}^*(S_{t+1},a')$ has a discrete distribution, the assumptions hold only in certain situations: an obvious case is when the current stage cost has a density and is independent of $S_{t+1}$. For example, such a property holds when $W_{t+1}$ can be written as two independent components $(W_{t+1}^1,W_{t+1}^2)$ where the current stage cost depends on $W_{t+1}^1$ and the downstream state depends on $W_{t+1}^2$. This model is relevant in a number of applications; notable examples include multi-armed bandits \citep{Whittle1980}, shortest path problems with random edge costs \citep{Ryzhov2011}, trade execution with temporary (and temporally independent) price impact \citep{Bertimas1998}, and energy trading in two-settlement markets \citep{Lohndorf2014}. Small algorithmic extensions (requiring more complex notation) to handle the general case are possible, but the fundamental concepts would remain unchanged. Hence, we will assume the cleaner setting for the purposes of this paper.



\subsection{Almost Sure Convergence} First, we discuss the necessary algorithmic assumptions, many of which are standard to the field of stochastic approximation.
\begin{assumption} 
For all $(s,a) \in \mathcal U$ and $t \in \mathcal T$, suppose the following are satisfied:
\begin{enumerate}[label=(\roman*),labelindent=1in]
\vspace{0.5em}
\item $\gamma_t^n(s,a) = \tilde{\gamma}_t^{n-1} \, \mathbf{1}_{\{(s,a) = (S_t^n,a_t^n)\}}$, for some $\tilde{\gamma}_t^{n-1}\in \mathbb R$ that is $\mathcal G_t^{n-1}$-measurable,
\vspace{0.5em}
\item $\eta_t^n(s,a) = \tilde{\eta}_t^{n-1} \, \mathbf{1}_{\{(s,a) = (S_t^n,a_t^n)\}}$, for some $\tilde{\eta}_t^{n-1}\in \mathbb R$ that is $\mathcal G_t^{n-1}$-measurable,
\vspace{0.5em}
\item $\displaystyle \sum_{n=1}^\infty \gamma_t^n(s,a) = \infty, \quad \sum_{n=1}^\infty \eta_t^n(s,a) = \infty  \quad a.s.$,
\item $\displaystyle \sum_{n=1}^\infty \gamma_t^n(s,a)^2 < \infty, \quad \sum_{n=1}^\infty \eta_t^n(s,a)^2 < \infty \quad a.s.$,
\vspace{0.5em}
\item $\exists \; \varepsilon > 0$, such that state sampling policy satisfies
\[
\mathbf{P}\bigl((S_t^n,a_t^n) =(s,a) \, \bigl| \, \mathcal G^{n}_{t-1} \bigr) \ge \varepsilon \quad \mbox{and} \quad \mathbf{P}\bigl((S_0^n,a_0^n) =(s,a) \, \bigl| \, \mathcal G^{n-1}_{T} \bigr) \ge \varepsilon,
\]
\item the projection sets are chosen large enough so that $u_t^{i,*} \in \mathcal X_t^u$ for each $i \in \{1,2,\ldots,m\}$ and $Q_t^* \in \mathcal X_t^q$.
\end{enumerate}
\label{ass:algorithm}
\end{assumption}

Assumption \ref{ass:algorithm}(i) and (ii) represent the asynchronous nature of the algorithm, sending the stepsize to zero whenever a state is not visited, while (iii) and (iv) are standard conditions on the stepsize. Assumption \ref{ass:algorithm}(v) is an \emph{exploration} requirement; by the Extended Borel-Cantelli Lemma (see \citet{Breiman1992}), sampling with this exploration requirement guarantees that we will visit every state infinitely often with probability one. In particular, for the case with an $\varepsilon$-greedy sampling policy (i.e., explore with probability $\varepsilon$, follow current policy otherwise), this assumption holds. We discuss this policy in greater detail in Section \ref{sec:algconv}. Part (vi) is a technical assumption. The second group of assumptions that we present are related to the problem parameters.


\begin{assumption}
The following hold:
\begin{enumerate}[label=(\roman*),labelindent=1in]
\item the risk-aversion function $\Phi : \mathbb R^{m+1} \rightarrow \mathbb R$ (from the QBRM within the one-step conditional risk measure $\rho_t^\alpha$) is Lipschitz continuous with constant $L_\Phi > 0$, i.e., for all $v,\,w \in \mathbb R^{m+1}$, $|\Phi(v) - \Phi(w) | \le L_\Phi \, \|v-w\|_1$,\vspace{0.5em}
\item $\exists \; C_\textnormal{max} > 0$ such that $\mathbf{E} \bigl[ c_t(s,a,W_{t+1}) ^2  \bigr] \le C_\textnormal{max}$ for all $(s,a) \in \mathcal U$ and $t \in \mathcal T$,\vspace{0.5em}
\item the distribution function $F_t(x\<|\<s,a)$ is strictly increasing and Lipschitz continuous with constant $L_F > 0$, i.e.,
\[
| F_t(x\<|\<s,a) - F_t(y\<|\<s,a) | \le L_F \, | x - y |,
\]
for all $x,\,y \in \mathbb R$, $(s,a) \in \mathcal U$, and $t \in \mathcal T$.\end{enumerate}
\label{ass:problem}
\end{assumption}

As an example for Assumption \ref{ass:problem}(i), in the case of the QBRM being CVaR, where $\Phi(X,q) = q + \frac{1}{1-\alpha}\,(X-q)^+$, it is easy to see that we can take $L_\Phi = 1+\frac{1}{1-\alpha}$. Assumption \ref{ass:problem}(ii) states that the second moment of the cost function is bounded. Assumption \ref{ass:problem}(iii) and Assumption \ref{ass:algorithm}(vi) together imply that $u_t^{i,*}(s,a)$ is the unique $u \in \mathcal X_t^u(s,a)$ such that $F_t(u\<|\<s,a) = \alpha_i$. 





\begin{lemma}
Under Assumptions \ref{ass:algorithm} and \ref{ass:problem}, if $\bar{Q}_{t+1}^n \rightarrow Q_{t+1}^*$ almost surely, then the sequences of iterates $\bar{u}_t^{i,n}$ generated in Step 4 of Algorithm \ref{alg:dq} satisfy $\bar{u}_t^{i,n} \rightarrow u_t^{i,*}$ almost surely for each $i \in \{1,2,\ldots,m\}$.
\label{lem:uconv}
\end{lemma}
\begin{proof}
We distinguish between two types of noise sequences (in $n$) for each $t \in \mathcal T$, the \emph{statistical error} and the \emph{approximation error}, denoted by $\epsilon_{t+1}^{i,n} \in \mathbb R^d$ and $\xi_{t+1}^{i,n} \in \mathbb R^d$, respectively. The definitions are
\begin{align*}
\epsilon_{t+1}^{i,n} &= \psi_t^i\bigl(\bar{u}_t^{i,n-1},Q_{t+1}^{*},W_{t+1}^{u,n}\bigr)-\mathbf{E} \bigl[ \psi^i_t\bigl(\bar{u}^{i,n-1}_t,Q^*_{t+1},W_{t+1}\bigr)  \bigr],\\
\xi_{t+1}^{i,n} &= \psi^i_t\bigl(\bar{u}_t^{i,n-1},\bar{Q}_{t+1}^{n-1},W_{t+1}^{u,n}\bigr)- \psi_t^i\bigl(\bar{u}_t^{i,n-1},Q_{t+1}^{*},W_{t+1}^{u,n}\bigr),
\end{align*}
 and we see that the random variable $\epsilon_{t+1}^{i,n}$ represents the error that the sample gradient deviates from its mean, computed using the \emph{true future cost} distribution (i.e., assuming we have $Q_{t+1}^*$). On the other hand, $\xi_{t+1}^{i,n}$ is the error between the two evaluations of $\psi^i_t$ given the same sample $W_{t+1}^{u,n}$, due only to the difference between $\bar{Q}_{t+1}^{n-1}$ and $Q_{t+1}^*$. Rearranging, we have
\[
\psi^i_t \bigl(\bar{u}_t^{i,n-1},\bar{Q}_{t+1}^{n-1}, W_{t+1}^{u,n}\bigr) = \mathbf{E} \bigl[ \psi^i_t\bigl(\bar{u}^{i,n-1}_t,Q^*_{t+1},W_{t+1}\bigr)  \bigr] + \epsilon_{t+1}^{i,n} + \xi_{t+1}^{i,n},
\]
which implies that the update given in Step 4 of Algorithm \ref{alg:dq} can be rewritten as
\[
\bar{u}_t^{i,n} = \Pi_{\mathcal X_t^u} \Bigl \{ \bar{u}_t^{i,n-1} - \textnormal{diag}(\gamma_t^n) \, \Bigl[\mathbf{E} \bigl[ \psi^i_t\bigl(\bar{u}^{i,n-1}_t,Q^*_{t+1},W_{t+1}\bigr)  \bigr] + \epsilon_{t+1}^{i,n} + \xi_{t+1}^{i,n} \Bigr]\Bigr\}.
\]
Note that the term in the square brackets is a \emph{biased stochastic gradient} and observe that it is bounded (since $\psi^i_t$ only takes two finite values). For the present inductive step at time $t$, let us fix a state $(s,a)$. It now becomes convenient for us to view $\bar{u}_t^{i,n}(s,a)$ as a stochastic process in $n$, adapted to the filtration $\{\mathcal G_{t+1}^n\}_{n \ge 0}$ (since $\mathcal G_{t+1}^n \subseteq \mathcal G_{t+1}^{n+1} \subseteq \mathcal G_{t+1}^{n+2} \cdots$). It is clear that by the definition of $\epsilon_{t+1}^{i,n}(s,a)$:
\begin{equation}
\mathbf{E} \bigl[ \epsilon_{t+1}^{i,n}(s,a) \, | \bigr. \, \mathcal G_{t+1}^{n-1} \bigr ] =0 \quad a.s.
\label{eq:unbiased}
\end{equation}
Therefore, $\epsilon_{t+1}^n(s,a)$ are unbiased increments that can be referred to as \emph{martingale difference noise}.
Before continuing, notice the following useful fact:
\begin{equation}
\min_{a \in \mathcal A_s} Q_{t+1}^*(s,a) - \min_{a \in \mathcal A_s} \bar{Q}_{t+1}^{n-1}(s,a) \le \bigl\| \bar{Q}^{n-1}_{t+1} - Q_{t+1}^* \bigr\|_\infty.
\label{eq:minmax}
\end{equation}
The proof follows from $\min v = -\max \,(-v)$ and $\max v -\max w \le \max \, |v-w|$, where the minimum and maximum are taken over the components of some vectors $v$ and $w$. Now, let $S_{t+1}^{n}= S^M\bigl(s,a,W_{t+1}^{u,n}\bigr)$.
Expanding the definition of $\xi_{t+1}^{i,n}(s,a)$ and using (\ref{eq:minmax}), we obtain
\begin{align*}
 \xi_{t+1}^{i,n}&(s,a) \\
 &= \begin{aligned}[t] &\frac{1}{1-\alpha_i} \biggl[\,\indicate{c_t \bigl(s,a,W_{t+1}^{u,n}\bigr) + \! \! \min_{a' \in \mathcal A_{S^{n}_{t+1}}} \!\!\! Q^{*}_{t+1}\bigl(S_{t+1}^{n},a'\bigr) \ge \bar{u}^{i,n-1}_t(s,a)} \\
&  - \indicate{c_t\bigl(s,a,W_{t+1}^{u,n}\bigr) + \!\! \min_{a' \in \mathcal A_{S^{n}_{t+1}}} \!\!\! \bar{Q}^{n-1}_{t+1}\bigl(S^{n}_{t+1},a'\bigr) \ge \bar{u}^{i,n-1}_t(s,a)} \biggr] \end{aligned}\\
& \le \begin{aligned}[t]&\frac{1}{1-\alpha_i} \biggl[\, \indicate{c_t \bigl(s,a,W_{t+1}^{u,n}\bigr) + \!\! \min_{a' \in \mathcal A_{S^{n}_{t+1}}}  \!\!\! Q^{*}_{t+1}\bigl(S_{t+1}^{n},a'\bigr) \ge \bar{u}^{i,n-1}_t(s,a)} \\
&  - \indicate{c_t\bigl(s,a,W_{t+1}^{u,n}\bigr) + \!\! \min_{a' \in \mathcal A_{S^{n}_{t+1}}} \!\!\! Q^{*}_{t+1}\bigl(S^{n}_{t+1},a'\bigr) \ge \bar{u}^{i,n-1}_t(s,a) + \bigl\|  \bar{Q}^{n-1}_{t+1} - Q_{t+1}^* \bigr\|_\infty} \biggr]. \end{aligned}
\end{align*}
Using the shorthand $F_t(\,\cdot\,) = F_t(\,\cdot \<|\<s,a)$ and taking the conditional expectation of both sides, we get (almost surely)
\begin{align}
\bigl|\mathbf{E} \bigl[ \xi_{t+1}^{i,n}(s,a)  \, | \bigr. \, \mathcal G_{t+1}^{n-1} \bigr ] \bigr| &\le \frac{1}{1-\alpha_i} \Bigl |  F_t \Bigl( \bar{u}_t^{i,n-1}(s,a) + \bigl\|  \bar{Q}^{n-1}_{t+1} - Q_{t+1}^*  \bigr \|_\infty\Bigr) - F_t \bigl( \bar{u}_t^{i,n-1}(s,a)\bigr)   \Bigr|\nonumber\\
&\le \frac{L_F}{1-\alpha_i} \, \bigl\| \bar{Q}^{n-1}_{t+1} - Q_{t+1}^* \bigr \|_\infty,
\label{eq:xibound}
\end{align}
where the second inequality follows by Assumption \ref{ass:problem}(iii).
Since we assumed in the statement of the lemma that $\bar{Q}_{t+1}^n \rightarrow Q_{t+1}^*$ almost surely, it must be the case that
\begin{equation}
\mathbf{E} \bigl[ \xi_{t+1}^{i,n}(s,a) \, \bigl| \bigr. \, \mathcal G_{t+1}^{n-1} \bigr ] \rightarrow 0 \quad a.s.,
 \label{eq:xi_negligible}
\end{equation}
and hence the noise ``vanishes asymptotically'' in expectation.
Now, given the boundedness of $\psi_t^i$ (and hence, finite second moment of $\psi_t^i$), the unbiasedness property (\ref{eq:unbiased}), the vanishing noise property (\ref{eq:xi_negligible}), the stepsize and sampling properties of Assumption \ref{ass:algorithm}, and the uniqueness of $u_t^*(s,a)$ from Assumption \ref{ass:problem}(iii), we can apply a classical theorem of stochastic approximation, \cite[Theorem 2.4]{Kushner2003}, to conclude that $\bar{u}_t^{i,n}(s,a) \rightarrow u_t^{i,*}(s,a)$ almost surely for each $i$. Because we chose an arbitrary $(s,a)$, this convergence holds for all $(s,a) \in \mathcal U$.
\end{proof}

\begin{restatable}[Almost Sure Convergence]{theorem}{thmasconv}
Choose initial approximations $\bar{Q}^0 \in \mathbb R^D$ and $\bar{u}^{i,0} \in \mathbb R^D$ for each $i \in \{1,2,\ldots,m\}$ so that $\bar{Q}_t^0 \in \mathcal X_t^q$ and $\bar{u}_t^{i,0} \in \mathcal X_t^u$ for all $t \in \mathcal T$. Under Assumptions \ref{ass:algorithm}--\ref{ass:problem},  Algorithm \ref{alg:dq} generates a sequence of iterates $\bar{Q}^n$ that converge almost surely to the optimal value function $Q^*$.\label{thm:convcvar}
\label{thm:asconv}
\end{restatable}
\begin{proof}[Sketch of Proof:] The idea of the proof is to induct backwards on $t$ and repeatedly apply Lemma \ref{lem:uconv}. At each step, the induction hypothesis is that $\bar{Q}_{t+1}^n \rightarrow Q_{t+1}^*$ from which we obtain convergence of $\bar{u}_t^{i,n} \rightarrow u_t^{i,*}$ from Lemma \ref{lem:uconv}. After making sure certain technical details are satisfied, stochastic approximation theory allows us to show that $\bar{Q}_{t}^n \rightarrow Q_{t}^*$. The full details are given in Appendix \ref{sec:appendix}.
\end{proof}

\subsection{Convergence Rate}
\label{sec:algconvrate}
In this section, we discuss the convergence rate (in terms of the expected deviation to $Q^*$) of the procedure described in Algorithm \ref{alg:dq}. Because we are working with an asynchronous algorithm where only one state per time period is visited every iteration, it is necessary for us to specify the assumed \emph{state sampling policy} for visiting states before deriving convergence rate results. Due to Assumption \ref{ass:algorithm}(i)--(ii), we naturally must consider the sampling policy and the stepsize sequences jointly, as is done in the upcoming proposition. We employ the $\varepsilon$\emph{-greedy} policy, a popular choice that balances exploitation and exploration using a tunable parameter $\varepsilon$; it is defined as follows. For any iteration $n>0$, choose $(S_0^n,a_0^n)$ independently and uniformly at random. At iteration $n$, time $t > 0$, let 
\[
s_{t+1}^n = S^M(S_{t}^n, a_t^n, W_{t+1}^n) \quad \mbox{and} \quad a^n_{t+1} = \argmin_{a \in \mathcal A} \bar{Q}^{n-1}_{t+1}(s_{t+1}^n, a)
\]
and define a Bernoulli random variable $X_{t+1}^n$ with parameter $1-\varepsilon d$ that is independent from $\mathcal G_t^n$ (i.e., the rest of the system).
The next state to visit is selected by the rule
\begin{equation}
\bigl(S_{t+1}^n, a_{t+1}^n\bigr) = \begin{cases} \bigl(s^n_{t+1}, a^n_{t+1}\bigr) &\mbox{if } X_{t+1}^n = 1,\\ 
\mbox{Choose uniformly over } \, \mathcal U & \mbox{otherwise.}
\end{cases}
\label{eq:epsilongreedy}
\end{equation}
Note that Assumption \ref{ass:algorithm}(v) is clearly satisfied as each state is visited with probability at least $\varepsilon$.
We also choose our stepsize sequences with $\tilde{\gamma}_t^{n-1}$ and $\tilde{\eta}_t^{n-1}$ as \emph{deterministic harmonic sequences} (for ease of analysis) so that
\begin{equation}
\gamma_t^n(s,a) = \frac{\gamma_t}{n} \, \mathbf{1}_{\{(s,a) = (S_t^n,a_t^n) \}} \quad \mbox{and} \quad \eta_t^n(s,a) = \frac{\eta_t}{n} \, \mathbf{1}_{\{(s,a) = (S_t^n,a_t^n) \}},
\label{eq:harmonic_step}
\end{equation}
where $\gamma_t > 0$ and $\eta_t > 0$ are deterministic, time-dependent constants. Since neither $\tilde{\gamma}_t^{n-1} = \gamma_t^n/n$ nor $\tilde{\eta}_t^n = \eta_t^n/n$ depend on the history of visited states, they are known as \emph{centralized learning rates} (see \cite{Szepesvari1996} for a discussion). 
The main difficulty for the centralized stepsizes of (\ref{eq:harmonic_step}) is that when the frequency of visits to a state decays quickly enough, then Assumption \ref{ass:algorithm}(iii) and (iv) may not hold (for example, consider when a state is only visited on iteration numbers that are powers of two). In fact, it is often not immediately obvious when Assumption \ref{ass:algorithm}(iii) and (iv) are satisfied. The next proposition shows that under the $\varepsilon$-greedy sampling policy, the states are visited often enough that the assumption remains satisfied.
\begin{restatable}{proposition}{propepsgreedy}
Under the $\varepsilon$-greedy sampling policy given in (\ref{eq:epsilongreedy}) and the deterministic harmonic stepsizes given in (\ref{eq:harmonic_step}), Assumption \ref{ass:algorithm}(iii)--(iv) is satisfied.
\label{prop:sampling_step}
\end{restatable}
\begin{proof}
See Appendix \ref{sec:appendix}.
\end{proof}

For a technical reason needed to prove the convergence rate results, we need to constrain the iterates of the algorithm to a region \emph{within} the support of the distribution of future costs (stated formally in Assumption \ref{ass:strongconv} below). We shall see that when this assumption is satisfied, we get a notion of \emph{strong convexity}. To be more precise, observe that $\partial^2 \mathbf{E} [u +  (X-u)^+/(1-\alpha_i)  ]/\partial u^2 = f_X(u)/(1-\alpha)$, where $f_X$ is the density of $X$. If $u$ is constrained to be within the support of the distribution of $X$ (i.e., where $f_X > 0$), then we are able to lower bound the second derivative by some constant (that depends on the constraint set), thereby attaining strong convexity within the region. This is useful for deriving the convergence rate results (see Lemma \ref{lem:bound1}); unfortunately, this condition is in general difficult to check in practice.

\begin{assumption} The density $f_t(x\<|\<s,a) = dF_t(x\<|\<s,a)/dx $ exists and the stochastic approximation projection set is within the support of the density: $\mathcal X_t^{u}(s,a) \subseteq \textnormal{supp}\bigl(f_t(\,\cdot \<|\<s,a)\bigr)$ for all $(s,a) \in \mathcal U$ and $t \in \mathcal T$. Let
\[
l_f = \min_{(s,a,t)} \, \min \bigl \{ x \in \mathcal X_t^u(s,a) : f_t(x\<|\<s,a) \bigr\}
\]
be a positive real number that lower bounds the density function over all $t$ and $(s,a) \in \mathcal U$.
\label{ass:strongconv}
\end{assumption}

We now provide a few lemmas that will be useful in establishing the final result. The first lemma relates the error of $\bar{u}_t^{i,n}$ to the error in the last iteration (i.e., of $\bar{u}_t^{i,n-1}$) and the error of the value function in the next time period (i.e., of $\bar{Q}_{t+1}^{n-1}$).

\begin{lemma}
Under Assumptions \ref{ass:algorithm}(vi)--\ref{ass:strongconv}, the $\varepsilon$-greedy sampling policy of (\ref{eq:epsilongreedy}), and the deterministic harmonic stepsizes given in (\ref{eq:harmonic_step}), the sequence of approximations $\bar{u}_t^{i,n}$ generated by Algorithm \ref{alg:dq} satisfies, for any $\kappa > 0$,
\begin{align*}
\mathbf{E} \Bigl[ &\bigl\|\bar{u}_t^{i,n} - u_t^{i,*} \bigr  \|_2^2   \Bigr] \\
&\le \left[1- \frac{\gamma_t}{n} \left(  2\,\varepsilon \, C_{l_f} - \kappa \, C_{L_F}\right)\right] \, \mathbf{E} \Bigl[ \bigl\|\bar{u}_t^{i,n-1} - u_t^{i,*} \bigr \|_2^2   \Bigr] + \frac{\gamma_t}{n \, \kappa} \, \mathbf{E} \Bigl[ \bigl \|\bar{Q}_{t+1}^{n-1} - Q_{t+1}^* \bigr  \|_2^2   \Bigr] +C_{\alpha_i} \, \frac{\gamma_t^2}{n^2},
\end{align*}
where $C_{l_f} = \frac{l_f}{1-\alpha_i}$,  $C_{L_F} = \frac{L_F^2}{(1-\alpha_t)^2}$, and $C_{\alpha_i} =  \Bigl[\max \bigl( 1, \frac{\alpha_i}{1-\alpha_i}\bigr) \Bigr]^2$.
\label{lem:bound1}
\end{lemma}
\begin{proof}
Fix an $i$. Recall the update equation given in Step 5 of Algorithm \ref{alg:dq}:
\begin{equation*}
\bar{u}_t^{i,n} = \Pi_{\mathcal X_t^u} \Bigl \{ \bar{u}_t^{i,n-1} - \textnormal{diag}(\gamma_t^n) \, \psi^i_t \bigl(\bar{u}_t^{i,n-1},\bar{Q}_{t+1}^{n-1}, W_{t+1}^{u,n}\bigr)\Bigr\}.
\end{equation*}
Using the non-expansive property of the projection operator, we have
\begin{align}
\bigl \|\bar{u}_t^{i,n}-u_t^{i,*} \bigr\|_2^2 &= \left\| \Pi_{\mathcal X_t^u} \Bigl \{ \bar{u}_t^{i,n-1} - \diag(\gamma_t^{n}) \,  \psi^i_t \bigl(\bar{u}_t^{i,n-1},\bar{Q}_{t+1}^{n-1},W_{t+1}^{u,n}\bigr) \Bigr\} -  \Pi_{\mathcal X_t^u} \bigl \{ u_t^{i,*} \bigr\} \right \|_2^2\nonumber \\
&\le \bigl\| \bar{u}_t^{i,n-1} - u_t^{i,*} - \diag(\gamma_t^{n}) \,  \psi^i_t \bigl(\bar{u}_t^{i,n-1},\bar{Q}_{t+1}^{n-1},W_{t+1}^{u,n}\bigr) \bigr \|_2^2 \nonumber\\
&\le \begin{aligned}[t]\bigl\| \bar{u}_t^{i,n-1} &- u_t^{i,*}  \bigr\|_2^2 + C_{\alpha_i} \, \frac{\gamma_t^2}{n^2} \\
&- 2 \, (\bar{u}_t^{i,n-1} - u_t^{i,*})^\mathsf{T} \diag(\gamma_t^{n})\, \psi^i_t \bigl(\bar{u}_t^{i,n-1},\bar{Q}_{t+1}^{n-1},W_{t+1}^{u,n}\bigr).\end{aligned}\label{eq:crossterm}
\end{align}
Recall from the proof of Lemma \ref{lem:uconv} that we can write 
\[
\psi^i_t \bigl(\bar{u}_t^{i,n-1},\bar{Q}_{t+1}^{n-1},W_{t+1}^{u,n}\bigr) = \mathbf{E} \bigl[ \psi^i_t\bigl(\bar{u}^{i,n-1}_t,Q^*_{t+1},W_{t+1}\bigr)  \bigr]+  \epsilon_{t+1}^{i,n} + \xi_{t+1}^{i,n},
\]
from which we see that the cross-term $(\bar{u}_t^{i,n-1} - u_t^{i,*})^\mathsf{T} \diag(\gamma_t^{n})\, \psi^i_t \bigl(\bar{u}_t^{i,n-1},\bar{Q}_{t+1}^{n-1},W_{t+1}^{u,n}\bigr)$ in the chain of inequalities (\ref{eq:crossterm}) can be expanded into three terms. These terms can now be analyzed separately. By (\ref{eq:psi_def}) and Assumption \ref{ass:problem}(iv), we see that the derivative of $\mathbf{E} \bigl[ \psi^i_t\bigl(u,Q^*_{t+1},W_{t+1}\bigr)  \bigr]$ in $u(s,a)$ can be expressed as
\begin{equation}
 \frac{\partial}{\partial u(s,a)}\mathbf{E} \bigl[ \psi^i_t\bigl(u,Q^*_{t+1},W_{t+1}\bigr)  \bigr](s,a) = \frac{f_t(u(s,a)\<|\<s,a)}{1-\alpha_i} \ge \frac{l_f}{1-\alpha_i}.
 \label{eq:derivativebound}
\end{equation}
Since every state is visited with probability larger than $\varepsilon$, we know that $\mathbf{E}\bigl[  \gamma_t^n(s,a) \, \bigl| \, \mathcal G_{t+1}^{n-1} \bigl] \ge \frac{\varepsilon \gamma_t}{n}$. Combining this with (\ref{eq:derivativebound}) and $\mathbf{E} \bigl[ \psi^i_t\bigl(u^{*}_t,Q^*_{t+1},W_{t+1}\bigr) = 0$, it follows that (almost surely)
\begin{align}
\mathbf{E} \Bigl[ &(\bar{u}_t^{i,n-1} - u_t^{i,*})^\mathsf{T} \diag(\gamma_t^{n})\,\mathbf{E} \bigl[ \psi^i_t\bigl(\bar{u}^{i,n-1}_t,Q^*_{t+1},W_{t+1}\bigr)  \bigr] \, \bigl|  \, \mathcal G_{t+1}^{n-1} \Bigr]\nonumber \\
&\ge \frac{\varepsilon \gamma_t}{n} \, \bigl(\bar{u}_t^{i,n-1} - u_t^{i,*}\bigr)^\mathsf{T} \Bigl[\mathbf{E} \bigl[ \psi^i_t\bigl(\bar{u}^{i,n-1}_t,Q^*_{t+1},W_{t+1}\bigr)  \bigr] - \mathbf{E} \bigl[ \psi^i_t\bigl(u^{*}_t,Q^*_{t+1},W_{t+1}\bigr)  \bigr]\Bigr] \label{eq:le1}\\
&\ge \frac{\varepsilon \gamma_t}{n} \, \frac{l_f}{1-\alpha_i} \bigl \| \bar{u}_t^{i,n-1} - u_t^{i,*}  \bigr \|_2^2. \nonumber
\end{align}
Note that $\gamma_t^n(s,a)$ depends on $(S_t^n,a_t^n)$ and $\epsilon_{t+1}^{i,n}(s,a)$ depends on $W_{t+1}^{u,n}$. By independence of $(S_{t}^{n},a_{t}^{n})$ and $W_{t+1}^{u,n}$ and the unbiased property of $\epsilon_{t+1}^{i,n}(s,a)$ of (\ref{eq:unbiased}),
\begin{align}
\mathbf{E} \Bigl[ (\bar{u}_t^{i,n-1} - u_t^{i,*})^\mathsf{T} \diag(\gamma_t^{n})\, \epsilon_{t+1}^{i,n} \, \bigl|  \, \mathcal G_{t+1}^{n-1} \Bigr] =0.\label{eq:le2}
\end{align}
Moving on to the third term, using (\ref{eq:xibound}), the fact that $\gamma_t^n$ has exactly one nonzero component, and the monotonicity of $l_p$ norms, we can deduce
\begin{align*}
\mathbf{E} \Bigl[ - (\bar{u}_t^{i,n-1} - u_t^{i,*})^\mathsf{T} \diag(\gamma_t^{n})\, \xi_{t+1}^{i,n} \, \bigl|  \, \mathcal G_{t+1}^{n-1} \Bigr] 
&\le \frac{\gamma_t}{n} \, \bigl\|\bar{u}_t^{i,n-1} - u_t^{i,*} \bigr\|_\infty \, \frac{L_F}{1-\alpha_i}\,\bigl\|\bar{Q}_{t+1}^{n-1}-Q_{t+1}^*\bigr\|_\infty\nonumber\\
&\le \frac{\gamma_t}{n} \,  \frac{ L_F }{1-\alpha_i} \, \bigl\|\bar{u}_t^{i,n-1} - u_t^{i,*} \bigr\|_2 \, \bigl\|\bar{Q}_{t+1}^{n-1}-Q_{t+1}^*\bigr\|_2.
\end{align*}
Using the inequality $2ab \le a^2 \, \kappa+b^2/\kappa$ for $\kappa >0$ on the above, we arrive at
\begin{equation}
\begin{aligned}
\mathbf{E} \Bigl[ -2\, (\bar{u}_t^{i,n-1} - u_t^{i,*})^\mathsf{T} &\diag(\gamma_t^{n})\, \xi_{t+1}^{i,n} \, \bigl|  \, \mathcal G_{t+1}^{n-1} \Bigr] \\
&\le \frac{\gamma_t \,  L_F^2 \, \kappa}{n \, (1-\alpha_i)^2} \, \bigl\|\bar{u}_t^{i,n-1} - u_t^{i,*} \bigr\|_2^2 + \frac{\gamma_t}{n \, \kappa} \, \bigl\|\bar{Q}_{t+1}^{n-1}-Q_{t+1}^*\bigr\|_2^2.
\end{aligned}
\label{eq:le3}
\end{equation}
Finally, the statement of the lemma follows by taking expectations of the inequalities (\ref{eq:crossterm}), (\ref{eq:le1}), (\ref{eq:le2}), and (\ref{eq:le3}) and combining.
\end{proof}
Similarly, the next lemma relates the error of the approximate value function $\bar{Q}_t^n$ to the error in the last iteration (i.e., of $\bar{Q}_t^{n-1}$), the error of the value function in the next time period (i.e., of $\bar{Q}_{t+1}^{n-1}$), and the error of all of the approximate quantiles $\bar{u}_t^{i,n}$. Because the analysis is similar to that of Lemma \ref{lem:bound1}, we relegate the proof to Appendix \ref{sec:appendix}.
\begin{restatable}{lemma}{boundtwo}
 Under the same conditions as Lemma \ref{lem:bound1}, the sequence of approximations $\bar{Q}_t^n$ generated by Algorithm \ref{alg:dq} satisfies, for any $\kappa_0,\, \kappa_1, \ldots, \kappa_m>0$,
\begin{align*}
\mathbf{E} \Bigl[ \bigl\|\bar{Q}_t^{n} - Q_t^* \bigr \|_2^2   \Bigr]  \le &\biggl[1 - \frac{\eta_t}{n} \Bigl( 2\,\varepsilon - L_\Phi \, \kappa_0 - L_\Phi \sum_{i=1}^m \kappa_i\Bigr)\biggr] \, \mathbf{E} \Bigl[ \bigl\|\bar{Q}_t^{n-1} - Q_t^* \bigr \|_2^2   \Bigr]\\
&+ \frac{\eta_t \, L_\Phi}{n \, \kappa_0} \, \mathbf{E} \Bigl[ \bigl\|\bar{Q}_{t+1}^{n-1} - Q_{t+1}^*  \bigr\|_2^2 \Bigr] + \frac{\eta_t}{n} \,\sum_{i=1}^m \frac{L_\Phi}{\kappa_i} \, \mathbf{E} \Bigl[ \bigl\|\bar{u}_{t}^{i,n-1} - u_{t}^{i,*} \bigr \|_2^2  \Bigr] + C_{H} \, \frac{\eta_t^2}{n^2},
\end{align*}
where $C_H$ bounds the term
\[
\mathbf{E}\Bigl[\bar{Q}_t^{n-1}(s,a) - H_t\bigl(\bar{u}_t^{1,n-1},\ldots,\bar{u}_t^{m,n-1},\bar{Q}_{t+1}^{n-1},W_{t+1}\bigr)(s,a)\Bigr]^2 \le C_{H},
\]
for all $(s,a) \in \mathcal U$.
\label{lem:bound2}
\end{restatable}
\begin{proof}
See Appendix \ref{sec:appendix}.
\end{proof}

With these preliminary results in mind, we move on to the theorem that states our $\mathcal O(1/n)$ convergence rate and provide a sketch of the proof.


\begin{restatable}[Convergence Rate]{theorem}{thmconvrate}
Choose initial approximations $\bar{Q}^0 \in \mathbb R^D$ and $\bar{u}^{i,0} \in \mathbb R^D$ for each $i \in \{1,2,\ldots,m\}$ so that $\bar{Q}_t^0 \in \mathcal X_t^q$ and $\bar{u}_t^{i,0} \in \mathcal X_t^u$ for all $t \in \mathcal T$. Under Assumptions \ref{ass:algorithm}(vi)--\ref{ass:strongconv}, the $\varepsilon$-greedy sampling policy of (\ref{eq:epsilongreedy}), and the deterministic harmonic stepsizes given in (\ref{eq:harmonic_step}), the sequences of iterates $\bar{u}^{i,n}$ for $i \in \{1,\ldots,m\}$ and $\bar{Q}^n$ generated by Algorithm \ref{alg:dq} satisfy convergence rates of the form $\mathbf{E} \bigl[ \bigl\|\bar{u}^{i,n} - u^{i,*}  \bigr\|_2^2   \bigr] \le  \mathcal O\left(1/n\right)$ for $i \in \{1,\ldots,m\}$ and $\mathbf{E} \bigl[ \bigl\|\bar{Q}^{n} - Q^*  \bigr\|_2^2   \bigr] \le  \mathcal O\left(1/n\right)$.
\label{thm:convrate1}
\end{restatable}
\begin{proof}[Sketch of Proof.]
The proof is also by backwards induction on $t$, where the induction hypothesis at each step is that $\bar{Q}_{t+1}^n$ converges at a rate of $\mathcal O(1/n)$. Applying Lemma \ref{lem:bound1} along with some additional analysis, it is possible to show that $\bar{u}_t^{i,n}$ also converges at a rate of $\mathcal O(1/n)$. Lemma \ref{lem:bound2} then completes the proof by showing the desired rate for $\bar{Q}_{t}^n$. The details are given in Appendix \ref{sec:appendix}.
\end{proof}

Although both $\bar{u}^{i,n}$ and $\bar{Q}^{n}$ converge at a rate of $\mathcal O\left(1/n\right)$, the slower sequence is $\bar{Q}^{n}$. To see why this step is slower, one can compare Lemmas \ref{lem:bound1} and \ref{lem:bound2}. From Lemma \ref{lem:bound1}, we see that the error of the quantity $\bar{u}^{i,n}_t$ depends only on the ``last iteration error'' (the error of $\bar{u}^{i,n-1}_t$) and the error of the next stage value function $\bar{Q}_{t+1}^{n-1}$. In contrast, Lemma \ref{lem:bound2} shows that the error of $\bar{Q}_t^n$ also depends on the error of the current stage quantile $\bar{u}^{i,n-1}$, in addition to its own ``last iteration error'' and the error of the next stage value function. The interpretation of these bounds is that accuracy of $\bar{u}^{i,n-1}$ is needed before we obtain an accurate approximation of $\bar{Q}_t^n$, exactly as our intuition would suggest, given Figure \ref{fig:algidea}. Indeed, this is observed in empirical experiments (see Section \ref{sec:numerical}), motivating the second contribution of this paper, a procedure aimed toward speeding up the slow step of the ADP algorithm.

\section{Efficient Sampling of the ``Risky'' Region}
\label{sec:sampling}
Arguably, the biggest practical issue with a Monte Carlo-based algorithm in the setting of risk-averse decision making is the question of \emph{sampling}. To illustrate, suppose the one-step conditional risk measure is CVaR at a level of $\alpha = 0.99$. Because $\alpha$ is close to 1, the iterates of the Dynamic-QBRM ADP algorithm are volatile and exhibit poor empirical convergence rates, as shown in Figure \ref{fig:emprate}. In this section, we discuss a method to address this issue.


\begin{figure}[h]
        \centering
        \begin{subfigure}[b]{0.4\textwidth}
                \centering
                \includegraphics[width=\textwidth]{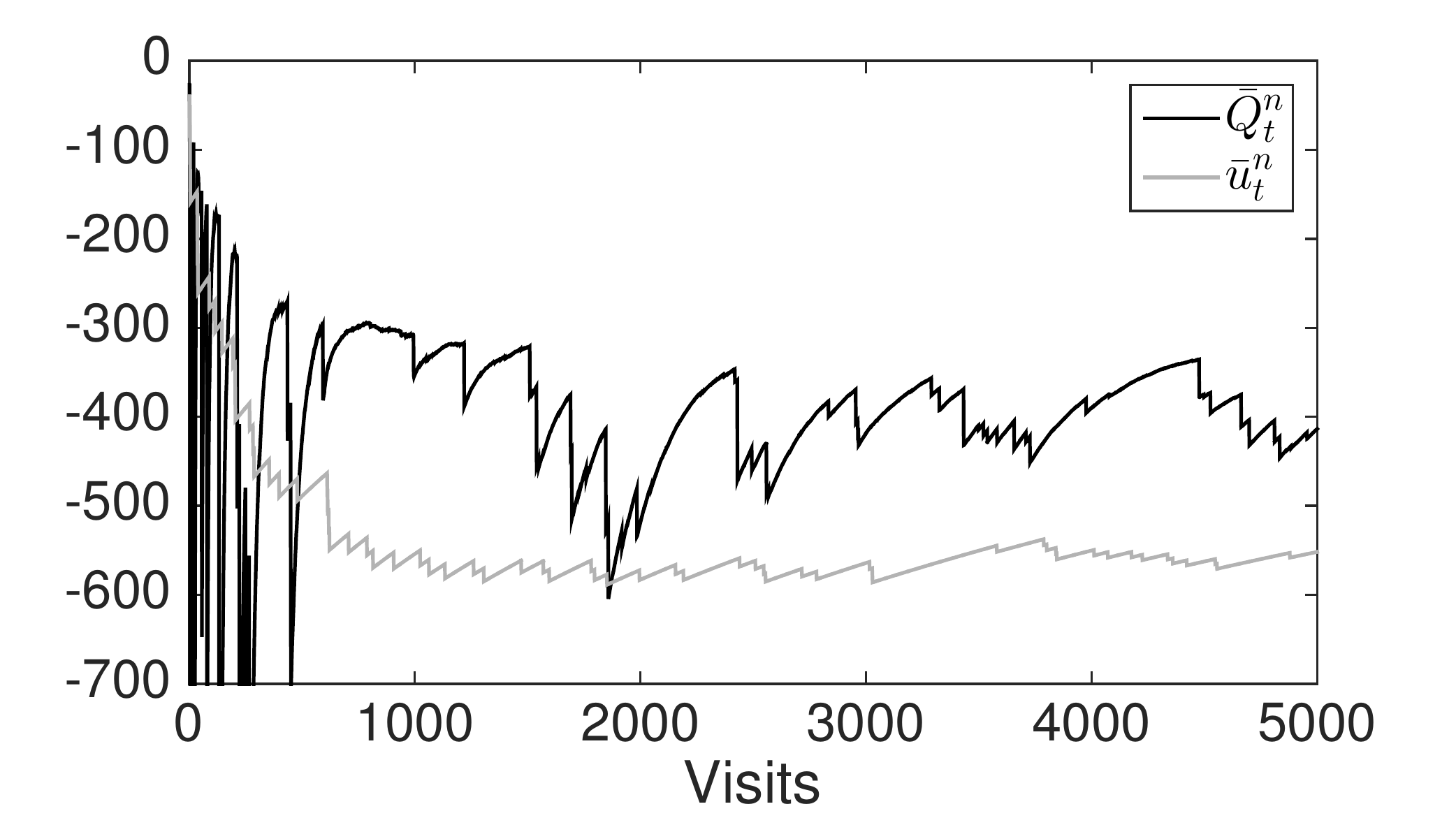}
        \end{subfigure}
        \begin{subfigure}[b]{0.4\textwidth}
                \centering
                \includegraphics[width=\textwidth]{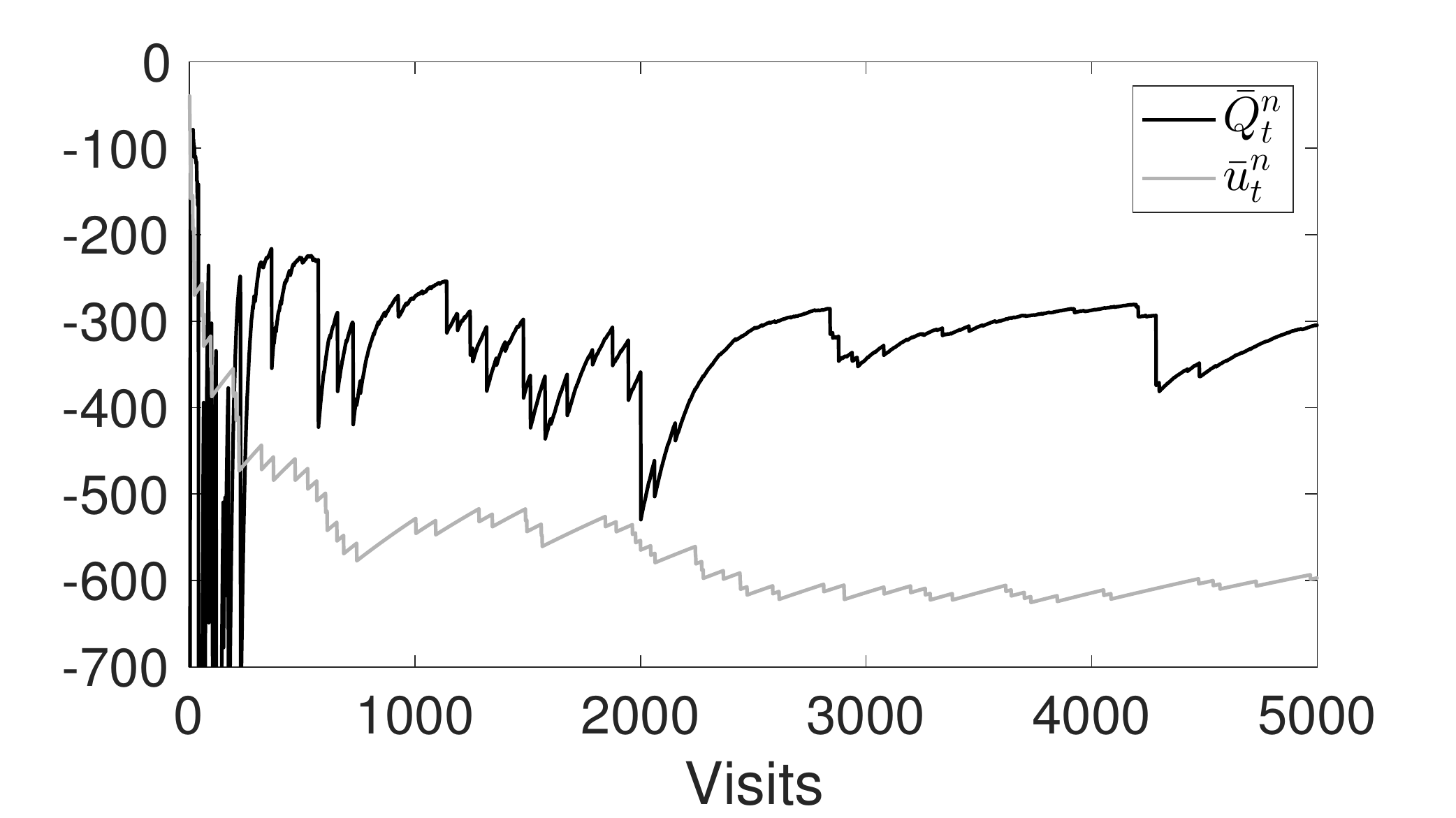}
        \end{subfigure}
        \caption{Sample Paths of Dynamic-QBRM ADP}
        \label{fig:emprate}
\end{figure}


\subsection{Overview of the Main Idea}
As we have mentioned, Dynamic-QBRM ADP can be applied in situations where the distribution of the information process $\{W_{t}\}$ is unknown, a common assumption for ADP algorithms. An example of when such a paradigm can be useful is when one has access to real data, but no stochastic model from which to simulate. Another example is a black-box simulator where the user has little to no control of its parameters. In these scenarios, a good remedy to any sampling issue is to implement an \emph{adaptive stepsize rule}, similar to the likes of \cite{George2006}, \cite{Schaul2013}, and \cite{Ryzhov2015} (the third reference develops a stepsize rule specifically in the context of ADP), that can adjust based on previously observed data points. However, if the distribution of the stochastic process $\{W_{t}\}_{t=0}^T$ is known, then we can propose a new companion procedure to control the sampling process as our ADP algorithm progresses. The procedure takes advantage of the idea of \emph{importance sampling} (see, e.g., \cite{Bucklew2004}) and is inspired by adaptive sampling approaches like the \emph{cross-entropy method} of \cite{Rubinstein1999}. 

\begin{figure}[h]
	\centering
	\includegraphics[width=\textwidth]{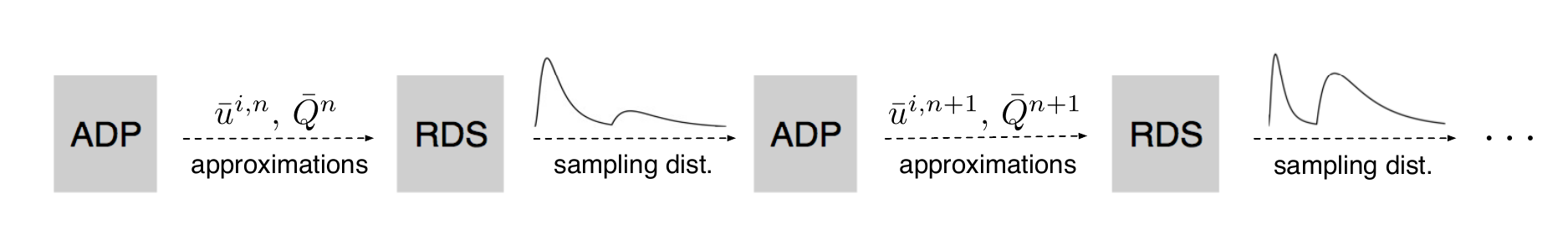}
	\caption{RDS Algorithm Idea}
	\label{fig:rdsalgidea}
\end{figure}

Recall from the results of the previous section that the convergence of $\bar{Q}_{t}^{n}$, the approximation of the value function, is expected to be slower than that of $\bar{u}_t^{i,n}$, the auxiliary variable. 
In this section, we propose a procedure called \emph{risk-directed sampling} (RDS) to improve the sampling efficiency for the step of Dynamic-QBRM ADP where $\bar{Q}_{t}^{n}$ is updated, i.e., Step 6 of Algorithm \ref{alg:dq}. The main idea, as illustrated in Figure \ref{fig:rdsalgidea}, is to use the ADP approximations to drive the learning of the sampling distribution and simultaneously use the sampling distribution to generate the samples for updating the ADP approximations.
As the figure shows, the hope is that our procedure updates the distribution in such a way that we dedicate samples to the regions of high risk (from where we may otherwise not see many samples). 


\subsection{Risk-Directed Sampling} 
Suppose the distribution of the exogenous information $W_{t+1}$ has a density $p_t(w)$. Notice that by (\ref{eq:Qustar}), we have for any $(s,a)$ and $t$,
\begin{equation}
Q^*_{t}(s,a) = \int H_t \bigl(u_t^{1,*}, u_t^{2,*}, \ldots, u_t^{m,*}, Q_{t+1}^{*}, w \bigr)(s,a) \, p_t(w) \, dw.
\label{eq:q_integral}
\end{equation}
For convenience, we use the shorthand notation
\[
H_t^{*}(w\<|\<s,a) = H_t \bigl(u_t^{1,*}, u_t^{2,*}, \ldots, u_t^{m,*}, Q_{t+1}^{*},\, w \, \bigr)(s,a),
\]
to emphasize the variable of integration, $w$. From the principle of importance sampling (see, e.g., \cite{Bucklew2004}), it is known that to produce a low-variance estimate of $Q_t^*(s,a)$ using Monte Carlo sampling, one should sample from a distribution whose density is nearly proportional to the absolute value of the integrand of (\ref{eq:q_integral}).

Our approach takes advantage of this proportionality property of the optimal density and directly constructs an approximation to the absolute value of the integrand of (\ref{eq:q_integral}) by minimizing a \emph{mean squared error}. The idea is to capture the ``risky regions,'' i.e., the parts of the outcome space where the integrand is large. When the parametric class of sampling distributions is chosen to be a mixture class (as we do), this introduces a simplification by allowing us to effectively remove the constraint $\|\theta\|_1 = 1$ from the optimization problem. In addition, we prove that our method converges without exact knowledge of the function $H_t^{*}(w\<|\<s,a)$. Instead, the algorithm works in conjunction with Dynamic-QBRM ADP by using approximations defined by
\[
H_t^{n}(w\<|\<s,a) = H_t \bigl(\bar{u}_t^{1,n-1}, \bar{u}_t^{2,n-1}, \ldots, \bar{u}_t^{m,n-1}, \bar{Q}_{t+1}^{n-1}, w\bigr)(s,a),
\]
where $\bar{u}_t^{1,n-1}, \bar{u}_t^{2,n-1}, \ldots, \bar{u}_t^{m,n-1}, \bar{Q}_{t+1}^{n-1}$ are outputs from Dynamic-QBRM ADP. Like the main ADP method, the procedure is fully adaptive and updates to the sampling distribution are made at every iteration. To our knowledge, an adaptive importance sampling approach for constructing risk-averse policies has not been considered in the literature. However, \cite{Kozmik2014} employs importance sampling from a different perspective: for the \emph{evaluation} of risk-averse policies in stochastic programming.

Let $\{\phi_t^k\}_{k=1}^K$ be the set of densities for $K$ prespecified basis distributions (we also refer to these as \emph{basis functions}) for time $t$, from which we create a mixture distribution used for sampling. Note that it is often sufficient to specify one set of distributions for all $t$, as we do in Section \ref{sec:numerical}.
Our goal is to take the traditional regression approach and develop an approximation of the form $\sum_{k} \bar{\theta}_t^{k,n}(s,a) \, \phi_t^k(w) \approx \bigl|  H_t^{*}(w\<|\<s,a) \bigr |\, p_t(w)$, where $\bar{\theta}_t^{k,n}(s,a)$ is the weight of the $k$-th basis density at iteration $n$ and state $(s,a)$.
For convenience, we also use the notation $\bar{\theta}_t^n(s,a)$ to be a vector in $\mathbb R^K$ with components $\bar{\theta}_t^{k,n}(s,a)$ for $k=1$ to $K$. Choose a (large) compact subset $\widebar{\mathcal W} \subseteq \mathcal W$
and define a random variable $W^u$ that is uniformly distributed over $\widebar{\mathcal W}$, i.e., assume that its density is given by the function $p^u(w) = C_{\widebar{\mathcal W}} \, \mathbf{1}_{\{w \in \widebar{\mathcal W}\}}$ for some constant $C_{\widebar{\mathcal W}}$. In what follows, we measure the mean squared error using this distribution, though of course, other reasonable choices may exist and the algorithm can be adapted accordingly. For a function $F:\mathcal W \rightarrow \mathbb R$ and $\phi = (\phi^1, \phi^2,\ldots,\phi^K)^\mathsf{T}$, define the projection operator 
\begin{equation}
\Pi_{\phi} F = \argmin_{\theta \ge 0} \mathbf{E} \Bigl[ \bigl[ \theta^\mathsf{T} \phi(W^u) - F(W^u)    \bigr]^2 \Bigr],
\label{eq:projection}
\end{equation}
which maps $F$ to a nonnegative coefficient vector $\theta$ corresponding to the best, i.e., minimum mean squared error, approximation under $\phi$.
\begin{proposition}
Under the condition that $ \mathbf{E} \bigl[ \phi(W^u) \, \phi(W^u)^\mathsf{T} \bigr]$ is positive definite, the optimization problem of (\ref{eq:projection}) has a unique solution.
\label{prop:uniquesol}
\end{proposition}
\begin{proof}
If the positive definiteness condition holds, then the problem is a strictly convex quadratic program (see, e.g., \cite{Boyd2004}).
\end{proof}
Exactly in the spirit of importance sampling, we define the optimal sampling density (within the mixture class) to be the probability density function corresponding to the best fit of the integrand using the given basis functions, i.e., $[ \Pi_\phi [\, |H_t^{*}(\,\cdot\,|\<s,a)| \; p_t(\<\cdot\<) ]]^\mathsf{T} \phi(w),$ which is computed by normalizing the weight vector so that it has unit norm.
The sampling density we use at time $t$, iteration $n$, and state $(s,a)$ and the corresponding likelihood ratio are given by
\[
\bar{p}_t^{n-1}(w\<|\<s,a) \propto \bigl[\bar{\theta}_t^{n-1}(s,a)\bigr]^\mathsf{T} \phi(w) \quad \mbox{and} \quad L_{t+1}^n(s,a) = p_t\bigl( W_{t+1}^{q,n} \bigr) \, \bar{p}_t^{n-1}\bigl(W_{t+1}^{q,n} \< \bigl|\<s,a \bigr)^{-1},
\]
where $L_{t+1}^n \in \mathbb R^d$ is defined for each $t$ and $n$ and $W_{t+1}^{q,n} \in \mathcal W$ is the random sample generated by from $\bar{p}_t^{n-1}(w\<|\<s,a)$. Assume that if $\bar{\theta}_t^{n-1}(s,a) = \mathbf{0}$ (the zero vector), then we set $\bar{p}_t^{n-1}(w\<|\<s,a) \propto \sum_k \phi^k(w)$. Finally, to update our coefficients $\bar{\theta}_t^{n}(s,a)$ from one iteration to the next, we need to introduce another $\mathcal G_t^n$-measurable stepsize sequence $\{\beta_t^n\}$ with $\beta_t^n \in \mathbb R^d$ for each $t$. The description of the new algorithm is given in Algorithm \ref{alg:dq_rds} below. 

Besides the additional input parameters, there are three points of departure from Algorithm \ref{alg:dq}. In Step 3, the information process sample $W_{t+1}^{q,n}$ is drawn according to $\bar{p}_t^{n-1}(w\<|\<s,a)$ rather than the true distribution, $p_t(w)$. In Step 5, we include an additional factor corresponding to the likelihood ratio $L_{t+1}^n$ in order to apply importance sampling. Step 7 is new and represents the updating step for the sampling coefficients $\bar{\theta}_t^n(s,a)$ using a stochastic approximation step. Before moving on to the convergence result, we need another stepsize assumption and an assumption on the basis functions. In addition, Lemma \ref{lem:betaxi_sum} verifies a technical condition that we need for the convergence result of Theorem \ref{thm:asconv_rds}.

\IncMargin{1em}
\begin{algorithm}[h]
\small
  \SetKwInput{Input}{Input}
  \SetKwInput{Output}{Output}
  \DontPrintSemicolon
\Indm  
  \Input{Initial estimates of the value function $\bar{Q}^0 \in \mathbb R^D$ and quantiles $\bar{u}^{i,0} \in \mathbb R^D$ for $i \in \{1,2,\ldots,m\}$.}\vspace{3pt}
  \myinput{Basis distributions $\phi^k$ and initial sampling coefficients $\bar{\theta}^{0}(s,a) \in \mathbb R^K$.}\vspace{3pt}
  \myinput{Stepsize rules $\gamma_t^n$, $\eta_t^n$, and $\beta_t^n$ for all $n$, $t$.}
  \BlankLine
  \Output{Approximations $\{\bar{Q}^n\}$ and $\{\bar{u}^{i,n}\}$.}
\Indp
  \BlankLine
  \nl Set $\bar{Q}_T^n = 0$ for all $n$.\;
  \For{$n = 1, 2, \ldots$}{
  \nl  Choose an initial state $(S_0^n, a_0^n)$.\;
  \For{$t = 0, 1, \ldots, T-1$}{
  \nl  Draw a sample of the information process $W_{t+1}^{u,n}$ from the distribution of $W_{t+1}$. Draw an IS sample $W_{t+1}^{q,n}$ so that $W_{t+1}^{q,n}(s,a) \sim \bar{p}_t^{n-1}(w\<|\<s,a)$.
  	\BlankLine
\nl  Update auxiliary variable approximations for $i =1, \ldots m$: \\
    $\quad \quad \quad \quad \bar{u}_t^{i,n} = \Pi_{\mathcal X_t^u} \Bigl \{ \bar{u}_t^{i,n-1} - \textnormal{diag}(\gamma_t^n) \, \psi^i_t \bigl(\bar{u}_t^{i,n-1},\bar{Q}_{t+1}^{n-1}, W_{t+1}^{u,n}\bigr)\Bigr\}.$\;
  \BlankLine

  \nl  Compute an estimate of the future cost based on the current approximations: \\
   $\quad \quad \quad \quad \hat{q}_t^n = \textnormal{diag}(L_{t+1}^n) \, \Bigl[ H_t\bigl(\bar{u}^{1,n-1}_t,\ldots,\bar{u}^{m,n-1}_t, \bar{Q}_{t+1}^{n-1}, W_{t+1}^{q,n}   \bigr) \Bigr]$.\;
  \BlankLine
 \nl  Update approximation of value function: \\
    $\quad \quad \quad \quad \bar{Q}_t^n = \Pi_{\mathcal X_t^q} \Bigl \{\bar{Q}_t^{n-1} - \textnormal{diag}(\eta_t^n) \, \bigl(\bar{Q}_t^{n-1}-\hat{q}_t^n\bigr) \Bigr \}.$\;
  \BlankLine

 \nl Update the sampling coefficients for each state. Let $w = W_{t+1}^{q,n}(s,a)$ and \\
    $\quad \quad \quad \quad \begin{aligned}\bar{\theta}_t^n(s,a) = \Bigl[\bar{\theta}_t^{n-1}(s,a) - &\beta_t^n(s,a)  \, \Bigl[ \bigl(\bar{\theta}_t^{n-1}(s,a)\bigr)^\mathsf{T} \, \phi(w)\\
     &- \bigl|H_t^{n}(w\<|\<s,a)\bigr| \, p_t(w)  \Bigr] \, \phi(w) \, p^u(w) \, \bar{p}_t^{n-1}(w\<|\<s,a)^{-1} \Bigr]^+,\end{aligned}$\;
	where $[\,\cdot\,]^+$ is taken componentwise.
  \BlankLine
 \nl  Choose next state $(S_{t+1}^n, a_{t+1}^n)$.
  }
  }
    \caption{Dynamic-QBRM ADP with Risk-Directed Sampling}
    \label{alg:dq_rds}
\end{algorithm}
\DecMargin{1em}

\begin{assumption} 
For all $(s,a) \in \mathcal U$ and $t \in \mathcal T$, suppose $\beta_t^n$ is $\mathcal G_t^n$-measurable and
\begin{enumerate}[label=(\roman*),labelindent=1in]
\vspace{0.5em}
\item $\beta_t^n(s,a) = \tilde{\beta}_t^{n-1} \, \mathbf{1}_{\{(s,a) = (S_t^n,a_t^n)\}}$, for some $\tilde{\beta}_t^{n-1}\in \mathbb R$ that is $\mathcal G_t^{n-1}$-measurable.
\vspace{0.5em}
\item $\displaystyle \sum_{n=1}^\infty \beta_t^n(s,a) = \infty, \quad \sum_{n=1}^\infty \beta_t^n(s,a)^2 < \infty  \quad a.s.$
\vspace{0.5em}
\item $\displaystyle \mathbf{E} \bigl[(\tilde{\beta}_t^{n-1})^2\bigr] \le \mathcal O(n^{-1-\epsilon})$ for some $\epsilon > 0$.
\end{enumerate}
\label{ass:stepbeta}
\end{assumption}

\begin{assumption}
With regard to the basis distributions $\phi^k$, $k \in \{1,2,\ldots,K\}$, the following hold:
\begin{enumerate}[label=(\roman*),labelindent=1in]
\item $ \mathbf{E} \bigl[ \phi(W^u) \, \phi(W^u)^\mathsf{T} \bigr]$ is a positive definite matrix (so that we may apply Proposition \ref{prop:uniquesol}),
\vspace{0.5em}
\item Let $\hat{q}_t^n \in \mathbb R^d$ be defined as in Step 5 of Algorithm \ref{alg:dq_rds}. There exists a constant $C_\phi$ such that the following holds for any $t$ and $(s,a)$: given $\theta \ge 0$ and $\|\theta\|_1 = 1$, the tails of the distribution $\sum_k \theta^k \phi^k$ are ``heavy'' enough to guarantee that if $W_{t+1}^{q,n}\sim \sum_k \theta^k \phi^k$, then $\mathbf{E}\bigl[\hat{q}_t^n(s,a)^2\bigr] \le C_\phi$.
\end{enumerate}
\label{ass:phisupport}
\end{assumption}

We remark that due to the compactness of the sets $\mathcal X_t^u$ and $\mathcal X_t^q$, many of the terms in the definition of $\hat{q}_t^n$ are bounded; therefore, the crucial terms that affect Assumption \ref{ass:phisupport}(ii) are $L_{t+1}^n$ and the cost function $c_t$. The condition that $ \mathbf{E} \bigl[\beta_t^{n}(s,a)^2\bigr] \le \mathcal O(n^{-1-\epsilon})$ is not particularly difficult to satisfy; for example, our deterministic harmonic stepsizes satisfy the condition with $\epsilon = 1$. In addition, polynomial rules of the form $\beta_t^n(s,a) = n^{-1/2-\epsilon/2} \, \mathbf{1}_{\{(s,a) = (S_t^n,a_t^n) \}}$ work as well.  We state the convergence result for Algorithm $\ref{alg:dq_rds}$ in Theorem $\ref{thm:asconv_rds}$. Since Step 7 of Algorithm \ref{alg:dq_rds} does not project to a compact set, we cannot make use of \cite[Theorem 2.4]{Kushner2003} as we did before. Instead, our convergence result is derived from a theorem of \cite{Pflug1996} for stochastic approximation, which requires the result of Lemma \ref{lem:betaxi_sum}. The rest of the proof is mostly standard and is thus deferred to Appendix \ref{sec:appendix}.

Because the importance sampled stochastic processes are corrected for in expectation using the likelihood ratio $L^n_{t+1}$, we should not expect the rate of convergence to change. Indeed, we have the following theorem.

\begin{restatable}[Convergence Rate]{theorem}{thmconvrate2}
Choose initial approximations $\bar{Q}^0 \in \mathbb R^D$ and $\bar{u}^{i,0} \in \mathbb R^D$ for each $i \in \{1,2,\ldots,m\}$ so that $\bar{Q}_t^0 \in \mathcal X_t^q$ and $\bar{u}_t^{i,0} \in \mathcal X_t^u$ for all $t \in \mathcal T$. Under Assumptions \ref{ass:algorithm}(vi)--\ref{ass:phisupport}, the $\varepsilon$-greedy sampling policy of (\ref{eq:epsilongreedy}), and the deterministic harmonic stepsizes given in (\ref{eq:harmonic_step}), the sequences of iterates $\bar{u}^{i,n}$ for $i \in \{1,\ldots,m\}$ and $\bar{Q}^n$ generated by Algorithm \ref{alg:dq} satisfy convergence rates of the form $\mathbf{E} \bigl[ \bigl\|\bar{u}^{i,n} - u^{i,*}  \bigr\|_2^2   \bigr] \le  \mathcal O\left(1/n\right)$ for $i \in \{1,\ldots,m\}$ and $\mathbf{E} \bigl[ \bigl\|\bar{Q}^{n} - Q^*  \bigr\|_2^2   \bigr] \le  \mathcal O\left(1/n\right)$.
\label{thm:convrate2}
\end{restatable}
\begin{proof}
The proof is analogous to that of Theorem \ref{thm:convrate1}. We need versions of Lemma \ref{lem:bound1} and Lemma \ref{lem:bound2} for Algorithm \ref{alg:dq_rds}. The difference is that we need to deal with the term $L_{t+1}^n(s,a)$, which has expectation equal to 1.
\end{proof}

\begin{restatable}[Convergence of Risk-Directed Sampling Procedure]{theorem}{thmrdsconv}
Choose initial approximations $\bar{Q}^0 \in \mathbb R^D$ and $\bar{u}^{i,0} \in \mathbb R^D$ for each $i \in \{1,2,\ldots,m\}$ so that $\bar{Q}_t^0 \in \mathcal X_t^q$ and $\bar{u}_t^{i,0} \in \mathcal X_t^u$ for all $t \in \mathcal T$. Under Assumptions \ref{ass:algorithm}(vi)--\ref{ass:phisupport}, the $\varepsilon$-greedy sampling policy of (\ref{eq:epsilongreedy}), and the deterministic harmonic stepsizes given in (\ref{eq:harmonic_step}), Algorithm \ref{alg:dq_rds} generates a sequence of iterates $\bar{Q}^n$ that converge almost surely to the optimal value function $Q^*$. Moreover, the sampling coefficients $\bar{\theta}_t^n(s,a)$ converge to the optimal sampling coefficients under $\phi$:
\[
\bar{\theta}_t^n(s,a) \longrightarrow \Pi_\phi \Bigl[\,\bigl|H_t^{*}(\< \cdot \<|\<s,a) \bigr| \; p_t(\< \cdot \<)\, \Bigr] \quad a.s.
\]
for each $t$ and $(s,a) \in \mathcal U$.
\label{thm:asconv_rds}
\end{restatable}
\begin{proof}
See Appendix \ref{sec:appendix}.
\end{proof}

Not surprisingly, the new algorithm retains the same theoretical properties of the standard Dynamic-QBRM ADP without RDS. In addition, we now have a sampling density that converges to the optimal sampling density as the algorithm progresses (Theorem \ref{thm:asconv_rds}). What remains for us to explore are the empirical convergence rates of the two approaches.

\section{Numerical Results}
\label{sec:numerical}
The recent surge of interest in energy and sustainability has shown that when pertaining to the question of risk, the problem of optimal control of energy storage assets is an especially rich application domain. In this section, we illustrate our proposed ADP algorithm by way of a stylized energy trading and bidding problem in which both heavy tails and extreme events play a prominent role.
\subsection{Model}
We consider the problem of using energy storage to trade in the electricity market, i.e., \emph{energy arbitrage}, with the caveat that there is the possibility of some financial penalty when the amount of stored energy is low.  
For example, since storage for backup purposes is rarely in use, one might consider using it to generate a stream of revenue by interacting with the market. However, this immediately introduces a source of risk: in the rare event when backup is needed, it is crucial that there is enough energy to cover demand. Failure to do can cause complications, so it is useful to consider a risk-averse policy to this problem. See \cite{Xi2014} for a detailed model of the shared storage situation, but solved with a risk-neutral objective.
When no penalty is assessed, we assume that there is a modest reward. Finally, we introduce a bidding aspect to the problem where one must place bids prior to the desired transaction time (see \cite{Jiang2013a} for detailed model of hour-ahead bidding).

For $t=0$ to $t=T$, let $S_t \in \mathcal S = \{0,1,\ldots,S_\textnormal{max}\}$ be the amount of energy in storage and let $P_t \sim \log \mathcal N(\mu_P(t), \sigma_P(t)^2)$ be the (heavy-tailed) spot price of electricity. Also, let $U_t \sim \mathcal N(0,\sigma_U^2)$ be independent of $P_t$, and suppose our simple model of storage-based penalties is as follows. Given two constants $0<a < b$, where $a$ represents the ``rate of reward'' and $b$ represents the ``rate of penalty,'' assume that the reward/penalty assessed at time $t+1$ is given by
\[
F_{t+1} = \bigr| \mu_S(S_t) + U_{t+1} \bigl| \, \Bigl[ b \cdot \mathbf{1}_{\{\mu_S(S_t)+U_{t+1} < 0\}} - a \cdot \mathbf{1}_{\{\mu_S(S_t)+U_{t+1} \ge 0\}}  \Bigr],
\]
where $\mu_S : \mathcal S \rightarrow \mathbb R$ is a nondecreasing function, signifying that as $S_t$ increases, the possibility of penalty decreases (rare event). Our two-dimensional action at each time $t$ is given by 
\[
a_t = (b_t^-, b_t^+) \in \mathcal A  \subseteq \bigl\{(b^-, b^+): 0 \le b^- \le b^+ \le b_\textnormal{max} \bigr\},
\] where $|\mathcal A| < \infty$. There are no constraints, so $\mathcal A_s = \mathcal A$ for all $s \in \mathcal S$. We call $b_t^-$ the \emph{buy bid} and $b_t^+$ the \emph{sell bid}: if $P_t$ fluctuates below the buy bid, we are obligated to buy from the market and if $P_t$ rises above the sell bid, we are obligated to sell to the market. In addition, we are penalized the amount of the spot price if we are to sell but the storage device is empty, i.e., $S_t=0$. In this problem, our information process $W_t$ is given by the pair $(P_t,U_t)$ and is independent of the past. We find it most natural to model this problem in the sense of maximizing revenues (or ``contributions'') rather than minimizing costs; hence, the \emph{contribution function} is
\[
c_t(S_t,a_t,W_{t+1}) = -F_{t+1} + P_{t+1} \,\Bigl[  \mathbf{1}_{\{ b_t^+ < P_{t+1} \}} - \mathbf{1}_{\{ b_t^- > P_{t+1} \}} - \mathbf{1}_{\{S_t=0\}}\,\mathbf{1}_{\{ b_t^+ < P_{t+1} \}} \Bigr],
\]
and the transition function is given by $S_{t+1} = [ \min \{ S_t + \mathbf{1}_{\{ b_t^- > P_{t+1} \}} - \mathbf{1}_{\{ b_t^+ < P_{t+1} \}}, \, S_\textnormal{max}  \} ]^+$.
The objective is to optimize the risk-averse model given in (\ref{eq:rmdp}) with the ``$\min$'' operator replaced by a ``$\max$'' and taking the QBRM to be mean-CVaR at $\alpha = 0.99$ (this refers to the lower 0.01 tail of the distribution since we are now in the setting of rewards, not costs). Various values of $\lambda$ are considered.
\subsection{Parameter Choices}
We let $T=12$ and the size of the storage device be $S_\textnormal{max} = 6$. Next, suppose the values of $\mu_S$ are chosen so that $\mathbf{P}\bigl(\mu_S(S_t)+U_{t+1} < 0\bigr)$ are 0.1, 0.05, 0.02, 0.01, 0.01, 0.001, and 0.001 for $S_t = 0, 1,\ldots,S_\textnormal{max}$ (representing increasingly rare penalties) with $\sigma_U^2=1$. The penalty parameters are set to be $a=5$ and $b=500$, making the penalty events relatively severe. We denote the (seasonal) mean of the spot price process by $m(t) = \mathbf{E}(P_t) = 50\,\sin(4\pi t/T)+100$, with a constant variance $v=\textnormal{Var}(P_t) = 3000$. By the properties of the lognormal distribution, this leads to the parameters $\mu_P(t) = \log ( m(t) \, (1+v/m(t)^2)^{-1/2}  )$ and $\sigma_P(t) = \sqrt{\log (1+v/m(t)^2)}$.
The maximum bid is chosen to be $b_\textnormal{max} = 500$ and each dimension (the range from $0$ to $500$) is discretized into increments of 50, resulting in an MDP with approximately 7{,}000 states.

Finally, we discuss our choice of parameters for the risk-directed sampling (RDS) procedure. For $\phi^1$, we simply take the true distribution of $W_{t+1}$ and for the remaining $\phi^k$, we select a set of bivariate normal distributions placed in a grid, with mean of the first component taking values in $\{50,175,300\}$ and the mean of the second component taking values in $\{-3,-1,1\}$. The standard deviations are chosen to be 750 and 0.25 for the respective components. This is a fairly general choice of basis functions that uses very little problem specific information and/or tuning. Figure \ref{subfig:basis500} shows an example of the shape of the sampling density after 500 iterations of RDS for a fixed state of $(S_t,b_t^-, b_t^+) = (0,150,300)$ at $t=T-1$ with $\lambda = 0.5$ (note: the $z$-axis is not normalized here and the range of the $P_t$ axis has been decreased to focus on the nonzero areas). As we would expect, the algorithm has chosen to allocate a relatively large sampling effort toward small values of $U_t$, presumably due to the fact that these events are very costly with $b=500$ and thus critical to the estimation of CVaR. The last few observations are shown as red points.
\begin{figure}[h]
        \centering
        \begin{subfigure}[b]{0.45\textwidth}
                \centering
                \includegraphics[width=\textwidth]{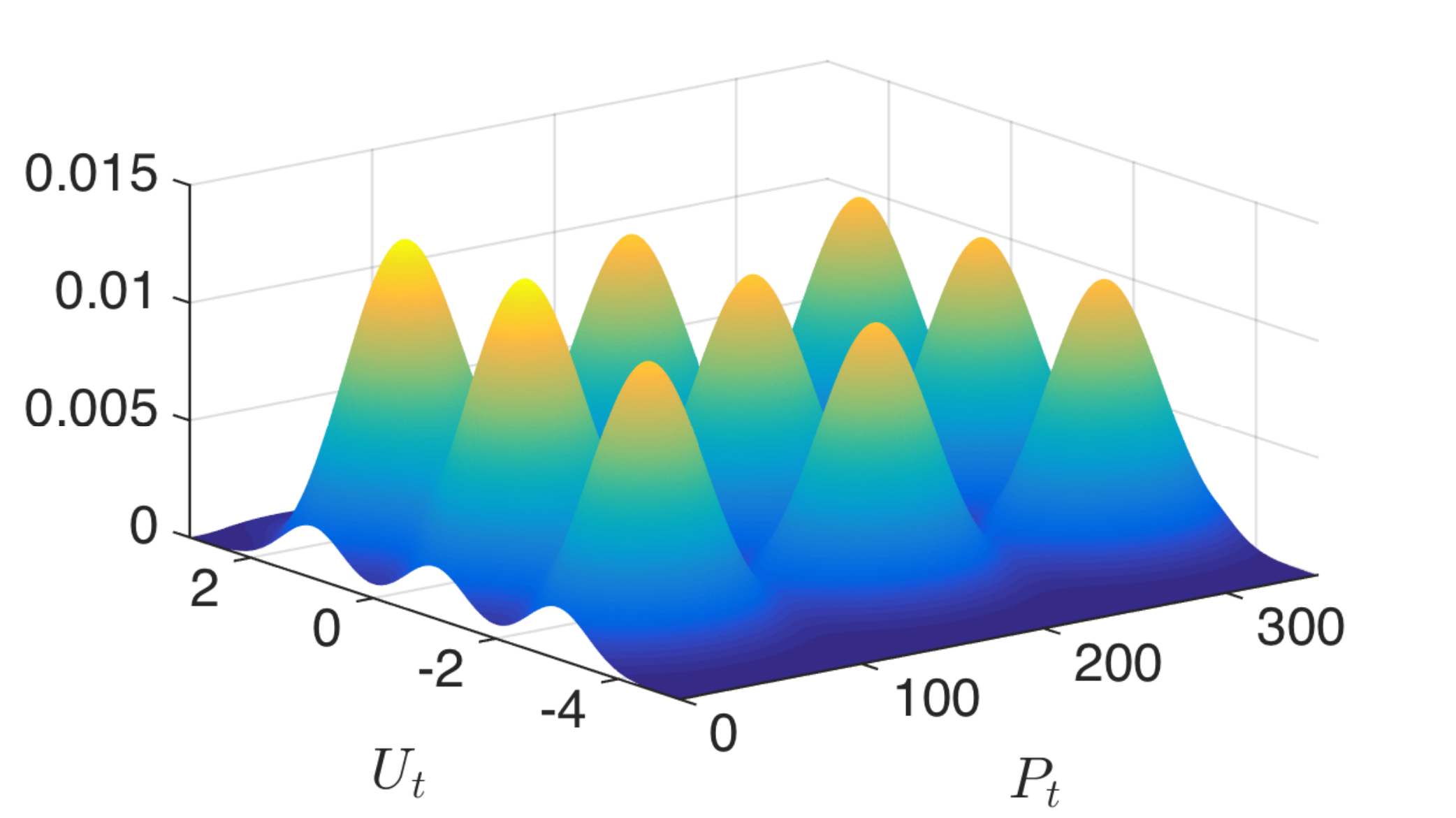}
                \caption{Basis distributions $\phi^k$ (equally weighted)}
        \label{subfig:basis}
        \end{subfigure}
        \begin{subfigure}[b]{0.45\textwidth}
                \centering
                \includegraphics[width=\textwidth]{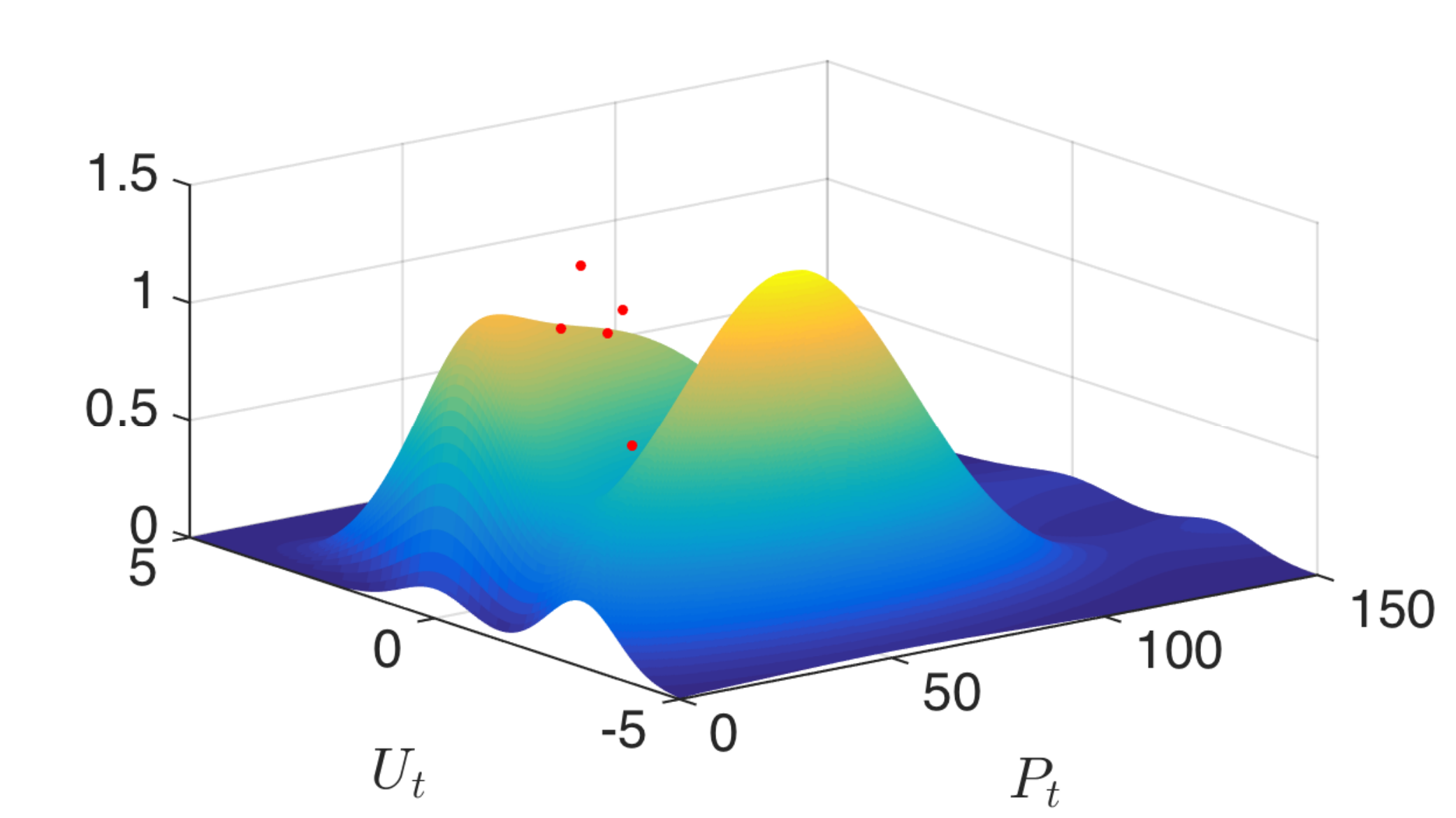}
                \caption{Shape of sampling density, $n=500$}
                \label{subfig:basis500}
        \end{subfigure}
        \caption{Example Illustration of Risk-Directed Sampling ($\lambda = 0.5$)}
        \label{fig:rds_illustration}
\end{figure}

\subsection{Evaluating Policies}
Here, we describe the procedure for evaluating policies under the objective (\ref{eq:rmdp}). Let $\{A_0^\pi, A_1^\pi, \ldots, A_{T-1}^\pi\}$ be a policy; to compute (\ref{eq:rmdp}), let $S_{t+1}^\pi = S^M(s,A_t^\pi(s), W_{t+1})$ and we solve the recursive equations given by
\begin{align}
V_t^\pi(s) &= \rho^\alpha_t \bigl(c_t(s, A_t^\pi(s), W_{t+1}) + V_{t+1}^\pi(S^\pi_{t+1}) \bigr) \textnormal{ for all } s \in \mathcal S, \; t \in \mathcal T,\label{eq:pibellman}\\
V_T^\pi(s) &= 0 \mbox{ for all } s \in \mathcal S.\nonumber
\end{align}
The value of the policy is then given by $V_0^\pi(S_0)$, but since $W_{t}$ is continuously distributed, we cannot, in general, solve these equations exactly. For this reason, we select a large, finite sample $\hat{\Omega} \subseteq \Omega$, and apply the standard \emph{sample average approximation} (SAA) technique of \cite{Kleywegt2002} to the Bellman recursion (\ref{eq:pibellman}) along with the standard linear programming method suggested in \cite{Rockafellar2000} for computing conditional value at risk. Next, in order to have an optimality benchmark against which ADP policies can be compared, we also apply the SAA technique to the Bellman recursion for the \emph{optimal} policy specified in Theorem \ref{thm:bellman} (with $\rho_t = \rho_t^\alpha$ and the ``$\min$'' replaced by ``$\max$'').


We use $50{,}000$ realizations of $W_t=(P_t,U_t)$ for each $t=1,2,\ldots,T$ in our simulations below. For the purposes of our numerical work and Figure \ref{fig:RDSconvergence}, we refer to the SAA optimal policy (which has value function $V_{0,\beta}^*(S_0)$) as our benchmark for ``100\% optimal'' and a corresponding SAA \emph{myopic policy}, denoted by $\pi_\text{m}$, (i.e., the policy which takes the action to maximize $\rho_t^\alpha(c_t(s,a,W_{t+1}))$ with zero continuation value to be ``0\% optimal.'' More precisely, we have
\[
\text{\% optimality of policy $\pi$ (with respect to myopic policy $\pi_\text{m}$)} = \frac{V_{0}^\pi(S_0) - V_{0}^{\pi_\text{m}}(S_0)}{ V_{0}^*(S_0) - V_{0}^{\pi_\text{m}}(S_0)}.
\]
Recall that larger $V^\pi_0(S_0)$ is better, as we have switched to the maximization setting for the numerical application.

\subsection{Results}
Let us first show a few plots that illustrate the effectiveness of the RDS procedure. Once again, we fix the state to be $(S_t,b^-_t,b_t^+) = (0,150,300)$ at $t=T-1$ and set $\lambda = 0.5$, but now we focus on the evolution of the approximations $\bar{Q}_t^n$ and $\bar{u}_t^n$; representative sample paths are shown in Figure \ref{fig:rds_illustration2} (the first plot is from Figure \ref{fig:emprate}, repeated for comparison purposes).

To make the comparison as fair as possible, the random number generator seed is set so that the sequence of observations used to update $\bar{u}_t^n$ is the same in both cases (i.e., the gray lines are identical). The true limit points are approximately $u^*_t(S_t,a_t) \approx -555$ and $Q_t^*(S_t,a_t) \approx -387$, and we notice that in the case of Figure \ref{subfig:yes_is}, a decent approximation of $Q_t^*$ is obtained around visit 2500, while the approximation in Figure \ref{subfig:no_is} does not settle until the end. This drastic difference in empirical convergence rate can have a significant impact on the performance of Dynamic-QBRM ADP as the large errors shown in Figure \ref{subfig:no_is} are propagated backwards in time.

\begin{figure}[h]
        \centering
        \begin{subfigure}[b]{0.43\textwidth}
                \centering
                \includegraphics[width=\textwidth]{FIGS/no_is_line_1_4_7-eps-converted-to.pdf}
                \caption{Without RDS}
        \label{subfig:no_is}
        \end{subfigure}
        \begin{subfigure}[b]{0.43\textwidth}
                \centering
                \includegraphics[width=\textwidth]{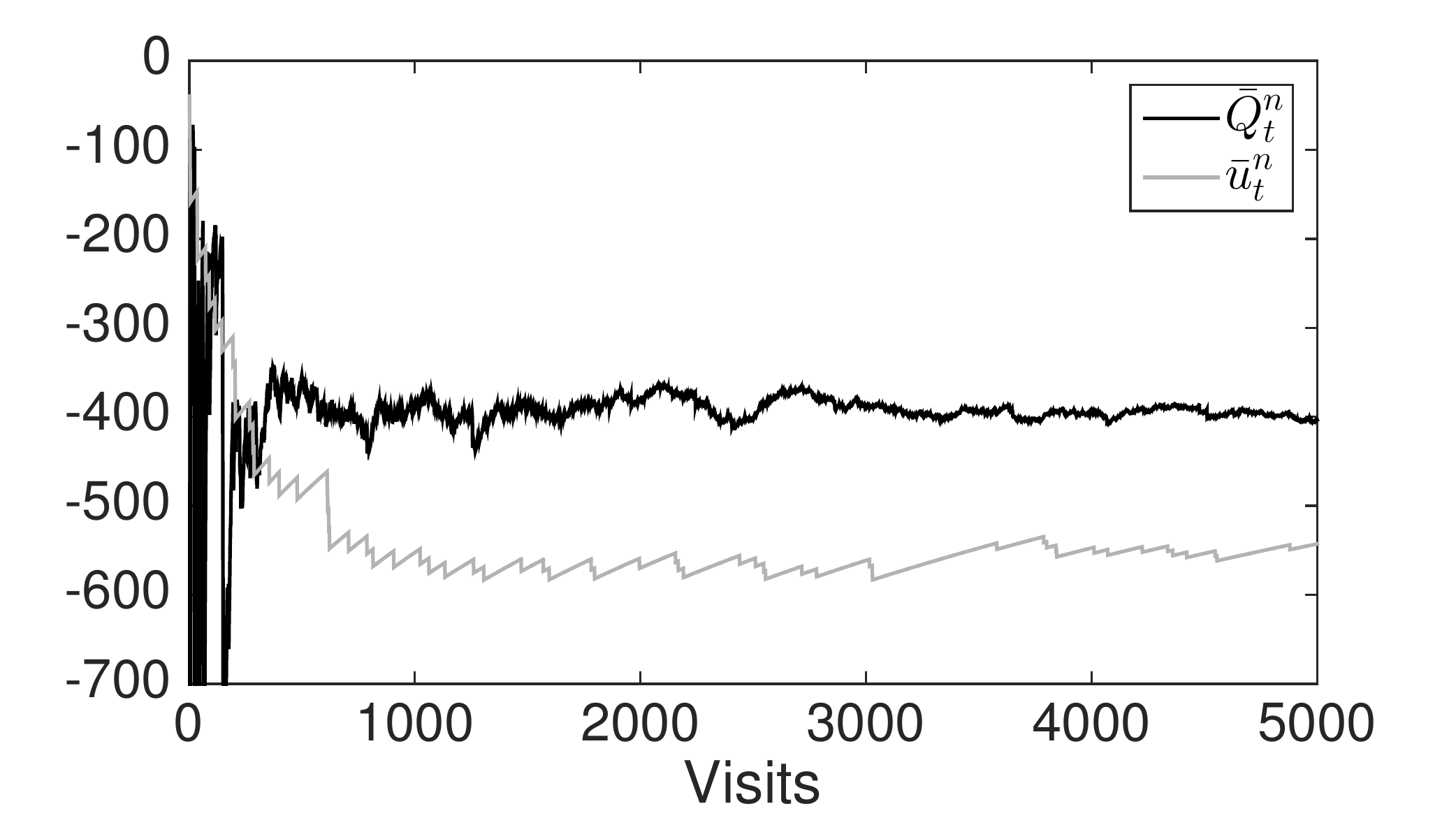}
                \caption{With RDS}
                \label{subfig:yes_is}
        \end{subfigure}
        \caption{Sample Paths of Approximations Generated by Dynamic-QBRM ADP ($\lambda = 0.5$)}
        \label{fig:rds_illustration2}
\end{figure}

In fact, we can illustrate this impact by examining the value function approximations back at time $t=0$, with and without RDS, compared to the SAA optimal value function. Figure \ref{fig:surfs} compares the two methods after $N=5{,}000{,}000$ iterations by varying the bid dimensions of a fixed state $S_0=0$. Notice that in Figure \ref{subfig:surfrds}, the approximation from the RDS procedure very closely resembles the optimal value function of Figure \ref{subfig:surfopt} even after errors are propagated through $T=12$ steps. The same cannot be said of the approximation in Figure \ref{subfig:surfnords}.

\begin{figure}[h]
        \centering
        \begin{subfigure}[b]{0.31\textwidth}
                \centering
                \includegraphics[width=\textwidth]{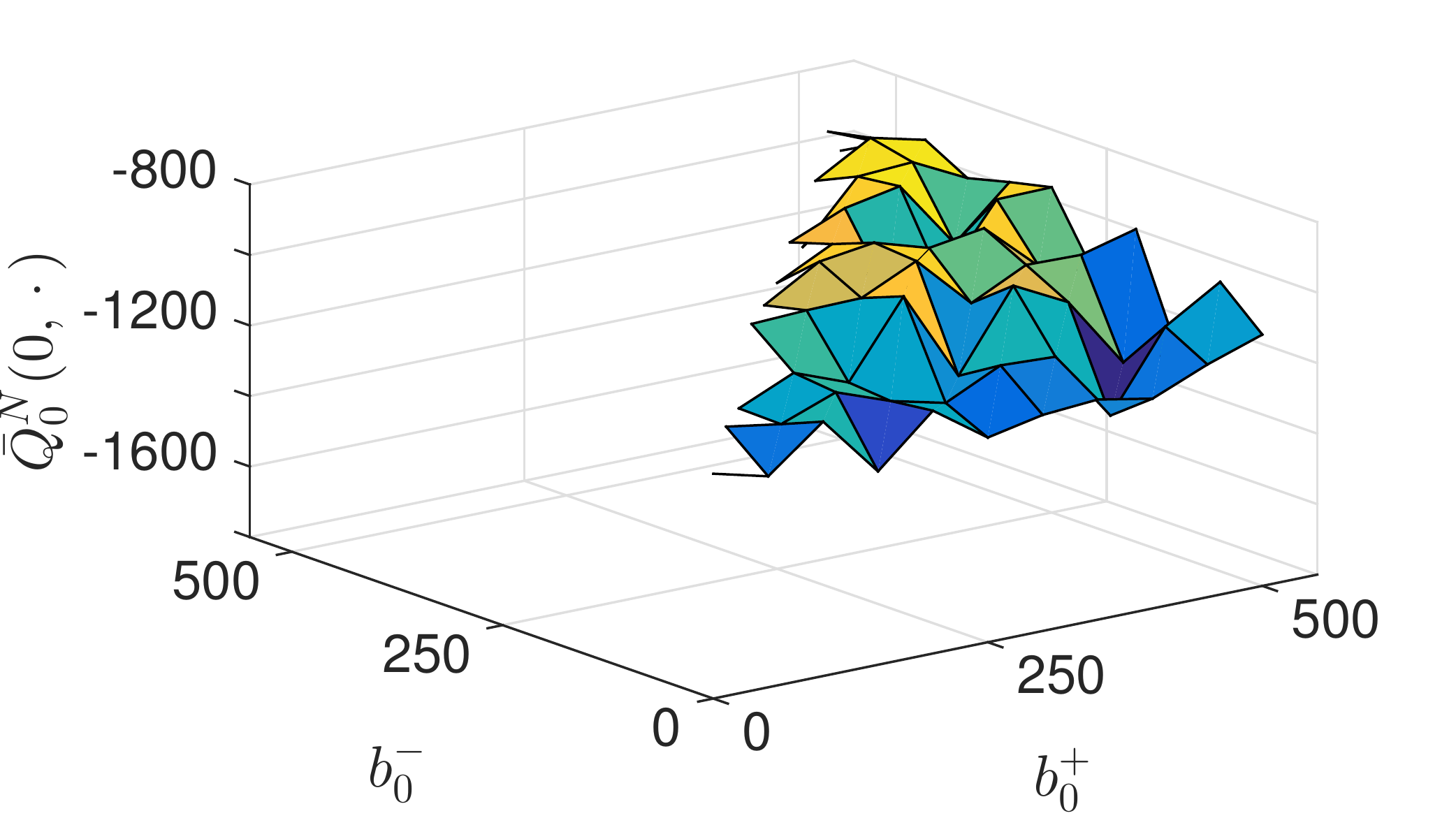}
                \caption{Without RDS}
                \label{subfig:surfnords}
        \end{subfigure}
        \begin{subfigure}[b]{0.31\textwidth}
                \centering
                \includegraphics[width=\textwidth]{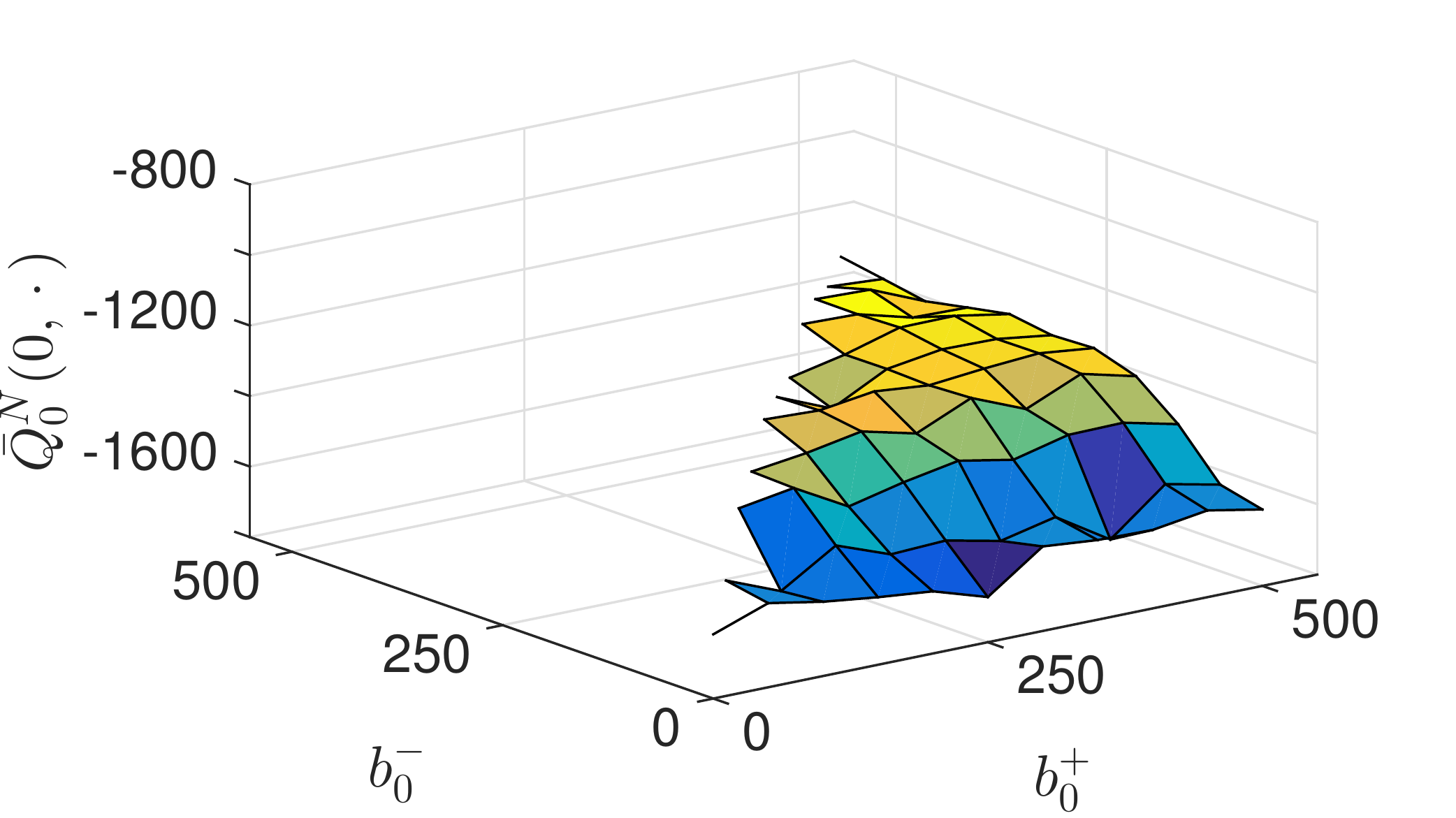}
                \caption{With RDS}
                \label{subfig:surfrds}
        \end{subfigure}
         \begin{subfigure}[b]{0.31\textwidth}
                \centering
                \includegraphics[width=\textwidth]{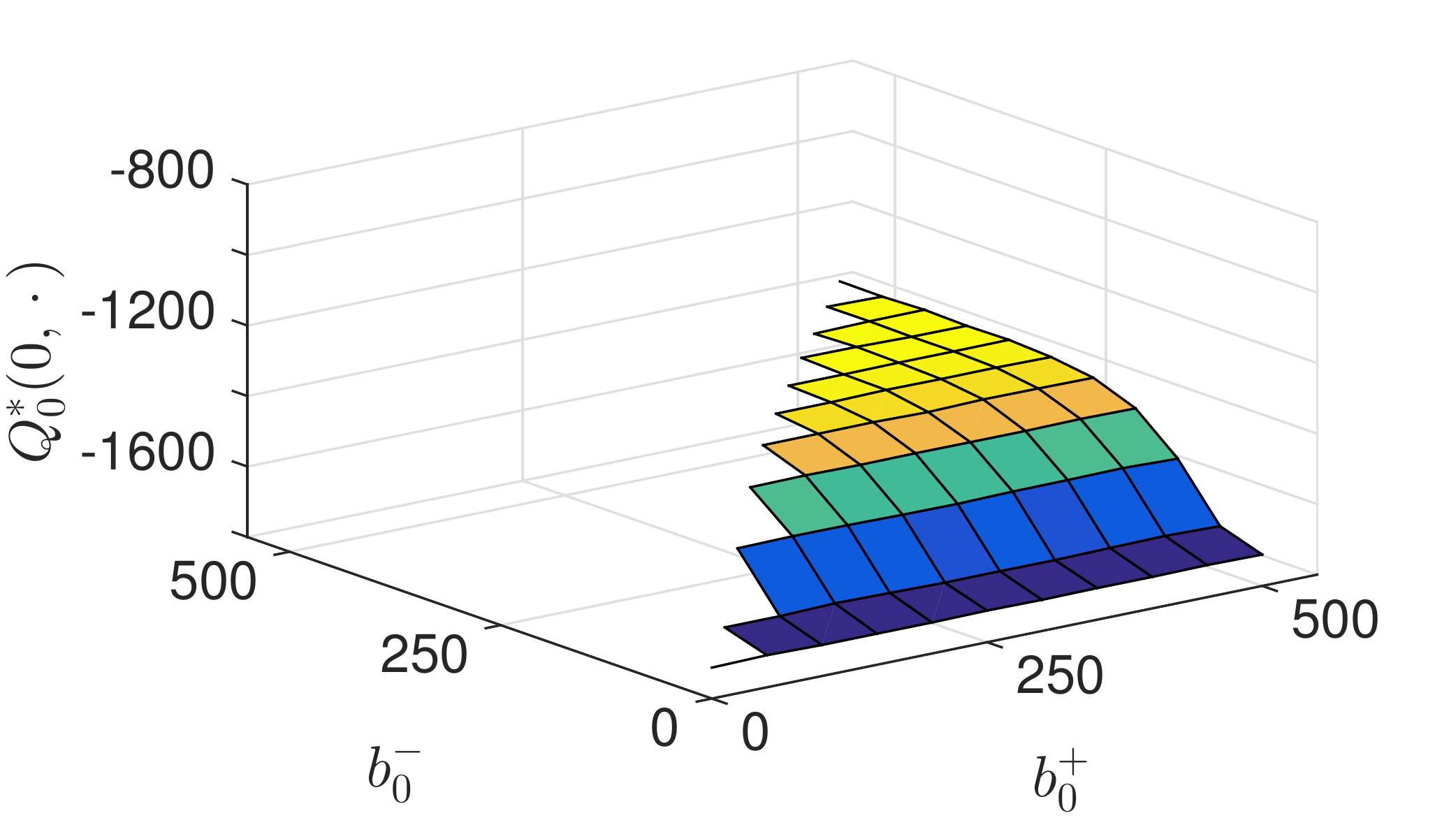}
                \caption{SAA Optimal}
                \label{subfig:surfopt}
        \end{subfigure}
        \caption{Surface Plots of Value Function Approximations at $t=0$ ($\lambda = 0.5$)}
        \label{fig:surfs}
\end{figure}

Lastly, we illustrate that in addition to the improved empirical convergence rates observed in Figures \ref{fig:rds_illustration2} and \ref{fig:surfs}, running Dynamic-QBRM ADP with RDS has a noticeable effect on the resulting policies as well. Using the backward recursive SAA evaluation procedure described above, we plot optimality percentages of risk-averse policies produced by the two variants of the algorithm for $\lambda \in \{0.6, 0.55, 0.5, 0.45, 0.4\}$ in Figure \ref{fig:RDSconvergence}. We see that RDS has the advantage in all cases, especially during the early iterations. Figure \ref{fig:RDSconvergence} uses a log-scale to display the early progress of the algorithms every 50,000 iterations from $N=0$ to $N=1{,}000{,}000$ and also the asymptotic progress every 1,000,000 iterations from $N=1{,}000{,}000$ to $N=5{,}000{,}000$. Our simulations also suggest that as $\lambda$ decreases, the advantage of using RDS diminishes; this is not surprising because decreasing $\lambda$ corresponds to deemphasizing the ``risky'' events that RDS seeks to sample. In fact, we find that in this particular example problem, the algorithms are practically indistinguishable for $\lambda \le 0.3$. Computation times for the three methods are given in Table \ref{table:cputimes}; notice that we are able to run the ADP methods for 5 million iterations in roughly half of the time it takes to compute the SAA optimal solution (the large number of scenarios presents a challenge for the optimization routine implemented in CPLEX). The online companion contains additional numerical results of evaluating our policies on more ``practical'' metrics of risk and reward, rather than the dynamic risk measure objective function. We observe that the policies generated by our algorithms behave in an intuitively appealing way.

\begin{figure}[h]
        \centering
        \begin{subfigure}[b]{0.32\textwidth}
                \centering
                \includegraphics[width=\textwidth]{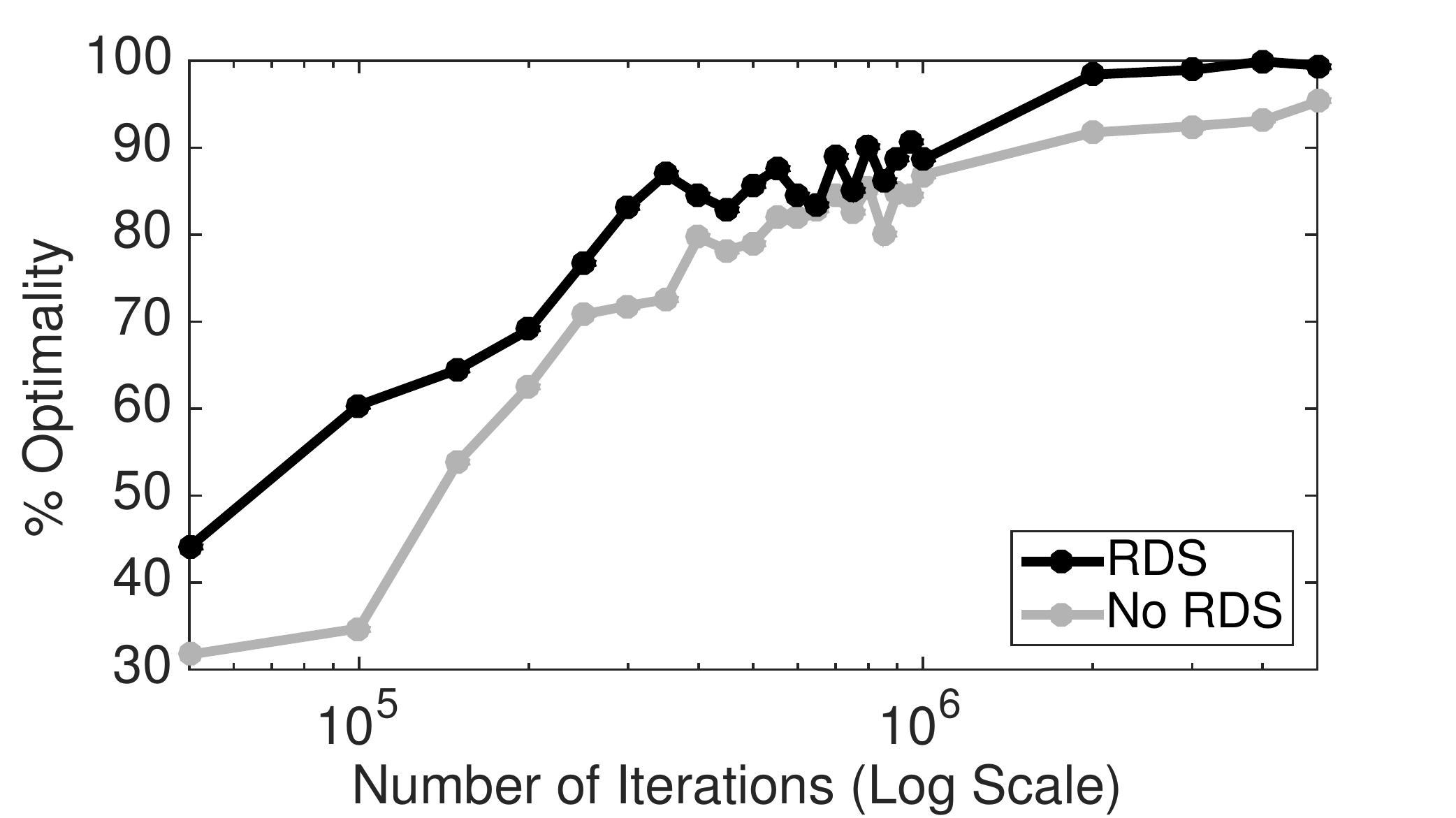}
                \caption{$\lambda = 0.6$}
        \end{subfigure}
        \begin{subfigure}[b]{0.32\textwidth}
                \centering
                \includegraphics[width=\textwidth]{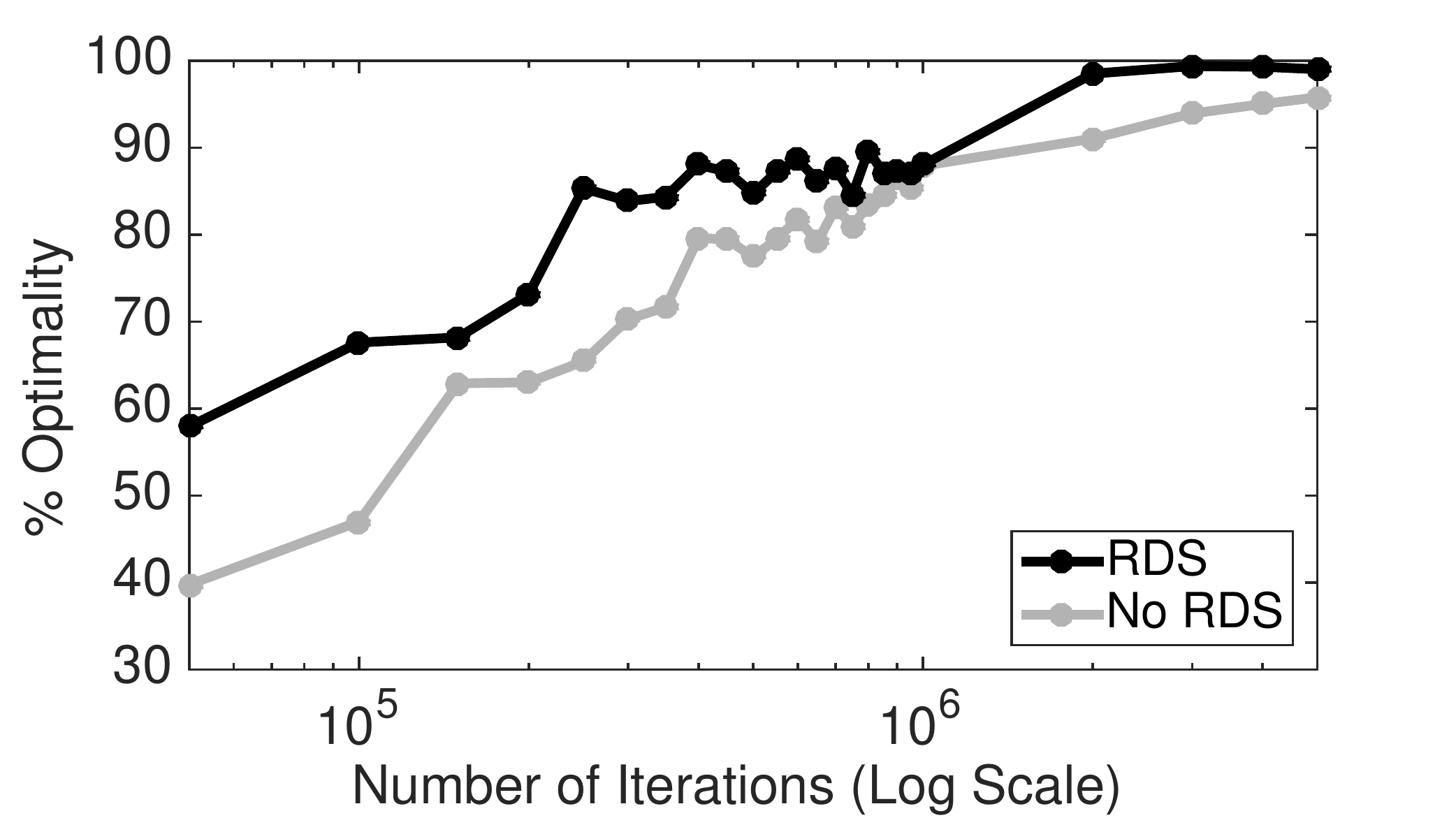}
                \caption{$\lambda = 0.55$}
        \end{subfigure}
        \begin{subfigure}[b]{0.32\textwidth}
                \centering
                \includegraphics[width=\textwidth]{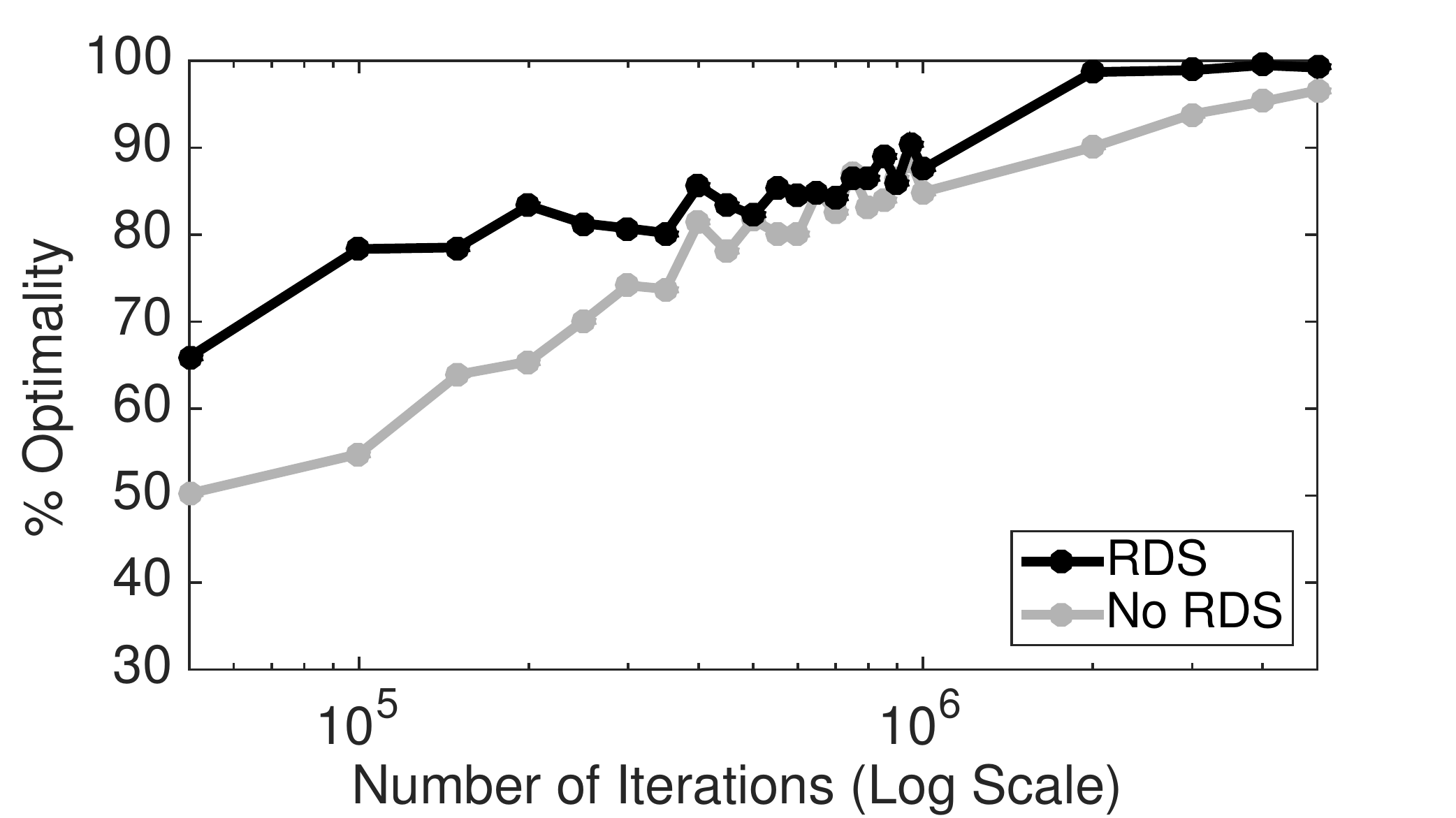}
        		\caption{$\lambda = 0.5$}
        \end{subfigure}\\
       
        \begin{subfigure}[b]{0.32\textwidth}
                \centering
                \includegraphics[width=\textwidth]{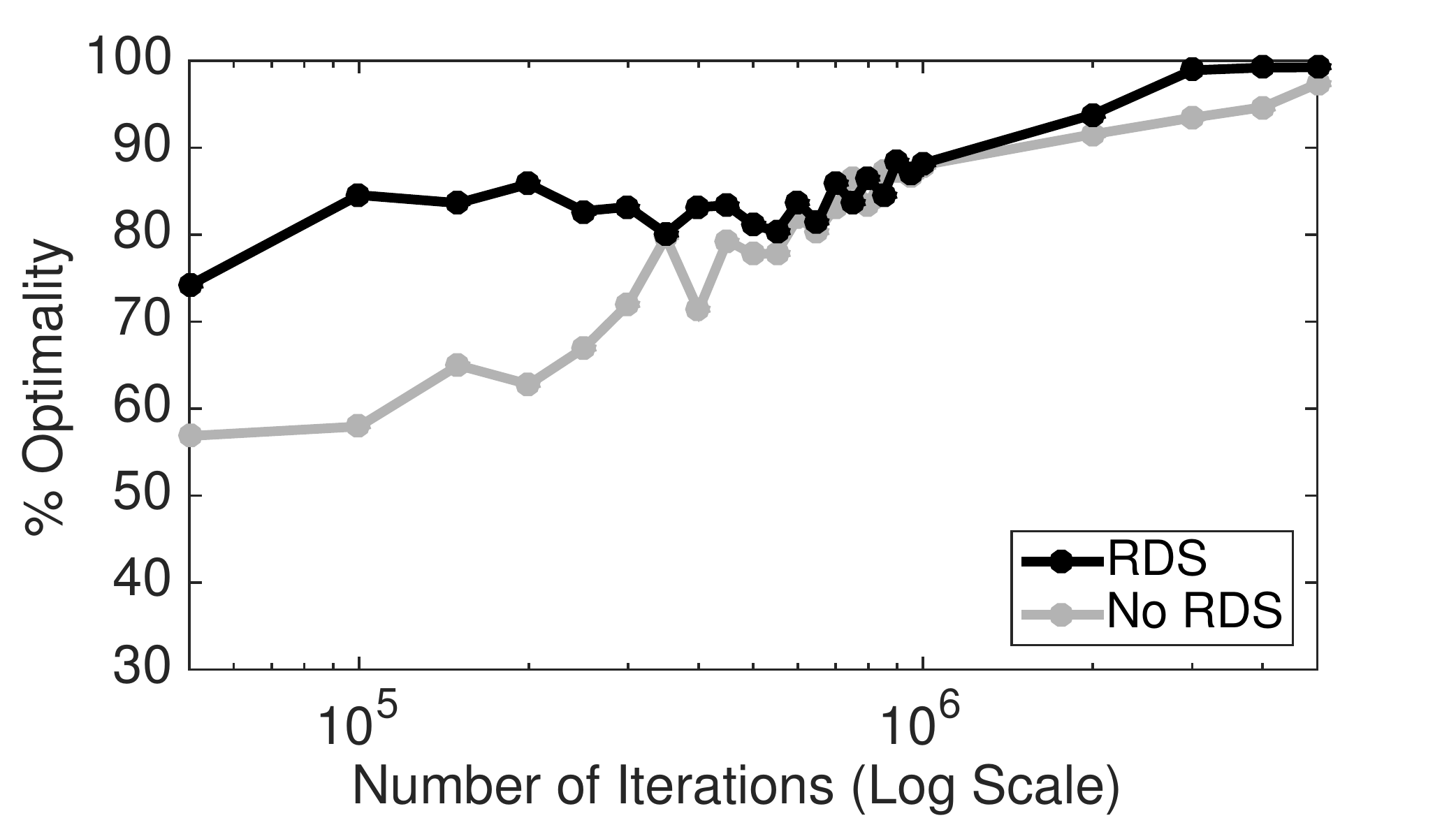}
                \caption{$\lambda = 0.45$}
        \end{subfigure}
        \begin{subfigure}[b]{0.32\textwidth}
                \centering
                \includegraphics[width=\textwidth]{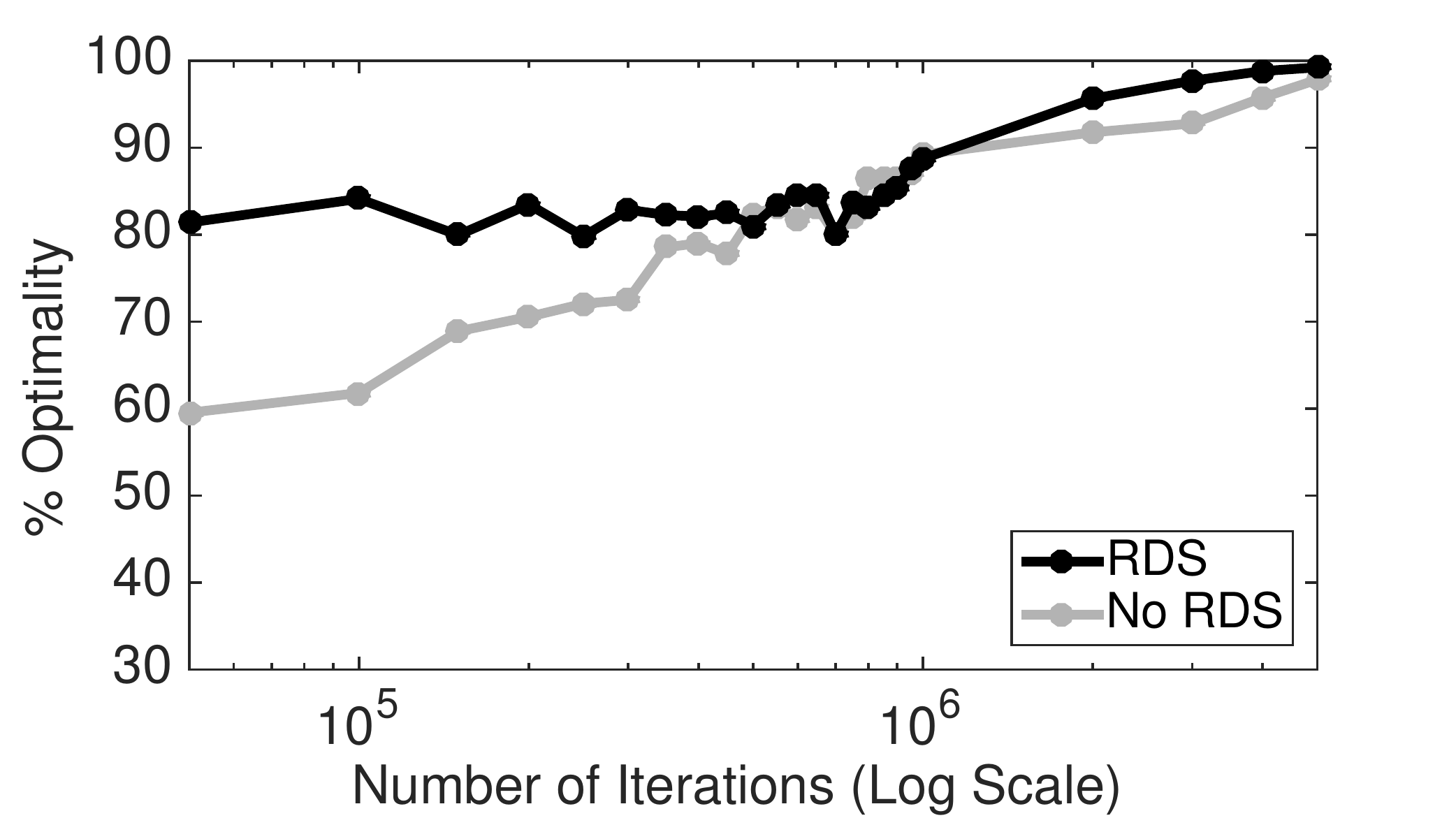}
                \caption{$\lambda = 0.4$}
        \end{subfigure}
        \caption{Comparison of Dynamic-QBRM ADP with and without RDS}
        \label{fig:RDSconvergence}
\end{figure}

\renewcommand{\arraystretch}{1.2}
\begin{table}[h]
\centering
\small
\begin{tabular}{@{}llllll@{}}\toprule
 Method & $\lambda = 0.6$\quad  & $\lambda=0.55$ \quad  & $\lambda=0.5$ \quad  & $\lambda=0.45$ \quad  & $\lambda=0.4$ \quad \\
 \midrule
Dynamic-QBRM ADP  & 810.24	& 819.40 & 816.40 & 830.69 & 811.45\\
Dynamic-QBRM ADP with RDS \quad \quad & 926.25 &	983.33 &	989.87 &	926.85	& 973.50 \\
SAA Backward Recursion & 2119.74 & 2058.36 & 2163.91 & 2118.84 & 2072.79\\
\bottomrule
\end{tabular}
\caption{CPU Times for $N=5,000,000$ Iterations of ADP and SAA Optimal Solution (minutes)}
\label{table:cputimes}
\end{table}

\section{Conclusion}
\label{sec:conclusion}
In this paper, we propose a new ADP algorithm for solving risk-averse MDPs \citep{Ruszczynski2010} under a class of quantile-based risk measures. The algorithm uses a two-step updating procedure to approximate the optimal value function and we prove that it converges almost surely with a rate of $\mathcal O(1/n)$. We also show a companion sampling procedure to more efficiently sample the ``risky'' region of the outcome space and then prove that the sampling distribution converges to one that is, in a sense, the best within a parametric class. Next, we describe an energy storage and bidding application, on which we demonstrate that the RDS sampling approach provides significant benefits in terms of the empirical convergence rate. Moreover, we illustrate that the approximate policies produced by Dynamic-QBRM ADP behave in an intuitive and appealing way in terms of the tradeoff between risk and reward (see online companion), suggesting that it can be readily applied to a variety of problem settings.

\clearpage
\appendix

\section{Proofs}
\label{sec:appendix}
\thmasconv*
\begin{proof}We induct backwards on $t$, starting with the base case $t=T$. Since $Q_T^n(s,a) = Q_T^*(s,a) = 0$ for all $n$ and all $(s,a)$, convergence trivially holds. The induction hypothesis is that $\bar{Q}_{t+1}^n(s,a) \rightarrow Q_{t+1}^*(s,a)$ almost surely, and we aim to show the same statement with $t$ replacing $t+1$.
As we did before, we analyze the statistical error and approximation error, $\epsilon_{t+1}^{q,n},\,\xi_{t+1}^{q,n} \in \mathbb R^d$, except they are now with respect to the stochastic Bellman operator $H_t$:
\begin{align*}
\epsilon_{t+1}^{q,n} &= Q_t^*-H_t \bigl(u^{1,*}_t,\ldots,u^{m,*}_t, Q_{t+1}^*,W_{t+1}^{q,n}\bigr),\\
\xi_{t+1}^{q,n} &= H_t \bigl(u^{1,*}_t,\ldots,u^{m,*}_t, Q_{t+1}^*,W_{t+1}^{q,n}\bigr)-H_t \bigl(\bar{u}^{1,n-1}_t,\ldots,\bar{u}^{m,n-1}_t, \bar{Q}_{t+1}^{n-1},W_{t+1}^{q,n}\bigr).
\label{eq:xihdef}
\end{align*}
Therefore, the update of Step 6 of Algorithm \ref{alg:dq} is equivalent to
\begin{equation}
\bar{Q}_t^n = \Pi_{\mathcal X_t^q} \Bigl \{\bar{Q}_t^{n-1} - \textnormal{diag}(\eta_t^n) \, \Bigl[\bar{Q}_t^{n-1}-Q_t^*+ \epsilon_{t+1}^{q,n} + \xi_{t+1}^{q,n} \Bigr] \Bigr \}.
\label{eq:Q_update}
\end{equation}
It is once again clear that $\mathbf{E} \bigl[ \epsilon_{t+1}^{q,n}(s,a) \, \bigl| \bigr. \, \mathcal G_{t+1}^{n-1} \bigr ] = 0$ almost surely. Furthermore, we can easily argue that for some positive constant $C_H$,
\begin{equation}
\mathbf{E} \Bigl[\bar{Q}_t^{n-1}(s,a) - H_t \bigl(\bar{u}^{1,n-1}_t,\ldots,\bar{u}^{m,n-1}_t, \bar{Q}_{t+1}^{n-1},W_{t+1}^{q,n}\bigr)(s,a)\Bigr]^2 \le C_H,
\label{eq:2momentg}
\end{equation}using Assumption \ref{ass:problem}(i), Assumption \ref{ass:problem}(ii), the fact that $\rho_t^\alpha(0)=\Phi(0,\ldots,0)=0$, and the boundedness of the iterates $\bar{Q}_{t+1}^{n-1}$ and $\bar{u}_t^n$.

Now, fix an $(s,a)$ and let $S_{t+1}^{n}= S^M\bigl(s,a,W_{t+1}^{q,n}\bigr)$. Expanding and using Assumption \ref{ass:problem}(i) and the property (\ref{eq:minmax}) we see that
\begin{align}
 |\xi_{t+1}^{q,n}(s,a)| &\le \begin{aligned}[t] L_\Phi \, \bigl| & \! \min_{a' \in \mathcal A_{S^n_{t+1}}} \!\!\! Q_{t+1}^*(S_{t+1}^{n},a') - \!\! \min_{a' \in \mathcal A^n_{S_{t+1}}}  \!\!\! \bar{Q}_{t+1}^{n-1}(S_{t+1}^{n},a') \bigr| \\
 &+ L_\Phi \, \sum_i \bigl|  \bar{u}^{i,n-1}_{t}(s,a) - u_t^{i,*}(s,a) \bigr| \end{aligned} \nonumber \\
 &\le  L_\Phi  \, \Bigl[ \bigl\| \bar{Q}^{n-1}_{t+1} -Q_{t+1}^* \bigr \|_\infty +   \sum_i \bigl\|  \bar{u}^{i,n-1}_{t} - u_t^{i,*}\bigr\|_\infty \Bigr].
\label{eq:xibound2}
\end{align}
Taking conditional expectation, applying Lemma \ref{lem:uconv}, and using the induction hypothesis (which tells us that $\bar{u}^{i,n}_t \rightarrow u^{i,*}_t$ and $\bar{Q}^n_{t+1} \rightarrow Q^*_{t+1}$ almost surely), we conclude
\[
\mathbf{E} \bigl[ \xi_{t+1}^{q,n}(s,a) \, \bigl| \bigr. \, \mathcal G_{t+1}^{n-1} \bigr ] \rightarrow 0 \quad a.s.,
\] satisfying one of the conditions of \cite[Theorem 2.4]{Kushner2003}. In addition, note the bounded second moment condition of (\ref{eq:2momentg}), the conditional unbiasedness of $\epsilon_{t+1}^{q,n}(s,a)$, the stepsize, sampling, and truncation properties of Assumption \ref{ass:algorithm}. We once again have the ingredients to apply the stochastic approximation convergence theorem \cite[Theorem 2.4]{Kushner2003} (the objective function for applying the theorem is $q \mapsto \|q-Q_t^*\|_2^2$) to the update equation of (\ref{eq:Q_update}) in order to conclude 
\begin{equation*}
\bar{Q}_t^n(s,a) \rightarrow Q_t^*(s,a) \quad a.s.,
\label{eq:qconv}
\end{equation*}
for every $(s,a) \in \mathcal U$, completing both the inductive step and the proof.
\end{proof}

\propepsgreedy*
\begin{proof}
Let us consider, for a fixed $t$ and $(s,a)$, the stepsize sequence
\[\gamma_t^n(s,a) = \frac{\gamma_t}n \, \mathbf{1}_{\{ (s,a) = (S_t^n, a_t^n) \}}.
\]
The argument for the case of $\eta_t^n(s,a)$ is, of course, exactly the same and is omitted. First, notice that the second part, i.e., $\sum_{n=1}^\infty \gamma_t^n(s,a)^2 < \infty$ almost surely,
is trivial. We thus focus on proving that the second part of the assumption, that $\sum_{n=1}^\infty \gamma_t^n(s,a) = \infty$ almost surely, holds in this case. An iteration $n$ is called a \emph{visit} to state $(s,a)$ if $(s,a)=(S_t^n, a_t^n)$ (so the stepsize is nonzero). Let $\{ \Delta_t^k(s,a) \}_{k \ge 1}$ be a process that describes the \emph{interarrival times} for visits to the state $(s,a)$; in other words, $\Delta_t^k(s,a)$ is the number of iterations that pass between the $(k-1)$-st and $k$-th visits to the state $(s,a)$. The \emph{arrival time} sequence (the iterations for which $(s,a)$ is visited) is thus given by the sum
\[
N_t^k(s,a) = \sum_{k' = 1}^k \Delta_t^{k'}(s,a) \; \mbox{for $k > 0$.}
\]
To simplify notation, we henceforth drop the dependence of these processes on $(s,a)$ and use $\Delta_t^k= \Delta_t^k(s,a)$ and $N_t^k = N_t^k(s,a)$. By selecting only the iterations for which the stepsizes are nonzero, we see that
\begin{align*}
\sum_{k=1}^\infty \gamma_t^{N_t^k}= \sum_{n=1}^\infty \gamma_t^n(s,a).
\end{align*}
Analogous to $N_t^k$ process, define the deterministic sequence
\[
n_t^k = \left \lceil \sum_{k'=1}^k \delta_t^{k'} \right \rceil \quad \mbox{and} \quad \delta_t^k = \frac{-2}{\log \left(1-\varepsilon \right)} \, \log k.
\]
Observe that under the $\varepsilon$-greedy sampling policy, the event $\bigl \{(s,a) \ne (S_t^n, a_t^n) \bigr \}$ occurs with probability at most $1-\varepsilon$. Hence, by independence, we can show that
\[
\mathbf{P} \bigl( \Delta_t^k \ge \delta_t^k \bigr) \le \left(1 - \varepsilon \right)^{\lceil \delta_t^k \rceil} \le \frac{1}{k^2}.
\]
By the Borel-Cantelli Lemma, with probability 1, the events $\bigl \{  \Delta_t^k \ge \delta_t^k \bigr \}$ occur finitely often. Hence, there exists an almost surely finite $K_t$ such that for all $k \ge K_t$, it is true that $\Delta_t^k \le \delta_t^k$. Clearly, $n_t^k \le \mathcal O( k\,\log k)$ and 
\begin{equation}
\sum_{k=1}^\infty \gamma_t^{n_t^k} = \sum_{k=1}^\infty \frac{\gamma_t}{n_t^k} = \infty
\label{eq:propcondition}
\end{equation}
holds. Let $C_1 = \sum_{i=1}^{K_t-1}\frac{\gamma_t}{N_t^{i}}$ and $C_2 = \sum_{i=1}^{K_t-1}\Delta_t^i$ be almost surely finite random variables. Some simple manipulations yield (the following chain of inequalities may be analyzed $\omega$-wise to obtain the $a.s.$ qualification)
\begin{align*}
\sum_{k=1}^\infty \gamma_t^{N_t^k}= \sum_{k=1}^\infty \frac{\gamma_t}{N_t^k} &= C_1 + \sum_{k=K_t}^\infty \frac{\gamma_t}{C_2 + \sum_{i=K_t}^k \Delta_t^i}\\
& \ge C_1 + \sum_{k=K_t}^\infty \frac{\gamma_t}{C_2 + \sum_{i=K_t}^k \delta_t^i}\\
& \ge C_1 + \sum_{k=K_t}^\infty \frac{\gamma_t}{C_2-\sum_{i=0}^{K_t-1} \delta_t^i + n_t^k} = \infty \quad a.s.,
\end{align*}
where the final equality follows from (\ref{eq:propcondition}).
\end{proof}

\boundtwo*
\begin{proof}
Note that although we omit its algebraic form, the existence of $C_{g}$ is guaranteed by Assumption \ref{ass:problem}(ii) and the boundedness of $\bar{Q}_t^{n-1}(s,a)$.
Let us recall the update equation given in Step 7 of Algorithm \ref{alg:dq} can be rewritten as
\[
\bar{Q}_t^n = \Pi_{\mathcal X_t^q} \Bigl \{\bar{Q}_t^{n-1} - \textnormal{diag}(\eta_t^n) \, \Bigl[\bar{Q}_t^{n-1}-Q_t^*+ \epsilon_{t+1}^{q,n} + \xi_{t+1}^{q,n} \Bigr] \Bigr \}.
\]
Expanding, we have
\begin{equation}
\begin{aligned}
\bigl \|\bar{Q}_t^{n}-Q_t^* \bigr\|_2^2 \le \bigl\| \bar{Q}_t^{n-1}- Q_t^*&  \bigr\|_2^2 + \bigl\| \diag(\eta_t^{n}) \bigl[ \bar{Q}_t^{n-1}-Q_t^*+ \epsilon_{t+1}^{q,n} + \xi_{t+1}^{q,n}  \bigr] \bigr\|_2^2\\
&- 2 \, \bigl(\bar{Q}_t^{n-1} - Q_t^* \bigr)^\mathsf{T} \diag(\eta_t^{n})\, \bigl[ \bar{Q}_t^{n-1}-Q_t^*+ \epsilon_{t+1}^{q,n} + \xi_{t+1}^{q,n}   \bigr].\\
\end{aligned}
\label{eq:2le0}
\end{equation}
We focus on the cross term. First, by the $\varepsilon$-greedy sampling policy, notice that
\begin{align}
\mathbf{E} \Bigl[ \bigl(\bar{Q}_t^{n-1} - Q_t^* \bigr)^\mathsf{T} \diag(\eta_t^{n})\, \bigl( \bar{Q}_t^{n-1} - Q_t^* \bigl) \, \bigl|  \, \mathcal G_{t+1}^{n-1} \Bigr] \ge \frac{\varepsilon \eta_t}{n} \, \bigl \| \bar{Q}_t^{n-1} - Q_t^*  \bigr \|_2^2, \label{eq:2le1}
\end{align}
and by the definition of $\epsilon_{t+1}^{q,n}$,
\begin{align}
\mathbf{E} \Bigl[ \bigl(\bar{Q}_t^{n-1} - Q_t^* \bigr)^\mathsf{T} \diag(\eta_t^{n})\, \epsilon_{t+1}^{q,n} \, \bigl|  \, \mathcal G_{t+1}^{n-1} \Bigr] =0.\label{eq:2le2}
\end{align}
Using the bound (\ref{eq:xibound2}), the fact that $\eta_t^n$ contains exactly one nonzero component, and the monotonicity of the $l_p$ norms, we can see that
\begin{align*}
\mathbf{E} \Bigl[ -\, (\bar{Q}_t^{n-1} - Q_t^*)^\mathsf{T} \diag(\eta_t^{n})\, \xi_{t+1}^{q,n} \, \bigl|  \, \mathcal G_{t+1}^{n-1} \Bigr]
&\le \begin{aligned}[t]\frac{\eta_t}{n} \, L_\Phi \, & \Bigl[  \bigl \| \bar{Q}_t^{n-1} - Q_t^*  \bigr\|_\infty \, \bigl\| \bar{Q}_{t+1}^{n-1} - Q_{t+1}^* \bigr\|_\infty \\
&+  \bigl \| \bar{Q}_t^{n-1} - Q_t^*  \bigr\|_\infty \, \sum_i \,\bigl\|  \bar{u}^{i,n-1}_{t} - u_t^{i,*}\bigr\|_\infty\Bigr]
\end{aligned}\\
&\le \begin{aligned}[t]\frac{\eta_t}{n} L_\Phi \, & \Bigl[   \bigl \| \bar{Q}_t^{n-1} - Q_t^*  \bigr\|_2 \, \bigl\| \bar{Q}_{t+1}^{n-1} - Q_{t+1}^* \bigr\|_2 \\
&+ \bigl \| \bar{Q}_t^{n-1} - Q_t^*  \bigr\|_2 \, \sum_i \,\bigl\|  \bar{u}^{i,n-1}_{t} - u_t^{i,*}\bigr\|_2\Bigr].
\end{aligned}
\end{align*}
Again applying $2ab \le a^2 \kappa +b^2/\kappa$, we see that for any constants $\kappa_k > 0$, $k = 0,1,\ldots,m$,
\begin{equation}
\begin{aligned}
\mathbf{E} \Bigl[ -2\, (\bar{Q}_t^{n-1} - Q_t^*)^\mathsf{T} \diag(\eta_t^{n})\, \xi_{t+1}^{\phi,n} \, \bigl|  \, \mathcal G_{t+1}^{n-1} &\Bigr] 
\le \frac{\eta_t}{n} \, L_\Phi \,\biggl[ \Bigl(\kappa_0 + \sum_{i=1}^m \kappa_i\Bigr) \, \bigr \| \bar{Q}_t^{n-1} - Q_t^* \bigr \|_2^2\\
&+\frac{1}{\kappa_0} \, \bigl\| \bar{Q}_{t+1}^{n-1} - Q_{t+1}^* \bigr\|_2^2 + \sum_{i=1}^m \frac{1}{\kappa_i} \, \bigl\|\bar{u}_t^{i,n-1} - u_t^{i,*} \bigr\|_2^2 \biggr].
\end{aligned}
\label{eq:2le3}
\end{equation}
The statement of the lemma follows by taking expectations of (\ref{eq:2le0}), (\ref{eq:2le1}), (\ref{eq:2le2}), and (\ref{eq:2le3}), using the $C_H$ bound, and combining.
\end{proof}

\thmconvrate*
\begin{proof}
First, we state a special case of Chung's Lemma (see \cite{Chung1954}), which is useful for analyzing a specific type of sequence that arises often in recursive optimization algorithms (see, e.g., \cite{Juditsky2009}, \cite{Rakhlin2012}).
\begin{lemma}[\cite{Chung1954}]
Consider a sequence $\{a^n\}$. Suppose that the following recursive inequality holds for some $b >1$ and every $n \ge 1$:
\[a^n \le \left(1- \frac{b}{n}\right) \, a^{n-1} + \frac{c}{n^2}.\]
Then, if $k \ge \max \left\{ \frac{c}{b-1}, \,a^0 \right\}$, it follows that $a^n \le \frac{k}{n}$ for $n \ge 1$.
\label{lem:seqrate}
\end{lemma}
To prove the theorem, we use an induction argument for $\bar{Q}^n$. Let us first consider $t=T$ as the base case. Since $\mathbf{E} \bigl[\bigl \|\bar{Q}_{T}^{n} - Q_{T}^*  \bigr\|_2^2\bigr] = 0$,
it is trivially $\mathcal O(1/n)$, so let us take our induction hypothesis to be $\mathbf{E} \bigl[\bigl \|\bar{Q}_{t+1}^{n} - Q_{t+1}^*  \bigr\|_2^2\bigr] \le \mathcal O(1/n)$ for a particular $t < T$. Hence, there exists $k_{t+1}^q > 0$ independent of $n$ such that $\mathbf{E} \bigl[\bigl \|\bar{Q}_{t+1}^{n-1} - Q_{t+1}^*  \bigr\|_2^2\bigr] \le \frac{k_{t+1}^q}{n}$ holds for $n \ge 1$ (note the $n-1$ on the left hand side).
We first apply Lemma \ref{lem:bound1} with $\gamma_t = \frac{2}{\varepsilon \, C_{l_f}}$ and $\kappa = \frac{\varepsilon \, C_{l_f}}{C_{L_F}}$, so by Lemma \ref{lem:seqrate}, we can take 
\[
k_t^i = 2\,\max \left\{\frac{2 \, C_{L_F}}{\varepsilon^2 \, C_{l_f}^2} \, k^q_{t+1} + \frac{4\,C_{\alpha_i}}{\varepsilon^2 \, C_{l_f}^2}, \, \bigl \|\bar{u}_t^{i,0} - u_t^{i,*} \bigr \|_2^2 \right\}
\]
to satisfy $\mathbf{E} \bigl[ \bigl \|\bar{u}_t^{n} - u_t^*  \bigr\|_2^2  \bigr] \le \frac{k_t^i}{2n}$ for $n \ge 1$. It follows that $\mathbf{E} \bigl[ \bigl \|\bar{u}_t^{n-1} - u_t^*  \bigr\|_2^2  \bigr] \le \frac{k_t^i}{n}$ for $n \ge 2$. Finally, we apply Lemma \ref{lem:bound2} with $\eta_t = \frac{2}{\varepsilon}$, $\kappa_0 = \frac{\varepsilon}{2 \, L_\Phi}$, and $\kappa_i = \frac{\varepsilon}{2 \, m \,L_\Phi}$ for all $i$. Again by Lemma \ref{lem:seqrate}, we can choose
\[
k_t^q = \max \left \{\frac{4\,L_\Phi^2}{\varepsilon^2} \, k^q_{t+1} + \frac{4\, m\,L_\Phi^2}{\varepsilon^2} \, \sum_i k_t^i + \frac{4 \, C_{H} }{\varepsilon^2}, \,\bigl \|\bar{Q}_t^{0} - Q_t^{*}  \bigr\|_2^2 \right\},
\]
which gives $\mathbf{E} \Bigl[ \bigl \|\bar{Q}_t^{n} - Q_t^*  \bigr\|_2^2  \Bigr] \le \frac{k_t^q}{n}$ for $n \ge 2$, the $\mathcal O(1/n)$ rate needed to complete the inductive step. Therefore, we can conclude that
\[
\mathbf{E} \bigl[ \bigl\|\bar{Q}^{n} - Q^* \bigr \|_2^2 \bigr] = \sum_{t=0}^T \, \mathbf{E} \bigl[ \bigl\|\bar{Q}_t^{n} - Q_t^* \bigr \|_2^2 \bigr] \le  \mathcal O\left(1/n\right),
\]
as desired. As for the convergence of $\mathbf{E} \bigl[ \bigl\|\bar{u}^{i,n} - u^{i,*} \bigr \|_2^2 \bigr]$, we note that the values $k_t^i$ obtained through the induction above can be used to deduce the $\mathcal O\left(1/n\right)$ convergence rate.
\end{proof}

Similar to the sequences $\xi_{t+1}^{i,n}$ and $\xi_{t+1}^{q,n}$ used in the proofs of Lemma \ref{lem:uconv} and Theorem \ref{thm:asconv}, let us introduce another useful process $\xi_{t+1}^{h,n}(s,a) \in \mathbb R^K$ for each $t$ and $(s,a)$. As in Step 7 of Algorithm \ref{alg:dq_rds}, let $w = W_{t+1}^{q,n}$. We define
\[
\xi_{t+1}^{h,n}(s,a) = \Bigr[ \bigl|H_t^{*}(w\<|\<s,a)\bigr|  - \bigl|H_t^{n}(w\<|\<s,a)\bigr| \Bigr]\phi(w) \, p_t(w) \, p^u(w) \, \bar{p}_t^{n-1}(w\<|\<s,a)^{-1},
\]
representing the error from using $H_t^{n}(w\<|\<s,a)$ to approximate $H_t^{*}(w\<|\<s,a)$.

\begin{restatable}{lemma}{lembetaxisum}
Under the $\varepsilon$-greedy sampling policy of (\ref{eq:epsilongreedy}) and Assumptions \ref{ass:algorithm}(vi)--\ref{ass:phisupport}, it follows that
\[
\sum_{n=1}^\infty \, \tilde{\beta}_t^{n-1} \,  \mathbf{E} \Bigl[ \bigl\| \xi_{t+1}^{h,n}(s,a)\bigr \|_2 \, \bigl | \, \mathcal G_{t+1}^{n-1} \Bigr] < \infty \quad a.s.
\]
for each $t$ and $(s,a)$.
\label{lem:betaxi_sum}
\end{restatable}

\begin{proof}

For simplicity, let $w = W_{t+1}^{q,n}$ and $P = \phi(w) \,p_t(w) \, p^u(w) \, \bar{p}_t^{n-1}(w\<|\<s,a)^{-1}$. We notice that
\begin{align*}
\bigl\|\xi_{t+1}^{h,n}(s,a)\bigr\|_2 &\le \|P\|_2 \,\bigl|H_t^{*}(w\<|\<s,a) - H_t^{n}(w\<|\<s,a)\bigr|\\
&\le \|P\|_2\, L_\Phi \, \bigl\| \bar{Q}^{n-1}_{t+1} -Q_{t+1}^* \bigr \|_2 + \|P\|_2 \, L_\Phi \,  \sum_i \bigl\|  \bar{u}^{i,n-1}_{t} - u_t^{i,*}\bigr\|_2,
\end{align*}
where the second inequality follows from the same steps used in deriving (\ref{eq:xibound2}) and then applying the monotonicity of $l_p$ norms. 
Squaring, applying the inequality $2ab \le a^2 \kappa + b^2/\kappa$ for any $\kappa > 0$ to the cross terms, and using the fact that each component of $P$ is bounded (due to $p^u$), we can write
\begin{equation}
\bigl\|\xi_{t+1}^{h,n}(s,a)\bigr\|_2^2 \le  C_Q \, \bigl\| \bar{Q}^{n-1}_{t+1} -Q_{t+1}^* \bigr \|^2_2 + \sum_i C_{u,i}\, \bigl\|  \bar{u}^{i,n-1}_{t} - u_t^{i,*}\bigr\|^2_2,
\label{eq:xi2norm}
\end{equation}
for some constants $C_Q, C_{u,1}, \ldots, C_{u,m} \ge 0$. Taking expectation of (\ref{eq:xi2norm}) and using the convergence rate result of Theorem \ref{thm:convrate2}, we see that $\mathbf{E} \bigl[\|\xi_{t+1}^{h,n}(s,a)\|_2^2 \bigr]\le \mathcal O(1/n)$. Now, by Cauchy-Schwarz and Assumption \ref{ass:stepbeta}(iii),
\begin{align*}
\mathbf{E} \Bigl[ \tilde{\beta}_t^{n-1} \,  \mathbf{E} \bigl[ \| \xi_{t+1}^{h,n}(s,a) \|_2 \, | \, \mathcal G_{t+1}^{n-1} \bigr] \Bigr] &= \mathbf{E} \Bigl[ \tilde{\beta}_t^{n-1} \,  \bigl\|\xi_{t+1}^{h,n}(s,a)\bigr\|_2 \Bigr] \\
&\le \sqrt{\mathbf{E} \bigl[(\tilde{\beta}_t^{n-1})^2\bigr] \, \, \mathbf{E} \bigl[\|\xi_{t+1}^{h,n}(s,a)\|_2^2\bigl]} \le \mathcal O (n^{-1-\epsilon/2}),
\end{align*}
so it is clear that the terms $\mathbf{E} \Bigl[ \tilde{\beta}_t^{n-1} \,  \mathbf{E} \bigl[ \| \xi_{t+1}^{h,n}(s,a) \|_2 \, | \, \mathcal G_{t+1}^{n-1} \bigr] \Bigr]$ are summable. By the monotone convergence theorem, we conclude
\[
\mathbf{E} \left[ \sum_{n=1}^\infty \, \tilde{\beta}_t^{n-1} \,  \mathbf{E} \Bigl[ \bigl\| \xi_{t+1}^{h,n}(s,a) \bigr \|_2 \, | \, \mathcal G_{t+1}^{n-1} \Bigr]\right] < \infty,
\]
from which the statement of the lemma follows (notice that the term within the expectation must be finite almost surely).
\end{proof}

\thmrdsconv*
\begin{proof}
First, notice that the convergence of $\bar{Q}^n$ and $\bar{u}^{i,n}$ is (mostly) unaffected by the addition of the sampling procedure. By the principle of importance sampling, the factor of $\diag(L_t^n)$ corrects, in expectation, for the fact that $W_{t+1}^{q,n}$ is sampled from the importance distribution $\bar{p}_t^{n-1}(w\<|\<s,a)$ rather than $p_t(w)$. In addition, Assumption \ref{ass:phisupport} implies that the condition stating that \emph{the second moment of the gradient term is finite}, which is needed for \cite[Theorem 2.4]{Kushner2003}, still holds and therefore the convergence follows.

We now focus on the last part of the theorem and analyze the convergence of the sampling coefficients. Let 
\[
\theta^*_t(s,a) = \Pi_\phi \Bigl[\,\bigl|H_t^{*}(\< \cdot \<|\<s,a) \bigr| \; p_t(\< \cdot \<)\, \Bigr],
\]
and we aim to show $\bar{\theta}_t^n(s,a) \rightarrow \theta_t^*(s,a)$ almost surely. The proof technique is standard and uses a supermartingale convergence argument; see, e.g, \cite[Theorem 5.3]{Pflug1996}, but several aspects need to be adapted for technical reasons in our setting. Throughout this proof, fix a $t$ and $(s,a)$. To simplify notation, we use the shorthand $\bar{p}_t^{n-1}(w) = \bar{p}_t^{n-1}(w|\<s,a)$, $w=W_{t+1}^{q,n} \sim \bar{p}_t^{n-1}(w)$, $H_t^n(w)=H_t^n(w\<|\<s,a)$, and $H_t^*(w) = H_t^*(w\<|\<s,a)$. First, we decompose the gradient term. Define a function $h_t : \mathbb R^K \rightarrow \mathbb R^K$, so that for $\theta \in \mathbb R^K$,
\[
h_t(\theta) = \mathbf{E} \Bigl[\bigl[\theta^\mathsf{T} \phi(w) - |H_t^{*}(w)| \, p_t(w)  \bigr] \, \phi(w) \, p^u(w) \, \bar{p}_t^{n-1}(w)^{-1} \Bigr].
\]
Note that $h_t$ is affine in $\theta$ and is the gradient of the (strictly convex) quadratic objective function in (\ref{eq:projection}), meaning that we can find positive constants $C_{h,1}$ and $C_{h,2}$ such that
\begin{equation}
\bigl\|h_t(\theta)\bigr\|^2_2 \le C_{h,1}\, \bigl\| \theta - \theta_t^{*}(s,a) \bigr \|^2_2 + C_{h,2}.
\label{eq:haffine}
\end{equation}
Next, the noise term $\epsilon_{t+1}^{h,n}(s,a)$ is given by
\[
\epsilon_{t+1}^{h,n}(s,a) =  \bigl[ \bigl(\bar{\theta}^{n-1}_t(s,a)\bigr)^\mathsf{T} \, \phi(w) - \bigl|H_t^{*}(w)\bigr| \, p_t(w)  \bigr] \, \phi(w) \, p^u(w) \, \bar{p}_t^{n-1}(w)^{-1} - h_t\bigl(\bar{\theta}^{n-1}_t(s,a)\bigr),
\]
and it follows that $\mathbf{E} \bigl[ \epsilon_{t+1}^{h,n}(s,a) \, | \, \mathcal G_{t+1}^{n-1} \bigr]=0$.
Using the fact that $h_t\bigl(\bar{\theta}^{n-1}_t(s,a)\bigr)$ is a deterministic affine function of $\bar{\theta}^{n-1}_t(s,a)$, we observe that $\epsilon_{t+1}^{h,n}(s,a)$ can be written in the form $X\, \bigl[\bar{\theta}^{n-1}_t(s,a) - \theta_t^*(s,a) \bigr]+Z$ where $X \in \mathbb R^{K \times K}$ and $Z \in \mathbb R^{K}$ are random. Since $p^{u}(w) = 0$ for $w$ outside of a compact set, the entries of both $X$ and $Z$ are bounded. Let $\|\cdot \|_2$ denote the \emph{operator norm} whenever its argument is a matrix and we have
\begin{equation}
\begin{aligned}
\bigl\|\epsilon_{t+1}^{h,n}(s,a)\bigr\|^2_2 &\le \bigl[ \|X\|_2 \bigl\| \bar{\theta}^{n-1}_t(s,a) - \theta_t^*(s,a) \bigr\|_2 + \|Z\|_2\bigr]^2 \\
&\le C_{\epsilon,1} \bigl\| \bar{\theta}^{n-1}_t(s,a) - \theta_t^*(s,a) \bigr\|_2^2+ C_{\epsilon,2},
\end{aligned}
\label{eq:eps2norm}
\end{equation}
for some constants $C_{\epsilon,1}$ and $C_{\epsilon,2}$. 
We also reproduce the definition of $\xi_{t+1}^{h,n}(s,a)$ given previously:
\[
\xi_{t+1}^{h,n}(s,a) = \Bigr[ \bigl|H_t^{*}(w)\bigr|  - \bigl|H_t^{n}(w)\bigr| \Bigr] \phi(w)\, p_t(w) \, p^u(w) \, \bar{p}_t^{n-1}(w)^{-1}.
\]
Step 7 of Algorithm \ref{alg:dq_rds} can therefore be written
\[
\bar{\theta}^{n}_t(s,a) = \Bigl[ \bar{\theta}^{n-1}_t(s,a) - \beta_t^{n}(s,a) \, \Bigl[h_t\bigl(\bar{\theta}_t^{n-1}(s,a) \bigr)  + \epsilon_{t+1}^{h,n}(s,a)  +   \xi_{t+1}^{h,n}(s,a)  \Bigr]  \Bigr]^+.
\]
For convenience, let us define
\[
A_t^n = \bar{\theta}^{n}_t(s,a) - \theta_t^{*}(s,a) \quad \mbox{and} \quad \hat{h}_t^{n} = h_t\bigl(\bar{\theta}_t^{n-1}(s,a) \bigr)  + \epsilon_{t+1}^{h,n}(s,a)  +   \xi_{t+1}^{h,n}(s,a). 
\]
Hence,
\begin{equation}
\begin{aligned}
\|A_t^n \|_2^2 &\le \bigl\|\bigl[ \bar{\theta}^{n-1}_t(s,a) - \beta_t^{n}(s,a) \, \hat{h}_t^{n}  \bigr]^+ - \bigl[\theta_t^{*}(s,a)\bigr]^+ \bigr\|_2^2\\
&\le \|A_t^{n-1}\|_2^2 - 2 \,\beta_t^n(s,a) \, (A^{n-1}_t)^\mathsf{T} \,\hat{h}_t^{n} + \bigl(\tilde{\beta}_t^{n-1}\bigr)^2 \, \|\hat{h}_t^{n} \|^2_2.
\end{aligned}
\label{eq:Aexpand}
\end{equation}
Taking conditional expectation of the cross term, using Assumption \ref{ass:stepbeta}, noting the inequality $(A^{n-1}_t)^\mathsf{T} \,h_t(\bar{\theta}_t^{n-1}(s,a) ) > 0$ (by strict convexity), and using Cauchy-Schwarz, we get
\begin{align}
\mathbf{E} \bigl[ - 2 \,\beta_t^n(s,a) \, (A^{n-1}_t)^\mathsf{T} \,\hat{h}_t^{n} \, | \, \mathcal G_{t+1}^{n-1} \bigr] &\le  \begin{aligned}[t]- 2\,\varepsilon\, &\tilde{\beta}_t^{n-1} (A^{n-1}_t)^\mathsf{T} \,h_t\bigl(\bar{\theta}_t^{n-1}(s,a) \bigr) \\
&+ 2\,\tilde{\beta}_t^{n-1} \| A_t^{n-1} \|_2 \, \mathbf{E} \bigl[ \| \xi_{t+1}^{h,n}(s,a) \|_2 \, \bigl | \, \mathcal G_{t+1}^{n-1} \bigr]
\end{aligned}
\label{eq:Aexpand2}
\end{align}
Moving onto the third term of (\ref{eq:Aexpand}), we have
\begin{align}
\mathbf{E} \bigl[\|\hat{h}_t^{n} \|^2_2\, | \, \mathcal G_{t+1}^{n-1} \bigr] &\le \begin{aligned}[t]&\Bigl[ \| h_t(\bar{\theta}_t^{n-1}(s,a)) \|_2^2  + \mathbf{E} \bigl[  \|\epsilon_{t+1}^{h,n}(s,a)\|^2_2  + \|\xi_{t+1}^{h,n}(s,a)\|^2_2  \\
 &  + 2\, \|\epsilon_{t+1}^{h,n}(s,a)\|_2\, \|\xi_{t+1}^{h,n}(s,a)\|_2 + 2\,\|h_t(\bar{\theta}_t^{n-1}(s,a)) \|_2\, \|\epsilon_{t+1}^{h,n}(s,a)\|_2  \\
 & + 2\,\|h_t(\bar{\theta}_t^{n-1}(s,a)) \|_2\, \|\xi_{t+1}^{h,n}(s,a)\bigr\|_2 \,  | \, \mathcal G_{t+1}^{n-1} \bigr] \Bigr]
\end{aligned}\nonumber\\
&\le \begin{aligned}[t]& 3\, \Bigl[\mathbf{E} \bigl[  \|\epsilon_{t+1}^{h,n}(s,a)\|^2_2  +  \|\xi_{t+1}^{h,n}(s,a)\|^2_2 \,  | \, \mathcal G_{t+1}^{n-1} \bigr]   +    \|  \,h_t(\bar{\theta}_t^{n-1}(s,a)) \|_2^2 \Bigr] \\
\end{aligned}\nonumber\\
&\le \begin{aligned}[t]&\Bigl[ C_{\hat{h},1} \,\| A_t^{n-1} \|_2^2+ C_{\hat{h},2}+ 3\, \mathbf{E} \bigl[\|\xi_{t+1}^{h,n}(s,a)\|^2_2 \,  | \, \mathcal G_{t+1}^{n-1} \bigr]  \Bigr]
\end{aligned}
\label{eq:Aexpand3}
\end{align}
where the second inequality is due to the inequality $2ab \le a^2+b^2$ and the third inequality, with new constants $C_{\hat{h},1}$ and $C_{\hat{h},2}$, is due to the bounds (\ref{eq:haffine}) and (\ref{eq:eps2norm}). Combining (\ref{eq:Aexpand}), (\ref{eq:Aexpand2}), and (\ref{eq:Aexpand3}),
\begin{equation*}
\mathbf{E} \bigl[ \| A_t^n \|_2^2 \,|\,  \mathcal G_{t+1}^{n-1}  \bigr ] \le \|A_t^{n-1}\|_2^2 \, ( 1+\zeta^{n-1}_t) + \mu^{n-1}_t - \nu^{n-1}_t,
\end{equation*}
where
\begin{align*}
\zeta^{n-1}_t &= C_{\hat{h},1}\,\bigl(\tilde{\beta}_t^{n-1}\bigr)^2 + 2\,\tilde{\beta}_t^{n-1} \,\mathbf{E} \bigl[ \| \xi_{t+1}^{h,n}(s,a) \|_2 \,  | \, \mathcal G_{t+1}^{n-1} \bigr],\\
\mu_t^{n-1} &= 3\,\bigl(\tilde{\beta}_t^{n-1}\bigr)^2\, \mathbf{E} \bigl[\|\xi_{t+1}^{h,n}(s,a)\|^2_2 \,  | \, \mathcal G_{t+1}^{n-1} \bigr] +  C_{\hat{h},2} \,\bigl(\tilde{\beta}_t^{n-1}\bigr)^2,\\
\nu_t^{n-1} &= 2\,\varepsilon\, \tilde{\beta}_t^{n-1} (A^{n-1}_t)^\mathsf{T} \,h_t\bigl(\bar{\theta}_t^{n-1}(s,a) \bigr).
\end{align*}
Note that $\sum_{n=1}^\infty \bigl(\tilde{\beta}_t^{n-1}\bigr)^2 < \infty$ almost surely by Assumption \ref{ass:stepbeta}(iii) and the monotone convergence theorem (we can apply the same logic as Lemma \ref{lem:betaxi_sum}). This, together with Lemma \ref{lem:betaxi_sum}, allows us to conclude that $\sum_{n=1}^\infty \zeta_t^{n-1} < \infty$ almost surely. Moreover, since $\tilde{\beta}_t^{n-1}$ must converge to zero and (\ref{eq:xi2norm}) implies $\mathbf{E} \bigl[\|\xi_{t+1}^{h,n}(s,a)\|^2_2 \,  | \, \mathcal G_{t+1}^{n-1} \bigr]$ converges to zero, we can see that $\sum_{n=1}^\infty \mu_t^{n-1} < \infty$ almost surely. The well-known supermartingale convergence lemma of \cite{Robbins1971} tells us that $\|A_t^n\|_2$ converges and $\sum_{n=1}^\infty \nu_t^{n-1} < \infty$ almost surely. Let $D_\delta = \bigl\{\lim_{n\to\infty}\|A_t^n\|_2 >\delta\bigr\}$. On the event $D_\delta$, using strict convexity, we know that $(A^{n-1}_t)^\mathsf{T} \,h_t\bigl(\bar{\theta}_t^{n-1}(s,a) \bigr)$ is positive and bounded away from zero for any $n$. Along with $\sum_{n=1}^\infty \tilde{\beta}_t^{n-1} = \infty$, this implies $\sum_{n=1}^\infty \nu_t^{n-1} = \infty$, which in turn shows us that $D_\delta$ must occur with probability zero for any $\delta > 0$.
\end{proof}

\clearpage
\bibliographystyle{abbrvnat}
\bibliography{/Users/drjiang/Documents/Dropbox/Pittsburgh/Bibtex/Risk,/Users/drjiang/Documents/Dropbox/Pittsburgh/Bibtex/Bib}

\end{document}